\documentclass[12pt]{article}

\usepackage{amsfonts,amsthm}
\usepackage{amsmath}
\usepackage{amssymb}
\usepackage{latexsym,color}
\usepackage{amsbsy}
\usepackage{graphicx}
\usepackage{youngtab}
\usepackage{hyperref}
\usepackage{enumerate}
\usepackage{tikz}
\usepackage{mathtools}
\usepackage{comment}
\DeclareMathAlphabet{\mathpzc}{OT1}{pzc}{m}{it}

\newtheorem{thm}{Theorem}[section]

\newtheorem{lemma}[thm]{Lemma}
\newtheorem{cor}[thm]{Corollary}

\newtheorem{example}[thm]{Example}
\newtheorem{definition}[thm]{Definition}

\newtheorem{remark}[thm]{Remark}

\allowdisplaybreaks

\newcommand{\ncom}{\newcommand}
\ncom{\ka}{\kappa}
\ncom{\Akrpn}{\mathbb{C}A_{k}(r,p,n)}
\ncom{\Akhrpn}{\mathbb{C}A_{k+\frac{1}{2}}(r,p,n)}
\ncom{\C}{\mathbb{C}}
\ncom{\Z}{\mathbb{Z}}
\ncom{\Grpn}{G(r,p,n)}
\ncom{\Drpn}{\mbox{D}(r,p,n)}
\ncom{\Lrpn}{L(r,p,n)}
\ncom{\tabx}{\mbox{tab}_G} 
\ncom{\Vk}{V^{\otimes k}}
\ncom{\Vkk}{V^{\otimes (k+1)}}
\ncom{\Vkh}{V^{\otimes (k+\frac{1}{2})}}
\ncom{\kh}{k+\frac{1}{2}}
\ncom{\tcr}{\textcolor{red}}
\ncom{\vi}{v_{i_1} \otimes v_{i_2} \otimes \cdots \otimes v_{i_k}}
\ncom{\vip}{v_{i_{1'}} \otimes v_{i_{2'}} \otimes \cdots \otimes v_{i_{k'}}}
\ncom{\itp}{{i_{1'}, i_{2'}, \ldots, i_{k'}}}
\ncom{\is}{{i_1,i_2, \ldots, i_k}}
\ncom{\phik}{(\phi_k(x_d))^{i_1,i_2, \ldots, i_k}_{i_{1'}, i_{2'}, \ldots, i_{k'}}}
\ncom{\Y}{{\cal Y}}
\ncom{\N}{{\cal N}}
\ncom{\La}{\lambda}
\ncom{\tla}{\tilde{\La}}
\ncom{\tm}{\tilde{\mu}}
\ncom{\sh}{\text{sh}}
\ncom{\ST}{\text{ST}}
\ncom{\spec}{\mbox{spec}_G}
\ncom{\A}{{\cal A}}
\ncom{\cont}{\mbox{cont}}
\ncom{\contg}{\mbox{cont}_G}
\ncom{\Tab}{\mbox{Tab}(r, \La)}
\ncom{\T}{\mathpzc{T}}
\ncom{\Tn}{\mathpzc{T}_k(r,p,n)}
\ncom{\Thn}{\mathpzc{T}_{k+\frac{1}{2}}(r,p,n)}
\ncom{\Tmn}{\mathpzc{T}_{k-\frac{1}{2}}(r,p,n)}
\ncom{\vo}{\varOmega}
\ncom{\g}{\zeta}
\ncom{\pv}{\mathpzc{v}}
\ncom{\PP}{\mathpzc{P}}

\newcommand\Myperm[2][^n]{\prescript{#1\mkern-2.5mu}{}P_{#2}}

\title{\Large{\textcolor{black}{\bf On representation theory of partition algebras for complex reflection groups}}}
\author{Ashish Mishra and Shraddha Srivastava}
\newcommand{\Addresses}{{
  \bigskip
  \footnotesize

  \textsc{Instituto de Ci\^{e}ncias Exatas e Naturais, Universidade Federal do Par\'{a}, Bel\'{e}m, Brazil}\par\nopagebreak
  \textit{E-mail address}: \texttt{ashishmsr84@gmail.com}

  \medskip

  \textsc{Institute of Mathematical Sciences, Chennai, India}\par\nopagebreak
  \textit{E-mail address}: \texttt{shrashri@imsc.res.in}

}}

\begin{document}

\date{}

\maketitle

\begin{abstract}
This paper defines the partition algebra, denoted by $\Tn$, for complex reflection group $G(r,p,n)$ acting on $k$-fold tensor product $(\C^n)^{\otimes k}$, where $\C^n$ is the reflection representation of $\Grpn$. A basis of the centralizer algebra of this action of $\Grpn$ was given by Tanabe and for $p =1$, the corresponding partition algebra was studied by Orellana.
We also define a subalgebra $\Thn$ such that $\Tn \subseteq \Thn \subseteq \T_{k+1}(r,p,n)$ and establish this subalgebra as partition algebra of a subgroup of $G(r,p,n)$ acting on $(\C^n)^{\otimes k}$. We call the algebras $\Tn$ and $\Thn$ as Tanabe algebras. The aim of this paper is to study representation theory of Tanabe algebras: parametrization of their irreducible modules, and construction of Bratteli diagram for the tower of Tanabe algebras $$\T_0(r,p,n) \subseteq \T_{\frac{1}{2}}(r,p,n) \subseteq \T_1(r,p,n) \subseteq \T_{\frac{3}{2}}(r,p,n) \subseteq \cdots \subseteq \T_{\lfloor \frac{n}{2}\rfloor}(r,p,n).$$ We conclude the paper by giving Jucys-Murphy elements of Tanabe algebras and their actions on the Gelfand-Tsetlin basis, determined by this multiplicity free tower, of irreducible modules.
\end{abstract}



\section{Introduction}\label{int}

The symmetric group $S_{k}$ acts on the $k$-fold tensor product $\Vk$ of the $n$-dimensional vector space $V = \C^n$ over the field of complex numbers $\C$. The general linear group $GL_{n}(\mathbb{C})$ acts on $\Vk$ diagonally where $V$ is the defining representation of $GL_n(\C)$. These two actions commute; moreover, they generate the centralizers of each other. This is known as the classical Schur-Weyl duality \cite{Gre07}. 

Jones \cite{J94} and Martin \cite{M94}, independently, defined partition algebra $\C A_{k}(q)$, where $q\in \C$, as a generalization of Temperley-Lieb algebras and Potts model in higher dimensional statistical mechanics. The symmetric group $S_{n}$, being the subgroup of permutation matrices in $GL_n(\C)$, acts on $V^{\otimes k}$. Jones \cite{J94} proved Schur-Weyl duality between the partition algebra $\C A_k(n)$ and the symmetric group $S_n$ acting on $\Vk$.  Furthermore, Martin and Saleur studied the structure of partition algebras in \cite{MS93, MS94} and proved that the partition algebra $\C A_k(n)$ is semisimple unless $n$ is an integer such that $0 \leq n < 2k-1$.

The subalgebra $\C A_{k+\frac{1}{2}}(n)$ of partition algebra $\C A_{k+1}(n)$ was introduced by Martin \cite{MR98, M00}. Halverson and Ram \cite{HR05} showed Schur-Weyl duality between $\C A_{k+\frac{1}{2}}(n)$ and the subgroup $S_{n-1}$ of $S_n$; and thus established it to be equally important as partition algebra $\C A_k(n)$. It was also shown in \cite{HR05} that the branching rule is multiplicity free for $\C A_{l-\frac{1}{2}}(n) \subseteq \C A_l(n)$ for $ l \in \frac{1}{2}\Z_{>0}$ whenever both the algebras are semisimple.
Recursively building the Bratteli diagram for the tower of partition algebras $$\C A_0(n) \subseteq \C A_{\frac{1}{2}}(n) \subseteq \C A_1(n) \subseteq \C A_{\frac{3}{2}}(n)  \subseteq \cdots,$$ the Jucys-Murphy elements of partition algebras were also given in \cite[Theorem 3.37]{HR05}. Later, the seminormal representations of parition algebra were derived by Enyang \cite{E13}.

The complex reflection group $G(r,p,n)$, where $r,p$ and $n$ are positive integers such that $p$ divides $r$, is a subgroup of $GL_{n}(\C)$. The group $G(r,1,n)$ is the wreath product of the cyclic group $\Z/r\Z$ by the symmetric group $S_n$ and $\Grpn$ is a normal subgroup of index $p$ of $G(r,1,n)$.
Shephard and Todd \cite{ST54} gave a classification of finite irreducible complex reflection groups. It was shown there that the families of groups $S_n$ for $n>1$, $\Z/r\Z$ for $r >1$, and $\Grpn$ (except when $(r,p,n) = (2,2,2) \mbox{ or } (1,1,1)$) are the only infinite families of finite irreducible complex reflection groups and there are exactly $34$ more finite irreducible complex reflection groups. Also they characterized the group $\Grpn$, for $n>1$, by showing that these are the only finite irreducible imprimitive complex reflection groups up to isomorphism.

 The restriction of the action of $GL_n(\C)$ on $V$ to $\Grpn$ is the reflection representation of $\Grpn$. Tanabe \cite[Lemma 2.1]{T97} described a basis of the centralizer algebra of the action of $G(r,p,n)$ on the tensor space $V^{\otimes k}$. Orellana \cite{Ore07} defined a subalgebra $\mathpzc{T}_{k}(n,r)$ of partition algebra $\C A_k(n)$, and proved Schur-Weyl duality between $\T_{k}(n,r)$ and $G(r,1,n)$ \cite[Theorem 5.4]{Ore07}. Also, she recursively constructed the Bratteli diagram for the tower of algebras $$\T_{0}(n,r)\subseteq \T_{1}(n,r)\subseteq \T_{2}(n,r) \subseteq \cdots$$ in \cite[Proposition 5.6]{Ore07}.

In this paper, we define a subalgebra, denoted by $\T_k(r,p,n)$, of partition algebra $\C A_k(n)$ such that there is Schur-Weyl duality between $\T_k(r,p,n)$ and the complex reflection group $\Grpn$. In particular, for $p=1$, the algebra $\T_k(r,1,n)$ is equal to the algebra $\T_k(n,r)$ defined by Orellana. Along the lines of \cite{HR05}, we introduce a subgroup $\Lrpn$ of $\Grpn$ which plays a role analogous to the subgroup $S_{n-1}$ of $S_n$ in the study of partition algebra. We define a subalgebra, denoted by $\T_{k+\frac{1}{2}}(r,p,n)$, of partition algebra $\C A_{k+\frac{1}{2}}(n)$ and exhibit Schur-Weyl duality between $\T_{k+\frac{1}{2}}(r,p,n)$ and $\Lrpn$. Thus, the algebras $\T_k(r,p,n)$ and $\T_{k+\frac{1}{2}}(r,p,n)$ are partition algebras for the complex reflection group $\Grpn$ and its subgroup $\Lrpn$ respectively. We call the algebras $\Tn$ and $\Thn$ as {\em Tanabe algebras}. 

The main results in this paper are as follows.
\begin{enumerate}[(a)]
 \item For Tanabe algebras:
       \begin{enumerate}[(i)]
         \item Decomposition of the centralizer algebras $\mbox{End}_{\Grpn}(\Vk)$ and $\mbox{End}_{\Lrpn}(\Vk)$ into their irreducible modules which, in particular, gives parametrization of the irreducible modules of Tanabe algebras $\Tn$ and $\Thn$ for $n \geq 2k$ and $n \geq 2k+1$ respectively (Theorem \ref{irrend}). 
         
         \item Construction of Bratteli diagram recursively for the tower 
          $$\T_0(r,p,n) \subseteq \T_{\frac{1}{2}}(r,p,n) \subseteq \T_1(r,p,n) \subseteq \T_{\frac{3}{2}}(r,p,n) \subseteq \cdots \subseteq \T_{\lfloor \frac{n}{2}\rfloor}(r,p,n) .$$
          In this case, the Bratteli diagram is a simple graph (Section \ref{bd}).
        \item Description of a specific set of commuting elements, called Jucys-Murphy elements, which act as scalars on the canonical basis, called Gelfand-Tsetlin basis, of each irreducible module of Tanabe algebras (Theorem \ref{jm}).
       \end{enumerate}
 \item For complex reflection groups:
       \begin{enumerate}[(i)]
        \item Construction of a basis of irreducible  $\Grpn$-modules (Theorem \ref{eb}) using a combination of ideas from Okounkov-Vershik approach to the representation theory of $G(r,1,n)$ in \cite{MS16}, Clifford theory and higher Specht polynomials in \cite{MY98}.
        \item Branching rule from $\Grpn$ to $\Lrpn$ (Theorem \ref{brgpl}).
        \item Decomposition of $\Vk$ in terms of irreducible $\Grpn$-modules and $\Lrpn$-modules (Theorem \ref{decvk}).
       \end{enumerate}
\end{enumerate}

Using theory of the basic construction, \cite[Theorem 3.27]{HR05} shows that the necessary and sufficient condition for the semisimplicity of partition algebra $\C A_l(n)$, for $n \in \Z_{\geq 2}$ and $l \in \frac{1}{2}\Z_{\geq 0}$, is $ l \leq \frac{n+1}{2}$. In the case of Tanabe algebras, an important question that still remains to be done is to find a necessary and sufficient condition for their semisimplicity.

The inductive approach to the representation theory of symmetric groups was done by Okounkov and Vershik in \cite{vo, vo1}. This approach considers the chain of symmetric groups $$\{1\} = S_1 \subset S_2 \subset \cdots \subset S_n \subset \cdots $$ to study their representation theory recursively. The advantage over the traditional approach is that the appearance of Young diagrams and standard Young tableaux is given a spectral explanation, and the braching rule is determined simultaneously. The Gelfand-Tsetlin decomposition, the Gelfand-Tsetlin algebra, the canonical Gelfand-Tsetlin basis of the irreducible representations, and the Jucys-Murphy elements, a set of generators of Geland-Tsetlin algebra, are fundamental to this approach. The corresponding approach in the case of $G(r,1,n)$ proves fruitful in giving new proofs of some known results and also in establishing new results in this paper.

This paper is organized in the following sections. Section \ref{pre} gives a brief introduction to partition algebra, Okounkov-Vershik approach to the representation theory of $G(r,1,n)$, and Clifford theory. 
In Section \ref{ta}, we define Tanabe algebras, $\Tn$ and $\Thn$, as subspaces of partition 
algebras and prove that these subspaces are algebras.

Section \ref{crg} contains a description of representation theory of complex reflection group $G(r,p,n)$ and its subgroup $L(r,p,n)$ (Theorems \ref{irrgp} and \ref{irrlp}). We review the representation theory of $\Grpn$ using Clifford theory. We parametrize the irreducible $L(r,1,n)$-modules and, then by Clifford theory, determine the representation theory of $L(r,p,n)$. This section concludes with the branching rule from $\Grpn$ to $\Lrpn$. 

In Section \ref{swdu}, we demonstrate that $\Tn$ and $\Thn$ are in Schur-Weyl duality with $\Grpn$ and $\Lrpn$ respectively.
Using results from Section \ref{crg}, Section \ref{bd} starts with the decomposition of $\Vk$ as $\Grpn$-module and $\Lrpn$-module. Then, we give decomposition of $\Vk$ as $(\Grpn, \Tn)$-bimodule and $(\Lrpn, \Thn)$-bimodule (Theorem \ref{dvkp}) and use it to construct Bratteli diagram of Tanabe algebras. 
In Section \ref{jme}, we give Jucys-Murphy elements and their actions on the canonical Gelfand-Tsetlin basis of irreducible modules of Tanabe algebras.

{\bf Conventions:} Throughout this paper, we assume that 
\begin{enumerate}[(i)]
 \item $r,p,m$ and $n$ are positive integers such that $p$ divides $r$ and $m = \frac{r}{p}$, and 
 \item we index the components in a $w$-tuple from $1, \ldots, w$, therefore, for a multiple $t$ of $w$, the $t(\mbox{mod }w)$-th component means the $w$-th component.
\end{enumerate}

\section{Preliminaries}\label{pre}
In this section, we give an overview of partition algebra, Okounkov-Vershik approach and Clifford theory to set up notations and to state basic definitions and results used in the rest of the paper.

\subsection{Partition algebra}\label{pa}
For $k \in \Z_{\geq 0}$, let $A_k$ be the set of all set partitions of $\{1,2,\ldots,k,1', 2', \ldots, k'\}$. Given an element $d \in A_k$, we say that $i$ and $j$ are in the same block in $d$ if $i$ and $j$ belong to the same set partition in $d$. The elements of $A_k$ can be depicted as graphs, called partition diagrams, with the vertices $\{1,2,, \ldots, k\}$ and $\{1',2', \ldots, k'\}$ in the top and bottom rows respectively and two vertices in the same block are connected by an edge. By $d = (B_1, B_2, \ldots, B_s)$, we denote that there are exactly $s$ blocks $B_1, B_2, \ldots, B_s$ in $d$. Also, $|d|$ denotes the number of blocks in $d$. 

The multiplication of two elements $d_1, d_2 \in A_k$, denoted by $d_1 \circ d_2$, is obtained by concatenating the diagrams $d_1$ and $d_2$ in the following way: place $d_1$ above $d_2$, identify the vertices in the bottom row of $d_1$ with the vertices in the top row of $d_2$, then remove all the connected blocks which are entirely in the middle row. The multiplication $\circ$ makes $(A_k, \circ)$ a monoid with the identity element given by 

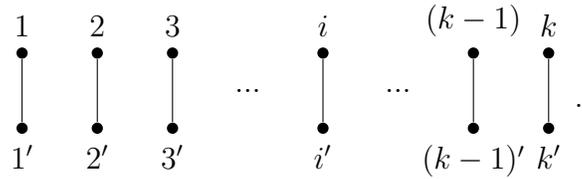
\begin{center}
\begin{tikzpicture}[scale=1,mycirc/.style={circle,fill=black, minimum size=0.1mm, inner sep = 1.5pt}]
    \node[mycirc,label=above:{$1$}] (n1) at (0,1) {};
    \node[mycirc,label=above:{$2$}] (n2) at (1,1) {};
    \node[mycirc,label=above:{$3$}] (n3) at (2,1) {};
    \node(a) at (3,0.5) {...};
    \node[mycirc,label=above:{$i$}] (ni) at (4,1) {};
    \node(b) at (5,0.5) {...};
    \node[mycirc,label=above:{$(k-1)$}] (nk-1) at (6,1) {};
    \node[mycirc,label=above:{$k$}] (nk) at (7,1) {};
    \node[mycirc,label=below:{$1'$}] (n1') at (0,0) {};
    \node[mycirc,label=below:{$2'$}] (n2') at (1,0) {};
    \node[mycirc,label=below:{$3'$}] (n3') at (2,0) {};
    \node[mycirc,label=below:{$i'$}] (ni') at (4,0) {};
    \node[mycirc,label=below:{$(k-1)'$}] (nk-1') at (6,0) {};
    \node[mycirc,label=below:{$k'$}] (nk') at (7,0) {};
    \draw (n1) -- (n1');
    \draw (n2) -- (n2');
    \draw (n3) -- (n3');
    \draw (ni) -- (ni');
    \draw (nk-1) -- (nk-1');
    \draw (nk) -- (nk');
    \node(a) at (7.4,0.3) {.};
\end{tikzpicture}
\end{center}

Define a subset $A_{k+\frac{1}{2}}$ of $A_{k+1}$ consisting of those elements which have $(k+1)$ and $(k+1)'$ in the same block. It can be easily seen that $A_{k+\frac{1}{2}}$ is a submonoid of $A_{k+1}$. The monoids $A_k$ and $A_{k+\frac{1}{2}}$ are called partition monoids. 

\begin{example}\label{pd} Taking $k=6$, the elements $d_1$ and $d_2$ in $A_6$ with $$d_1 = (\{1,2,1'\},\{3,5,3'\},\{4\},\{6,5'\},\{2',4'\},\{6'\}) \in A_6$$ $$\mbox{ and }d_2 = (\{1,5,2',3'\},\{2,4\},\{3\},\{6,6'\},\{1',4'\},\{5'\}) \in A_6,$$ can be written in terms of partition diagrams as: 

  
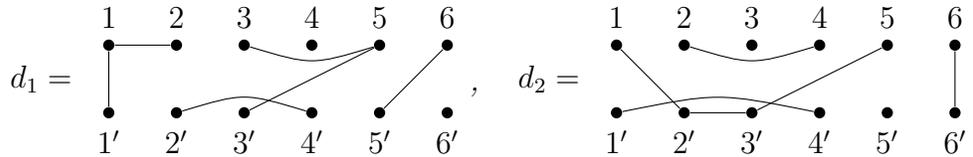
\begin{tikzpicture}[scale=0.9,mycirc/.style={circle,fill=black, minimum size=0.1mm, inner sep = 1.5pt}
    ]
    \node (1) at (-1,0.5) {$d_1 =$};
    \node[mycirc,label=above:{$1$}] (n1) at (0,1) {};
    \node[mycirc,label=above:{$2$}] (n2) at (1,1) {};
    \node[mycirc,label=above:{$3$}] (n3) at (2,1) {};
    \node[mycirc,label=above:{$4$}] (n4) at (3,1) {};
    \node[mycirc,label=above:{$5$}] (n5) at (4,1) {};
    \node[mycirc,label=above:{$6$}] (n6) at (5,1) {};
    \node[mycirc,label=below:{$1'$}] (n1') at (0,0) {};
    \node[mycirc,label=below:{$2'$}] (n2') at (1,0) {};
    \node[mycirc,label=below:{$3'$}] (n3') at (2,0) {};
    \node[mycirc,label=below:{$4'$}] (n4') at (3,0) {};
    \node[mycirc,label=below:{$5'$}] (n5') at (4,0) {};
    \node[mycirc,label=below:{$6'$}] (n6') at (5,0) {};
    \draw (n1) -- (n2);
    \draw (n1) -- (n1');
    \draw (n3) .. controls(3,0.7) .. (n5) -- (n3');
    \draw (n2') .. controls(2,0.3) .. (n4');
    \draw (n6) -- (n5');
    \node(a) at (5.4,0.3) {,};

    \node (1) at (6.5,0.5) {$d_2 =$};
    \node[mycirc,label=above:{$1$}] (n7) at (7.5,1) {};
    \node[mycirc,label=above:{$2$}] (n8) at (8.5,1) {};
    \node[mycirc,label=above:{$3$}] (n9) at (9.5,1) {};
    \node[mycirc,label=above:{$4$}] (n10) at (10.5,1) {};
    \node[mycirc,label=above:{$5$}] (n11) at (11.5,1) {};
    \node[mycirc,label=above:{$6$}] (n12) at (12.5,1) {};
    \node[mycirc,label=below:{$1'$}] (n7') at (7.5,0) {};
    \node[mycirc,label=below:{$2'$}] (n8') at (8.5,0) {};
    \node[mycirc,label=below:{$3'$}] (n9') at (9.5,0) {};
    \node[mycirc,label=below:{$4'$}] (n10') at (10.5,0) {};
    \node[mycirc,label=below:{$5'$}] (n11') at (11.5,0) {};
    \node[mycirc,label=below:{$6'$}] (n12') at (12.5,0) {};
    \draw (n7) -- (n8') -- (n9') -- (n11);
    \draw (n8) .. controls(9.5, 0.7) .. (n10);
    \draw (n7') .. controls(9,0.3) .. (n10');
    \draw (n12) -- (n12');
    \node(a) at (12.8,0.3) {.};   
    \end{tikzpicture}

\noindent The multiplication $d_1 \circ d_2$ is:
  
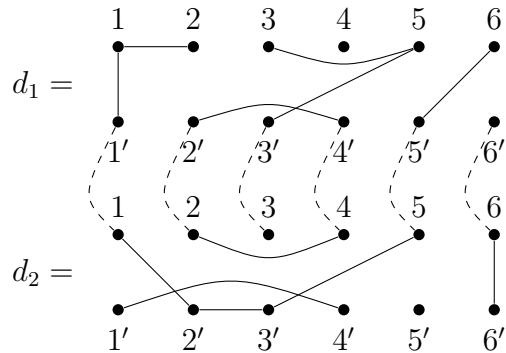
\begin{center} 
  \begin{tikzpicture}[scale=1,mycirc/.style={circle,fill=black, minimum size=0.1mm, inner sep = 1.5pt}
    ]
    \node (1) at (-1,0.5) {$d_1 =$};
    \node[mycirc,label=above:{$1$}] (n1) at (0,1) {};
    \node[mycirc,label=above:{$2$}] (n2) at (1,1) {};
    \node[mycirc,label=above:{$3$}] (n3) at (2,1) {};
    \node[mycirc,label=above:{$4$}] (n4) at (3,1) {};
    \node[mycirc,label=above:{$5$}] (n5) at (4,1) {};
    \node[mycirc,label=above:{$6$}] (n6) at (5,1) {};
    \node[mycirc,label=below:{$1'$}] (n1') at (0,0) {};
    \node[mycirc,label=below:{$2'$}] (n2') at (1,0) {};
    \node[mycirc,label=below:{$3'$}] (n3') at (2,0) {};
    \node[mycirc,label=below:{$4'$}] (n4') at (3,0) {};
    \node[mycirc,label=below:{$5'$}] (n5') at (4,0) {};
    \node[mycirc,label=below:{$6'$}] (n6') at (5,0) {};
    \draw (n1) -- (n2);
    \draw (n1) -- (n1');
    \draw (n3) .. controls(3,0.7) .. (n5) -- (n3');
    \draw (n2') .. controls(2,0.3) .. (n4');
    \draw (n6) -- (n5');

    \node (1) at (-1,-2) {$d_2 =$};
    \node[mycirc,label=above:{$1$}] (n7) at (0,-1.5) {};
    \node[mycirc,label=above:{$2$}] (n8) at (1,-1.5) {};
    \node[mycirc,label=above:{$3$}] (n9) at (2,-1.5) {};
    \node[mycirc,label=above:{$4$}] (n10) at (3,-1.5) {};
    \node[mycirc,label=above:{$5$}] (n11) at (4,-1.5) {};
    \node[mycirc,label=above:{$6$}] (n12) at (5,-1.5) {};
    \node[mycirc,label=below:{$1'$}] (n7') at (0,-2.5) {};
    \node[mycirc,label=below:{$2'$}] (n8') at (1,-2.5) {};
    \node[mycirc,label=below:{$3'$}] (n9') at (2,-2.5) {};
    \node[mycirc,label=below:{$4'$}] (n10') at (3,-2.5) {};
    \node[mycirc,label=below:{$5'$}] (n11') at (4,-2.5) {};
    \node[mycirc,label=below:{$6'$}] (n12') at (5,-2.5) {};
    \draw (n7) -- (n8') -- (n9') -- (n11);
    \draw (n8) .. controls(2, -1.9) .. (n10);
    \draw (n7') .. controls(1.5,-2) .. (n10');
    \draw (n12) -- (n12');
    
    \draw[dashed] (n1') .. controls(-0.5,-1)..(n7);
    \draw[dashed] (n2') .. controls(0.5,-1)..(n8);
    \draw[dashed] (n3') .. controls(1.5,-1)..(n9);
    \draw[dashed] (n4') .. controls(2.5,-1)..(n10);
    \draw[dashed] (n5') .. controls(3.5,-1)..(n11);
    \draw[dashed] (n6') .. controls(4.5,-1)..(n12);
\end{tikzpicture}
\end{center}

Thus,

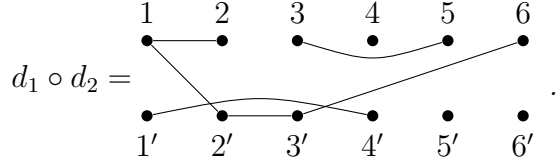
\begin{center}

\begin{tikzpicture}[scale=1,mycirc/.style={circle,fill=black, minimum size=0.1mm, inner sep = 1.5pt}
    ]
    \node (1) at (-1,0.5) {$d_1 \circ d_2 =$};
    \node[mycirc,label=above:{$1$}] (n1) at (0,1) {};
    \node[mycirc,label=above:{$2$}] (n2) at (1,1) {};
    \node[mycirc,label=above:{$3$}] (n3) at (2,1) {};
    \node[mycirc,label=above:{$4$}] (n4) at (3,1) {};
    \node[mycirc,label=above:{$5$}] (n5) at (4,1) {};
    \node[mycirc,label=above:{$6$}] (n6) at (5,1) {};
    \node[mycirc,label=below:{$1'$}] (n1') at (0,0) {};
    \node[mycirc,label=below:{$2'$}] (n2') at (1,0) {};
    \node[mycirc,label=below:{$3'$}] (n3') at (2,0) {};
    \node[mycirc,label=below:{$4'$}] (n4') at (3,0) {};
    \node[mycirc,label=below:{$5'$}] (n5') at (4,0) {};
    \node[mycirc,label=below:{$6'$}] (n6') at (5,0) {};
    \draw (n1) -- (n2);
    \draw (n1) -- (n2') -- (n3') -- (n6);
    \draw (n3) .. controls(3,0.7) .. (n5);
    \draw (n1') .. controls(1.5,0.3) .. (n4');
    \node(a) at (5.4,0.3) {.};
    \end{tikzpicture}
\end{center}
\end{example}
 
For a complex number $q$, define 
\begin{center}
$\C A_k(q) := \C\mbox{-}span \{d \in A_k\}$.
\end{center}
The multiplication of basis elements, which when extended linearly makes $\C A_k(q)$ an associative algebra, is defined as: for $d_1, d_2 \in A_k$, define 
\begin{equation*}
 d_1d_2 := q^l d_1\circ d_2
\end{equation*}
where $l$ is the number of blocks removed from the middle row while computing $d_1\circ d_2$. Also, $\C A_{k+\frac{1}{2}}(q)$ is a subalgebra of $\C A_{k+1}(q)$. The algebras $\C A_k(q)$ and $\C A_{k+\frac{1}{2}}(q)$ are called {\em partition algebras}.

\begin{example}
In example \ref{pd}, the product $d_1d_2$ in $\C A_6(q)$ is given by

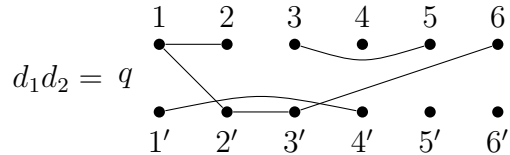
\begin{center}
 \begin{tikzpicture}[scale=0.9,mycirc/.style={circle,fill=black, minimum size=0.1mm, inner sep = 1.5pt}
    ]
    \node (1) at (-1.5,0.5) {$d_1d_2 =$};
    \node (2) at (-0.5,0.5) {$q$};
    \node[mycirc,label=above:{$1$}] (n1) at (0,1) {};
    \node[mycirc,label=above:{$2$}] (n2) at (1,1) {};
    \node[mycirc,label=above:{$3$}] (n3) at (2,1) {};
    \node[mycirc,label=above:{$4$}] (n4) at (3,1) {};
    \node[mycirc,label=above:{$5$}] (n5) at (4,1) {};
    \node[mycirc,label=above:{$6$}] (n6) at (5,1) {};
    \node[mycirc,label=below:{$1'$}] (n1') at (0,0) {};
    \node[mycirc,label=below:{$2'$}] (n2') at (1,0) {};
    \node[mycirc,label=below:{$3'$}] (n3') at (2,0) {};
    \node[mycirc,label=below:{$4'$}] (n4') at (3,0) {};
    \node[mycirc,label=below:{$5'$}] (n5') at (4,0) {};
    \node[mycirc,label=below:{$6'$}] (n6') at (5,0) {};
    \draw (n1) -- (n2);
    \draw (n1) -- (n2') -- (n3') -- (n6);
    \draw (n3) .. controls(3,0.7) .. (n5);
    \draw (n1') .. controls(1.5,0.3) .. (n4');
    \end{tikzpicture}
\end{center}    
since one block has been removed from the middle row.
\end{example}

\subsection{The Okounkov-Vershik approach}\label{ov}
Let $G^n$ denote the direct product of $n$-copies of a finite group $G$.
The action of the symmetric group $S_n$ on $G^n$ by permuting the coordinates defines the semidirect product of $G^n$ by $S_n$. The group $G^n \rtimes S_n$ is also known as wreath product of $G$ by $S_n$. We use the notation $G(r,1,n) := G^n \rtimes S_n$ throughout for the particular case when $G = \Z/r\Z = \langle \zeta \rangle$, the cyclic group of order $r$ with $\zeta$ being a primitive $r$-th root of unity. Thus, $$G(r,1,n) = \{(g_1, g_2, \ldots, g_n, \pi) \mid g_i  \in \Z/r\Z \mbox{ for } i = 1, \ldots, n \mbox{ and } \pi \in S_n \}.$$

In this section, we follow \cite{MS16} and present here a brief summary of Okounkov-Vershik approach to the representation theory of $G(r,1,n)$.

Consider the following chain of subgroups of $G(r,1,n)$
\begin{equation} \label{mfic}
H_{1,n}\subseteq H_{2,n} \subseteq \cdots \subseteq H_{n,n} := G(r,1,n),
\end{equation}
where, for $1\leq i \leq n$, 
$$ H_{i,n} := \{(g_1,\ldots ,g_n,\pi)\in G(r,1,n) \;|\;
            \pi(j)=j \mbox{ for }i+1\leq j\leq n\}.
$$
The irreducible representations of  $H_{1,n} \cong G^n = (\Z/r\Z)^n$ are one-dimensional. 

The following well-known result of Wigner is useful in proving that the chain (\ref{mfic}) is multiplicity free. 
\begin{thm} \label{mfc} 
Let $M$ be a complex finite dimensional semisimple algebra and let $N$ be a
semisimple subalgebra. 
Then the relative commutant of the pair $M$ and $N$, denoted by $Z(M,N)$, consisting of all the elements of $M$ that commute with the elements of $N$ is semisimple and the
following conditions are equivalent:
\begin{enumerate}[(a)]
\item The restriction of any finite dimensional complex irreducible
representation of $M$ to $N$ is multiplicity free.

\item The relative commutant $Z(M,N)$ is commutative.
\end{enumerate}
\end{thm}

Using Theorem 4.3 in \cite{MS16}, we can conclude that the relative commutant of the pair of group algebras $\C [H_{m,n}]$ and $\C [H_{m-1,n}]$ is commutative for all $2 \leq m \leq n$.

For each $i=1,2, \ldots, n$, suppose that $H_{i,n}^{\wedge}$ denotes the indexing set of irreducible $H_{i,n}$-modules and given $\La \in H_{i,n}^{\wedge}$, assume that $V^{\La}$ denotes the corresponding $H_{i,n}$-module. Bratteli diagram of the chain (\ref{mfic}) is a simple graph in which the vertices at $i$-th level are elements of $H_{i,n}^{\wedge}$ and a vertex $\mu\in H_{i-1,n}^{\wedge}$ is joined by an edge with a vertex $\lambda\in H_{i,n}^{\wedge}$ if $V^\mu$ appears in the restriction of $V^\lambda$ to $H_{i-1,n}$. 

For a fixed $1 \leq m \leq n$, consider the $H_{m,n}$-module $V^{\La_m}$, where $\lambda_m\in H_{m,n}^{\wedge}$. The branching rule being multiplicity free implies that the decomposition of $V^{\La_m}$ into irreducible $H_{m-1,n}$-modules is canonical, and the decomposition is 
$$V^{\lambda_m}=\bigoplus_{\La_{m-1}} V^{\La_{m-1}},$$
where the sum is over all $\La_{m-1}\in H_{m-1,n}^{\wedge}$ with an edge from  $\La_{m-1}$ to $\La_m$ such that $V^{\La_{m-1}}$ is identified with the corresponding submodule of $V^{\lambda_m}$. Iterating this decomposition, a canonical decomposition of $V^{\lambda_m}$ into irreducible $H_{1,n}^{\wedge}$-submodules is
\begin{equation} \label{decomp}
V^{\lambda_m}=\bigoplus_{T} V_T,
\end{equation}
where the sum is over all possible paths $T=(\lambda_1, \lambda_2, \ldots, \lambda_m)$ from a vertex in $H_{1,n}$ to $\La_m$ in Bratteli diagram.
\begin{equation} \label{ch}
T=\lambda_1\nearrow\lambda_2 \nearrow \cdots \nearrow \lambda_m,
\end{equation}
with $\lambda_i\in H_{i,n}^{\wedge}$ for $ 1 \leq i \leq m$.

The decomposition (\ref{decomp}) is called the {\em Gelfand-Tsetlin decomposition (GZ-decomposition)} of $V^{\lambda_m}$ and each $V_T$ in (\ref{decomp}) is called a {\em Gelfand-Tsetlin subspace (GZ-subspace)} of $V^{\lambda_m}$. In our case, each GZ-subspace $V_T$ is one-dimensional. Choose a non-zero vector $v_T \in V_T$. The basis $$\{v_T \mid T \mbox{ is a chain in GZ-decomposition of } V^{\La_m}\}$$ of $V^{\La_m}$ is called the {\em Gelfand-Tsetlin basis (GZ-basis)} of $V^{\La_m}$ and it is unique up to scalars
and
\begin{equation*}
\C[H_{i,n}]\cdot v_T = V^{\lambda_{i}},\;\;\;i=1,2,\ldots ,m.
\end{equation*}

Considering the Fourier transform, i.e., the algebra isomorphism
\begin{equation}\label{iso}
\C[H_{m,n}] \cong \bigoplus_{\La_m\in H_{m,n}^{\wedge}} \mbox{End}(V^{\La_m}),
\end{equation}
given by
$$g \mapsto ( V^{\La_m} \buildrel {g}\over \rightarrow V^{\La_m}\;:\;
\La_m
\in H_{m,n}^{\wedge}),\;\;g\in H_{m,n}, 1 \leq m \leq n, $$
we can define {\em Gelfand-Tsetlin algebra (GZ-algebra)}, a subalgebra of $\C[H_{m,n}]$ based on the GZ-decomposition (\ref{decomp}):
\begin{align*}
 GZ_{m,n}  = \{ a\in \C[H_{m,n}] \mid \, & a \mbox{ acts diagonally in the GZ-basis of } V^{\lambda_m},\mbox{ for all }\lambda_m\in H_{m,n}^{\wedge}\}.
\end{align*}
Thus, $GZ_{m,n}$ is a maximal commutative subalgebra of $\C[H_{m,n}]$.
\begin{thm}\cite[Theorem 3.1(i)]{MS16} Let $Z_{i,n}$ denote the center of $\C[H_{i,n}]$ for each $i =1,2, \ldots,m$. Then, 
$$GZ_{m,n} = \langle Z_{1,n}, Z_{2,n}, \ldots ,Z_{m,n} \rangle.$$
\end{thm}

By a $GZ$-subspace of $H_{m,n}$ we mean a $GZ$-subspace of some irreducible $H_{m,n}$-module $V^{\lambda_m}$ for some $\lambda_m \in H_{m,n}^{\wedge}$. 

The theorem above implies the following result which is \cite[Lemma 3.2]{MS16}.

\begin{lemma} \phantomsection \label{sc}
\begin{enumerate}[(a)]
\item Let $v\in V^{\lambda_m}$ for some $\lambda_m\in H_{m,n}^{\wedge}$ such that $v$ is an eigenvector of every element of $GZ_{m,n}$, then (a scalar multiple of) $v$ belongs to the $GZ$-basis of $V^{\lambda_m}$.

\item Let $v$ and $u$ be elements in $V^{\La_m}$ and $V^{\mu_m}$ respectively for some $\La_m, \mu_m \in H_{m,n}^{\wedge}$ such that $v$ and $u$ have the same eigenvalues for every element of $GZ_{m,n}$, then $v=u$ and $\La_m = \mu_m$.
\end{enumerate}
\end{lemma}
Thus, a $GZ$-subspace uniquely determines the irreducible representation of $H_{m,n}$ of which it is a $GZ$-subspace. 

The Jucys-Murphy elements for the wreath product of a finite group by a symmetric group were given in \cite{Pu97}. For our particular case $G(r,1,n)$,
the Jucys-Murphy elements can be written as:   
\begin{align*}
& X_1 = 0, \\
\mbox{ and }&  X_j = \sum_{i=1}^{j-1}\sum\limits_{l=0}^{r-1}\zeta_i^{l}\zeta_j^{-l}s_{ij},\;\;2\leq
j\leq n
\end{align*}
where $\zeta_i^{l}\zeta_j^{-l}s_{ij} = (1, \ldots, 1, \g^l, 1 \ldots, 1, \g^{-l}, 1, \ldots, 1, s_{ij}) \in G(r,1,n)$, with $\g^l$ and $\g^{-l}$ as $i$-th and $j$-th coordinates respectively. It is clear that the element $X_j$ belongs to $H_{j,n}$ also.

\begin{thm}\cite[Theorem 4.4]{MS16}\label{yjm3} We have
$$GZ_{m,n} = \langle Z[\C[G^n]], X_1,X_2,\ldots ,X_m \rangle. $$
\end{thm}

A $GZ$-subspace $V$ of $H_{m,n}$ is isomorphic to $\rho_1\otimes \cdots \otimes
\rho_n,\;\rho_i\in G^{\wedge}$ for all $i$. We call $\rho=(\rho_1,\ldots
,\rho_n)$ the {\em label} of $V$.  And define the {\em weight} $\alpha(V)$ of $V$ by  
\begin{equation}
\alpha(V) = (\rho,\alpha_1,\ldots ,\alpha_m),
\end{equation}
where $\alpha_i = \mbox{eigenvalue of $X_i$ on $V$}$.
Using Lemma \ref{sc} and Theorem \ref{yjm3}, it can be easily shown that a $GZ$-subspace is uniquely determined by its weight.

Let $\Y$ be the set of all Young diagrams. The unique Young diagram with zero boxes is empty Young diagram denoted by $\emptyset$. For $\lambda \in \Y$, let $|\lambda|$ denote the number of boxes in $\La$. Define $$\Y(r,n) := \{\lambda = (\lambda_1, \lambda_2, \ldots, \lambda_r) \mid \La_i \in \Y \mbox{ for all }i =1,2, \ldots,r \mbox{ and } \sum\limits_{i=1}^{r} |\La_i|=n\}, $$ i.e., $\Y(r,n)$ is the set of $r$-tuples of Young diagrams such that the total number of boxes is $n$. The irreducible representations of $G(r,1,n)$ are parametrized by elements of the set $\Y(r,n)$.  

Let $\mu\in \Y$. A standard Young tableau of shape $\mu$ is obtained by filling the boxes in the Young diagram $\mu$ with the distinct numbers $1,2,\ldots ,| \mu |$ such that the numbers in the boxes strictly increase along each row and each column of $\mu$. For $\La \in \Y(r,n)$, a standard $r$-tuple of Young tableau of shape $\La$ obtained by filling the $n$-boxes in the $r$-tuple $\La$ with the distinct numbers $1,2,\ldots ,n$ such that the numbers in the boxes strictly increase along each row and each column of all Young diagrams occurring in $\La$. Define
$\Tab$ as the set of all standard $r$-tuple of Young tableau and set $\mbox{Tab}(r,n) := \cup_{\La \in \Y(r,n)}\Tab$.  

For each $i = 1,2, \ldots,r$, define the irreducible representation $\sigma_i$ of $G$:
\begin{align*}
 \sigma_i:~  G & \rightarrow \C^* \\
 \zeta & \mapsto \zeta^{i-1}.
\end{align*}
The irreducible representations of $G$ are $\sigma_1, \sigma_2, \ldots, \sigma_r$.

The content $c(b)$ of a box $b$ of a Young diagram is its $y$-coordinate $-$ its $x$-coordinate (We draw Young diagrams by following the convention of writing down matrices with $x$-axis running downwards and $y$-axis running to the right). Given $\La = (\La_1, \ldots, \La_r) \in \Y(r,n)$, $T \in \Tab$ and $1 \leq i \leq n$, the number $i$ resides in exactly one box of one of $\La_1, \ldots, \La_r$, let $b_T(i)$ be this box in Young diagram $\La_{j_i}$ for a unique $j_i \in \{1, \ldots, r\}$ and let $r_T(i) := \sigma_{j_i}$ . 
 
The following result for $G(r,1,n)$ can be easily seen from Theorem 6.5 in \cite{MS16}.
\begin{thm} \label{gzt}
Let $\La \in \Y(r,n)$. Then the GZ-subspaces of
$V^{\La}$ can be parametrized by $T \in \Tab$ and the GZ-decomposition of $V^{\La}$ can be written as
\begin{equation} \label{gzd} 
V^{\La} = \bigoplus\limits_{T \in \Tab} V_T,
\end{equation}
where each $V_T$ is closed under the action of $G^n$ and, as a $G^n$-module, is isomorphic to the irreducible
$G^n$-module
$$r_T(1)\otimes r_T(2)\otimes \cdots \otimes r_T(n)$$  
For $i=1,\ldots ,n$, the eigenvalue of $X_i$ on $V_T$ is given by
$rc(b_T(i))$.  
\end{thm}

Let $R$ denote the element of $\Tab$ defined as follows: for $\La = (\La_1, \ldots, \La_r)$, we start with $\La_1$ by filling the Young diagram $\La_1$ with the
numbers $1,\ldots ,|\La_1|$ in {\em row major order}, i.e., the first row is filled with $1,2,\ldots ,l_1$ in increasing order where $l_1$ is the length of the first row, the second row is filled with $l_1+1,\ldots ,l_1+l_2$ in increasing order where 
$l_2$ is the length of the second row and so on till the last row of
$\La_1$ has been filled. Then we fill the Young diagram $\La_2$ with $|\La_1|+1,\ldots ,|\La_1|+|\La_2|$ in row major order and so on till the last Young diagram $\La_r$.

The irreducible representations of $G(r,1,n)$ are parametrized by $r$-tuple of Young diagrams in $\Y(r,n)$ and given $\La \in \Y(r,n)$, the $GZ$-basis elements (and hence, $GZ$-subspaces) of $V^{\La}$  are parametrized by $T \in \Tab$.

\subsection{Clifford Theory}\label{ct}
We give an outline of Clifford theory for a finite group $H$ and its normal subgroup $N$ such that $H/N$ is a cyclic group of order $p$ as done in \cite{S89, MY98, BB07}. 
The pair $H$ and $N$ on which they have applied Clifford theory is the pair $G(r,1,n)$ and $\Grpn$.
The group $\Grpn$ can be considered as the subgroup of $GL_n(\C)$ consisting of generalized permutation matrices such that the $m$-th power of the product of nonzero entries is one. We discuss the complex reflection group $\Grpn$ and its representation theory in detail in Section \ref{crg} and review Clifford theory for the rest of this section.

Let $H^{\wedge}$ denote the indexing set of irreducible representations of $H$. 
Identifying $H/N$ with the group $C$ consisting of one-dimensional representations of $H$ which contain $N$ in their kernel, we can define an action of $C$ on the set of irreducible representations of $H$ by
\begin{align*}
V^\rho \mapsto \delta \otimes V^\rho
\end{align*}
where $\delta \in C$ and $V^\rho$ is the irreducible representation of $H$ indexed by $\rho \in H^{\wedge}.$ Denote the orbit of $V^{\rho}$ by $[\rho]$ with respect to the action of $C$. The irreducible representations of $H$ which are in the same orbit are called associates of each other. Assume that $V^\rho$ has $b(\rho)$ associates. Then  the stabilizer subgroup of $C$ with respect to $V^\rho$, denoted by $C_{\rho}$, has the order $u(\rho) = \frac{p}{b(\rho)}$. Suppose that $\delta_0$ is a generator of $C_{\rho}$. It is easy to see that there exists a $N$-linear map $A : V^\rho \longrightarrow V^\rho$ such that $A(hv)= \delta_0(h)hA(v)$ for all $h \in H$ and $v \in V^\rho$. Then by Schur's lemma, the $H$-linear map $A^{u(\rho)}$ acts by a nonzero scalar. Normalizing the scalar, we obtain that $A^{u(\rho)}$ is the identity map on $V^\rho$. Such an $A$ is called the {\em associator} of $V^{\rho}$. Also, if $\mu \in [\rho]$, then the stabilizer subgroup $C_\mu = C_\rho.$ The following theorem gives parametrization of irreducible $N$-modules.

\begin{thm}\phantomsection\label{cd}
\begin{enumerate}[(a)]
\item The eigenspace decomposition of $V^\rho$ with respect to $A$ is given by
 \begin{equation}\label{esdec}
  V^\rho \cong \bigoplus_{l=0}^{u(\rho)-1} E^{(l)},
 \end{equation}
where $E^{(l)}$ is the eigenspace with eigenvalue $e^{\frac{2\pi i l}{u(\rho)}}$. The group $C_\rho$ can be identified with $\{e^{\frac{2\pi i l}{u(\rho)}} \mid l = 0,1, \ldots, u(\rho)-1\}.$
\item The eigenspaces $E^{(l)}$, for $0 \leq l \leq u(\rho)-1$, occuring in (\ref{esdec}) are inequivalent irreducible $N$-modules, each of which is of dimension $\dim(V^\rho)/u(\rho)$.   
\item For any $0 \leq l \leq u(\rho)-1$, we have $$\mbox{Ind}_N^H (E^{(l)}) = \bigoplus\limits_{\mu \in [\rho]} V^\mu.$$ 
\item Let $\mathcal{O}$ denote the set of all orbits in $H^{\wedge}$.
 The irreducible $N$-modules are parametrized by the pairs $([\rho], \epsilon)$ where $[\rho] \in \mathcal{O}$ and $\epsilon \in C_\rho$.
\end{enumerate}
\end{thm}

\section{Tanabe algebra}\label{ta}
The partition monoid is a poset with the partial order given as: for $d, d' \in A_k$, $d' \leq d$ if $d'$ is coarser than $d$, i.e., if $i$ and $j$ are in the same block of $d$, then $i$ and $j$ are in the same block of $d'$.
For $d \in A_k$, define the unique element $x_d \in \C A_k(n)$ satisfying
\begin{equation}\label{xd}
 d = \sum\limits_{d' \leq d} x_{d'}.
\end{equation}
This partial order on $A_k$ can be extended to a total order on $A_k$. 
It can be easily seen that the transition matrix between $\{d \mid d \in A_k \}$ and  $\{x_d \mid d \in A_k \}$ is an upper triangular matrix with $1'$s on the diagonal and thus, $\{x_d \mid d \in A_k \}$ is also a basis of the partition algebra $\C A_k(n)$, see also \cite[p. 879]{HR05}. 

An {\em internal block} in $ d_1 \circ d_2$, for $d_1, d_2 \in A_k$, is a block that is entirely in the middle while computing $d_1 \circ d_2$. We say that {\em the bottom row of $d_1$ matches with the top row of $d_2$} if the following condition is satisfied: $i'$ and $j'$ are in the same block in $d_1$ if and only if $i$ and $j$ are in the same block in $d_2$ for $1 \leq i, j \leq k$. For every $s$ in a block $B$ of $d \in A_k$, if we put $i_s = t$ for some $ 1 \leq t \leq n$, then $t$ is said to be a {\em mark of the block} $B$. The next lemma and the idea of its proof are from the online notes \cite{Ram2010S}. It gives the structure constants with respect to the basis $\{x_d \mid d \in A_k\}$ of $\C A_k(n)$.  

\begin{lemma}\label{mult}
 For $d_1, d_2 \in A_k$, the multiplication of $x_{d_1}$ and $x_{d_2}$ in $\C A_k(n)$ is given by 
\begin{align*}
  x_{d_1}x_{d_2} = 
  \begin{cases}
   \sum\limits_{d \in A_k} c_d x_d, & \mbox{if the bottom row of } d_1 \mbox{ matches with the top row of } d_2, \\
   0, & \mbox{otherwise,}
  \end{cases}
\end{align*}
where the sum is taken over all those $d$ in $A_k$ such that $d$ is coarser than $d_1 \circ d_2$ and the coarsening is done by connecting a block of $d_1$ which is contained entirely in the top row of $d_1$ with a block of $d_2$ which is contained entirely in the bottom row of $d_2$ and 
\begin{equation*}
c_d = (n - |d|)_{[d_1 \circ d_2]}, 
\end{equation*}
where $|d|$ is the number of blocks in $d$, $[d_1 \circ d_2]$ is the number of internal blocks in $[d_1 \circ d_2]$, and for $a \in \Z$, $b \in \Z_{\geq 0},$
\begin{align*}
  (a)_b := 
  \begin{cases}
   a(a-1) \cdots (a-b+1), & \mbox{if }  b > 0,\\
   1, & \mbox{if } b = 0, 
  \end{cases}
\end{align*}
such that when $a,b \in \Z_{\geq 0}$ and $a \geq b$, we have $(a)_b = \Myperm[a]{b}$, the number of permutations of $a$ objects taken $b$ at a time.
\end{lemma}

\begin{proof}
 Let $n \geq 2k$. Then $\phi_k: \C A_k(n) \cong \mbox{End}_{S_n}(\Vk)$ (by Schur-Weyl duality for partition algebras, Theorem \ref{oswd}). Identifying $x_d$ with $\phi_k(x_d)$, we have 
 \begin{align*}
 & (\vi) (x_{d_1}x_{d_2})  = \\
 &\sum\limits_{i_{1'}, i_{2'}, \ldots, i_{k'},i_{1''}, i_{2''}, \ldots, i_{k''}}(v_{i_{1''}} \otimes v_{i_{2''}} \otimes \cdots \otimes v_{i_{k''}})(x_{d_1})^{i_1,i_2, \ldots, i_k}_{i_{1'}, i_{2'}, \ldots, i_{k'}}(x_{d_2})^{i_{1'}, i_{2'}, \ldots, i_{k'}}_{i_{1''}, i_{2''}, \ldots, i_{k''}}
 \end{align*}
 If the bottom row of $d_1$ does not match with the top row of $d_2$, then using (\ref{xdcoeff}) it can be seen that $x_{d_1}x_{d_2} = 0$.  

 If the bottom row of $d_1$ matches with the top row of $d_2$, then again using  (\ref{xdcoeff}) we have 
 \begin{align*}
  \sum\limits_{i_{1'}, i_{2'}, \ldots, i_{k'},i_{1''}, i_{2''}, \ldots, i_{k''}} & (v_{i_{1''}} \otimes v_{i_{2''}} \otimes \cdots \otimes v_{i_{k''}})(x_{d_1})^{i_1,i_2, \ldots, i_k}_{i_{1'}, i_{2'}, \ldots, i_{k'}}(x_{d_2})^{i_{1'}, i_{2'}, \ldots, i_{k'}}_{i_{1''}, i_{2''}, \ldots, i_{k''}}  = \sum\limits_{d}\alpha_d x_d,
 \end{align*}
where $\alpha_d$ is some positive integer and the sum is over all $d$ obtained by coarsening $d_1 \circ d_2$ which is done by connecting a block of $d_1$ contained entirely in the top row of $d_1$ and a block of $d_2$ contained entirely in the bottom row of $d_2$. So, $ \alpha_d = $ number of ways the internal blocks of $d_1 \circ d_2$ can be marked distinctly after putting distinct marks on the blocks of $d$ $= (n - |d|)_{[d_1 \circ d_2]} = c_d.$

Fix k and vary n. For a given $n$, fix $d_{1}, d_{2},d\in A_{k}(n)$. Then the coefficient of $x_{d}$ in the product $x_{d_{1}}x_{d_{2}}$ is a polynomial $f_{d}(n)$ in $n$. Then by above arguments, for $n \geq 2k$, we have  $f_{d}(n)=(n-|d|)_{[d_{1}\circ d_{2}]}$. The fundamental theorem of algebra implies that $f_{d}(n)=(n-|d|)_{[d_{1}\circ d_{2}]}$ for all $n$.
\end{proof}

Let $B$ be a block of $d \in A_k$. Suppose that $N(B)$ is the number of elements in $B \bigcap \{1,2,\ldots, k\}$ and  $M(B)$ is the number of elements in $B \bigcap \{1',2',\ldots, k'\}$. Thus, $N(B)$ and $M(B)$ are the number of elements in the block $B$ in top row and bottom row of $d$ respectively. 

Define the following mutually disjoint subsets of $A_k$:
\begin{align*}
& \Pi_k(r) := \{ d =  (B_1, B_2, \ldots, B_s) \in A_k \mid s \geq 1 \mbox{ and } \\ & ~~~~~~~~~~~~~~~N(B_i) \equiv M(B_i)~(\mbox{mod }r) (1 \leq i \leq s)\}, \\
& \Lambda_k(r,p,n) := \{  d = (B_1, B_2, \ldots, B_n) \in A_k  \mid  
N(B_i) \equiv M(B_i)~(\mbox{mod }m),\\  & ~~~~~~~~~~~~~~~~~~~N(B_i) \not\equiv M(B_i)~(\mbox{mod }r), (1 \leq i \leq n), \mbox{ and } \\ & ~~~~~~~~~~~~~~~~~~~N(B_i) - M(B_i) \equiv N(B_j)-M(B_j)~(\mbox{mod } r), (1 \leq i, j \leq n)\}, \\
& \Theta_k(r,p,n) := \{  d = (B_1, B_2, \ldots, B_y) \in A_k  \mid  y > n, \\ 
& ~~~~~~~~~~~~~~~~~~~ N(B_i) \equiv M(B_i)~(\mbox{mod }m),(1 \leq i \leq y), \\
& ~~~~~~~~~~~~~~~~~~~ \mbox{  and for some } j \in \{1,\ldots,y\}, N(B_j) \not\equiv M(B_j)~(\mbox{mod }r)  \}.
\end{align*}
Also, define $A_k(r,p,n)$, a subset of $A_k$, by setting
\begin{equation*}
 A_k(r,p,n) := \Pi_k(r) \bigcup \Lambda_k(r,p,n) \bigcup \Theta_k(r,p,n).
\end{equation*}

\begin{definition}
Define $\Tn:= \C$-span$\{x_d \mid d \in A_k(r,p,n)\}$, a subspace of partition algebra $\C A_k(n)$.
\end{definition}
\begin{remark}\label{sm}
The set $\Pi_k(r)$ is a submonoid of $A_k$ and $A_k(r,1,n) = \Pi_k(r)$. Also, for $d \in \Pi_k(r)$, the elements $d' \leq d$ also belong to $\Pi_k(r)$ because the difference in the number of elements between top row and bottom row in each block remains $0~(\mbox{mod }r)$ even after coarsening. Thus, 
\begin{equation*}
 \T_k(r,1,n) = \C\mbox{-}span\{x_d \mid d \in \Pi_k(r) \} = \C\mbox{-}span\{d \in \Pi_k(r)\} 
\end{equation*}
is a subalgebra of $\C A_k(n)$.
\end{remark}

Let $V = \C^n$ be the $n$-dimensional vector space on which $GL_n(\C)$ acts naturally. The action of $\Grpn$ on $V$ is given by the restriction of the action of $GL_n(\C)$ on $V$. Also, $\Grpn$ acts on the $k$-fold tensor product $\Vk$ by the diagonal action.
The proof of the following theorem uses the basis of the centralizer algebra of the action of $\Grpn$ on $\Vk$ as given in Lemma \ref{ca}(a) (also, \cite[Lemma 2.1]{T97}).

\begin{thm}
 The vector space $\Tn$ is a subalgebra of $\C A_{k}(n)$.
\end{thm}
\begin{proof}
Let $d_{1}, d_{2}\in A_{k}(r,p,n)$. It is sufficient to assume that the bottom row of $d_{1}$ matches with the top row of $d_{2}$.
The multiplication $x_{d_{1}}x_{d_{2}}$ is given by 
\begin{equation}\label{stc}
x_{d_{1}}x_{d_{2}}=\sum\limits_{d \in A_k} c_{d} x_{d}
\end{equation} 
{\it Case (i)}: If $d_{1}, d_{2}\in \Pi_{k}(r)$, then by Remark \ref{sm}, we have  $x_{d_{1}}x_{d_{2}} \in \T_k(r,1,n) \subseteq \Tn$.

\noindent{\it Case (ii)}: One of $d_1$ or $d_2$ is in $\Theta_k(r,p,n)$. Without loss of generality, assume that  $d_{1}\in\Theta_k(r,p,n)$ and $d_{2}\in A_{k}(r,p,n)$. {\em Claim}:  $c_{d} = 0 \mbox{ for } d \notin A_{k}(r,p,n)$ in (\ref{stc}). Since $d_{1}$ has more than 
$n$ blocks, therefore using Schur-Weyl duality for partition algebra (Theorem \ref{oswd}), we get, in (\ref{stc}) 
$$ \sum_{\substack{d \in A_k, \\ |d| \leq n}} c_d \phi_k(x_d) = 0. $$
The linear independence of $\{\phi_k(x_d) \mid |d| \leq n \}$ implies that $c_d = 0$ for $d \in A_k, |d| \leq n.$ Thus, $c_d$ can be nonzero only when $|d| > n$. For such $d$, we show that either $d \in \Pi_{k}(r)$ or $d\in \Theta_{k}(r,p,n)$. 
Suppose $d\notin \Pi_{k}(r)$, then there exists $1 \leq j\leq |d| $ such that $N(B_{j})\not\equiv M(B_{j})({\mbox{mod }r)}$.

{\em Subcase (a)}: 
Suppose $d=d_{1}\circ d_{2}$.
If a block $B$ in $d_{1}$ is connected with a block $B'$ in $d_{2}$ then $N(B)\equiv M(B)({\mbox {mod }} m)$, $N(B')\equiv M(B')({\mbox {mod }} m)$ and $M(B)=N(B')$. Thus, $N(B)\equiv M(B')({\mbox {mod }} m)$ and $d \in \Theta_k(r,p,n)$. This also includes the cases when either of $B$ and $B'$ are entirely in the top or bottom row of $d_1$ and $d_2$ respectively.

{\em Subcase (b)}: Suppose that $d$ is obtained by coarsening of $d_1 \circ d_2$ as in Lemma \ref{mult}. Let the coarsening be done by connecting a block $B$ entirely in the top row of $d_{1}$ with a block $B'$ entirely in the bottom row of $d_{2}$. Then $$N(B)\equiv 0({\mbox{mod }} m) \mbox{ and } 0 \equiv M(B')(\mbox{mod } m).$$ Thus, $N(B)\equiv M(B')(\mbox{mod }m)$ and $d \in \Theta_k(r,p,n)$.

\noindent{\it Case (iii)}: One of $d_1$ and $d_2$ is in $\Pi_{k}(r)$ and the other is in $ \Lambda_{k}(r,p,n)$.  Without loss of generality, assume that $d_{1}\in \Pi_{k}(r)$ and $d_{2}\in \Lambda_{k}(r,p,n)$. If $|d_1| > n$, then we can argue similar to the case (ii) above. So, assume that $|d_{1}| \leq n$. From (\ref{stc}), we have $$0 \neq \sum\limits_{d \in A_k} c_{d}\phi_{k}(x_{d}) \in \mbox{End}_{\Grpn}(\Vk).$$  Using the basis of $\mbox{End}_{\Grpn}(\Vk)$ as in Lemma \ref{ca}(a), it follows that, for $d$ such that $|d| \leq n$, $c_{d}$ can be nonzero only when $d\in \Pi_{k}(r) \bigcup \Lambda_{k}(r,p,n)$.

If there exists $d$ in (\ref{stc}) with more than $n$ blocks such that $c_d \neq 0$, then by the arguments similar to the case (ii), we get either $d\in \Pi_{k}(r)$ or $d\in\Lambda_{k}(r,p,n)$.
\end{proof}

Define the following mutually disjoint subsets of $A_{k+\frac{1}{2}}$:
\begin{align*}
& \Pi_{k+\frac{1}{2}}(r) := \Pi_{k+1}(r) \bigcap A_{k+\frac{1}{2}}, \\
& \Lambda_{k+\frac{1}{2}}(r,p,n) := \Lambda_{k+1}(r,p,n) \bigcap A_{k+\frac{1}{2}}, \\
& \Theta_{k+\frac{1}{2}}(r,p,n) := \Theta_{k+1}(r,p,n) \bigcap A_{k+\frac{1}{2}}.
\end{align*}
Also, define $A_{k+\frac{1}{2}}(r,p,n)$, a subset of $A_{k+\frac{1}{2}}$, by setting
\begin{equation*}
A_{k+\frac{1}{2}}(r,p,n) := \Pi_{k+\frac{1}{2}}(r) \bigcup \Lambda_{k+\frac{1}{2}}(r,p,n) \bigcup \Theta_{k+\frac{1}{2}}(r,p,n).
\end{equation*}

\begin{definition}
Define $\Thn := \C\mbox{-}span\{x_d \mid d \in A_{k+\frac{1}{2}}(r,p,n)\}$, a subspace of partition algebra $\C A_{k+\frac{1}{2}}(n)$.
\end{definition}

\begin{thm}
 The vector space $\Thn$ is a subalgebra of 
 $\C A_{k+\frac{1}{2}}(n)$.
\end{thm}

\begin{proof}
Note that $\Thn = \T_{k+1}(r,p,n) \bigcap \C A_{k+\frac{1}{2}}(n)$, hence $\Thn$ is an algebra.
\end{proof}

\begin{definition}[Tanabe algebra]
We call the algebras $\Tn$ and $\Thn$ as Tanabe algebras.
\end{definition}

From \cite[Page 879]{HR05}, there is an injective algebra homomorphism
\begin{align*}
\C A_{k}(n) & \hookrightarrow \C A_{k+\frac{1}{2}}(n) \\
d & \mapsto d',
\end{align*}
where $d\in A_{k}$ and $d'\in A_{k+\frac{1}{2}}$ has same blocks as $d$ with an additional block $\{(k+1),(k+1)'\}$. It is easy to see that the corresponding element $x_{d}$ is mapped to $(x_{d'}+\sum x_{d''})$, where the sum is over all $d''\in A_{k+\frac{1}{2}} \setminus \{d'\}$ obtained by connecting a block in $d'$ with the block $\{(k+1),(k+1)'\}$. Using the description of the above map in terms of the elements
$x_{d}$, we see that the algebra $\Tn$ can be embedded inside the algebra $\Thn$.

\begin{example}\label{k2} In this example, we describe $\Tn$ for various specific values of $r,p$ and $n$ when $k = 2$. The monoid $A_2 = \{d_1, d_2, \ldots, d_{15}\}$ with the elements given as below:  
 
\begin{tikzpicture}[scale=0.9,mycirc/.style={circle,fill=black, minimum size=0.1mm, inner sep = 1.5pt}
    ]
    \node (1) at (-1,0.5) {$d_1 =$};
    \node[mycirc,label=above:{$1$}] (n1) at (0,1) {};
    \node[mycirc,label=above:{$2$}] (n2) at (1,1) {};
    \node[mycirc,label=below:{$1'$}] (n3) at (0,0) {};
    \node[mycirc,label=below:{$2'$}] (n4) at (1,0) {};
    \node (a) at (1.4,0.3) {,};

    \node (2) at (2.5, 0.5) {$d_2 =$};
    \node[mycirc,label=above:{$1$}] (n5) at (3.5,1) {};
    \node[mycirc,label=above:{$2$}] (n6) at (4.5,1) {};
    \node[mycirc,label=below:{$1'$}] (n7) at (3.5,0) {};
    \node[mycirc,label=below:{$2'$}] (n8) at (4.5,0) {};
    \draw (n5) -- (n6);
    \node (b) at (4.9,0.3) {,};

    \node (3) at (6.0, 0.5) {$d_3 =$};
    \node[mycirc,label=above:{$1$}] (n9) at (7,1) {};
    \node[mycirc,label=above:{$2$}] (n10) at (8,1) {};
    \node[mycirc,label=below:{$1'$}] (n11) at (7,0) {};
    \node[mycirc,label=below:{$2'$}] (n12) at (8,0) {};
    \draw (n9) -- (n11);
    \node (c) at (8.4,0.3) {,};

    \node (4) at (9.5, 0.5) {$d_4 =$};
    \node[mycirc,label=above:{$1$}] (n13) at (10.5,1) {};
    \node[mycirc,label=above:{$2$}] (n14) at (11.5,1) {};
    \node[mycirc,label=below:{$1'$}] (n15) at (10.5,0) {};
    \node[mycirc,label=below:{$2'$}] (n16) at (11.5,0) {};
    \draw (n13) -- (n16);
    \node (d) at (11.9,0.3) {,};
    
\end{tikzpicture}

\vspace{1cm}

\begin{tikzpicture}[scale=0.9,mycirc/.style={circle,fill=black, minimum size=0.1mm, inner sep = 1.5pt}
    ]
    \node (1) at (-1,0.5) {$d_5 =$};
    \node[mycirc,label=above:{$1$}] (n1) at (0,1) {};
    \node[mycirc,label=above:{$2$}] (n2) at (1,1) {};
    \node[mycirc,label=below:{$1'$}] (n3) at (0,0) {};
    \node[mycirc,label=below:{$2'$}] (n4) at (1,0) {};
    \node (a) at (1.4,0.3) {,};
    \draw (n2) -- (n3);

    \node (2) at (2.5, 0.5) {$d_6 =$};
    \node[mycirc,label=above:{$1$}] (n5) at (3.5,1) {};
    \node[mycirc,label=above:{$2$}] (n6) at (4.5,1) {};
    \node[mycirc,label=below:{$1'$}] (n7) at (3.5,0) {};
    \node[mycirc,label=below:{$2'$}] (n8) at (4.5,0) {};
    \draw (n6) -- (n8);
    \node (b) at (4.9,0.3) {,};

    \node (3) at (6.0, 0.5) {$d_7 =$};
    \node[mycirc,label=above:{$1$}] (n9) at (7,1) {};
    \node[mycirc,label=above:{$2$}] (n10) at (8,1) {};
    \node[mycirc,label=below:{$1'$}] (n11) at (7,0) {};
    \node[mycirc,label=below:{$2'$}] (n12) at (8,0) {};
    \draw (n11) -- (n12);
    \node (c) at (8.4,0.3) {,};

    \node (4) at (9.5, 0.5) {$d_8 =$};
    \node[mycirc,label=above:{$1$}] (n13) at (10.5,1) {};
    \node[mycirc,label=above:{$2$}] (n14) at (11.5,1) {};
    \node[mycirc,label=below:{$1'$}] (n15) at (10.5,0) {};
    \node[mycirc,label=below:{$2'$}] (n16) at (11.5,0) {};
    \draw (n13) -- (n14);
    \draw (n15) -- (n16);
    \node (d) at (11.9,0.3) {,};
    
\end{tikzpicture}

\vspace{1cm}

\begin{tikzpicture}[scale=0.9,mycirc/.style={circle,fill=black, minimum size=0.1mm, inner sep = 1.5pt}
    ]
    \node (1) at (-1,0.5) {$d_9 =$};
    \node[mycirc,label=above:{$1$}] (n1) at (0,1) {};
    \node[mycirc,label=above:{$2$}] (n2) at (1,1) {};
    \node[mycirc,label=below:{$1'$}] (n3) at (0,0) {};
    \node[mycirc,label=below:{$2'$}] (n4) at (1,0) {};
    \node (a) at (1.4,0.3) {,};
    \draw (n1) -- (n4);
    \draw (n2) -- (n3);

    \node (2) at (2.5, 0.5) {$d_{10} =$};
    \node[mycirc,label=above:{$1$}] (n5) at (3.5,1) {};
    \node[mycirc,label=above:{$2$}] (n6) at (4.5,1) {};
    \node[mycirc,label=below:{$1'$}] (n7) at (3.5,0) {};
    \node[mycirc,label=below:{$2'$}] (n8) at (4.5,0) {};
    \draw (n6) -- (n8);
    \draw (n5) -- (n7);
    \node (b) at (4.9,0.3) {,};

    \node (3) at (6.0, 0.5) {$d_{11} =$};
    \node[mycirc,label=above:{$1$}] (n9) at (7,1) {};
    \node[mycirc,label=above:{$2$}] (n10) at (8,1) {};
    \node[mycirc,label=below:{$1'$}] (n11) at (7,0) {};
    \node[mycirc,label=below:{$2'$}] (n12) at (8,0) {};
    \draw (n11) -- (n12);
    \draw (n11) -- (n9);
    \node (c) at (8.4,0.3) {,};

    \node (4) at (9.5, 0.5) {$d_{12} =$};
    \node[mycirc,label=above:{$1$}] (n13) at (10.5,1) {};
    \node[mycirc,label=above:{$2$}] (n14) at (11.5,1) {};
    \node[mycirc,label=below:{$1'$}] (n15) at (10.5,0) {};
    \node[mycirc,label=below:{$2'$}] (n16) at (11.5,0) {};
    \draw (n13) -- (n14);
    \draw (n15) -- (n13);
    \node (d) at (11.9,0.3) {,};
    
\end{tikzpicture}

\vspace{0.5cm}

\begin{tikzpicture}[scale=0.9,mycirc/.style={circle,fill=black, minimum size=0.1mm, inner sep = 1.5pt}
    ]
    \node (1) at (-1,0.5) {$d_{13} =$};
    \node[mycirc,label=above:{$1$}] (n1) at (0,1) {};
    \node[mycirc,label=above:{$2$}] (n2) at (1,1) {};
    \node[mycirc,label=below:{$1'$}] (n3) at (0,0) {};
    \node[mycirc,label=below:{$2'$}] (n4) at (1,0) {};
    \node (a) at (1.4,0.3) {,};
    \draw (n1) -- (n2);
    \draw (n2) -- (n4);

    \node (2) at (2.5, 0.5) {$d_{14} =$};
    \node[mycirc,label=above:{$1$}] (n5) at (3.5,1) {};
    \node[mycirc,label=above:{$2$}] (n6) at (4.5,1) {};
    \node[mycirc,label=below:{$1'$}] (n7) at (3.5,0) {};
    \node[mycirc,label=below:{$2'$}] (n8) at (4.5,0) {};
    \draw (n6) -- (n8);
    \draw (n8) -- (n7);
    \node (b) at (4.9,0.3) {,};

    \node (3) at (6.0, 0.5) {$d_{15} =$};
    \node[mycirc,label=above:{$1$}] (n9) at (7,1) {};
    \node[mycirc,label=above:{$2$}] (n10) at (8,1) {};
    \node[mycirc,label=below:{$1'$}] (n11) at (7,0) {};
    \node[mycirc,label=below:{$2'$}] (n12) at (8,0) {};
    \draw (n11) -- (n12);
    \draw (n10) -- (n9);
    \draw (n9) -- (n11);
    \draw (n10) -- (n12);
    \node (c) at (8.4,0.3) {.};

\end{tikzpicture}
 
\begin{enumerate}[(i)]
 \item For $r = 2$, we have $\Pi_2(2) = \{d_8, d_9, d_{10}, d_{15}\}.$ 
     \begin{enumerate}[(a)]
       \item For $p=2, n=2$, the sets $$\Lambda_2(2,2,2) =\{d_{11}, d_{12}, d_{13}, d_{14}\}$$ $$ \mbox{ and }\Theta_2(2,2,2) = \{d_1, d_2, d_3, d_4, d_5, d_6, d_7 \}.$$ Thus, $\T_2(2,2,2)$ is the partition algebra $\C A_2(2)$. 
       \item For $p=2, n=3$, $\Lambda_2(2,2,3)$ is an empty set and $\Theta_2(2,2,3) =\{d_1\}.$
       \item For $p=2, n=4$, we have $\Lambda_2(2,2,4) =\{d_1\}$ and $\Theta_2(2,2,4)$ is an empty set.
     \end{enumerate}
 \item For $r \neq 1,2$, we have $\Pi_2(r) = \{d_9, d_{10}, d_{15}\}.$  For $r = 3$,   $\Lambda_2(r,p,n)$ is nonempty if and only if $(r,p,n) = (3,3,3)$; and $\Lambda_2(3,3,3) = \{d_2, d_7\}.$ 
 For $r = 4$,   $\Lambda_2(r,p,n)$ is nonempty if and only if $(r,p,n) = (4,2,2)$ or $(4,4,2)$; and $\Lambda_2(4,2,2) = \Lambda_2(4,4,2) = \{d_8\}.$ For $r > 4$, $\Lambda_2(r,p,n)$ is empty for all values of $p$ and $n$. In general, for $r > 2k$, $\Lambda_k(r,p,n)$ is empty for all values of $p$ and $n$.
 \end{enumerate}
\end{example}

\begin{remark}\label{222k}
 For $n \geq 2k$, $\Theta_k(r,p,n)$ is an empty set. For $n \geq 2k$, the set $\Lambda_k(r,p,n)$ is nonempty if and only if $(r,p,n) = (2,2,2k)$; $\Lambda_k(2,2,2k) = \{d\}$, where $d$ is a partition diagram with $2k$ blocks, i.e., each block consists of a single vertex. Using the multiplication rule in Lemma \ref{mult}, it is easy to check that the corresponding $x_d$ is a central element of Tanabe algebra $\T_k(2,2,2k)$.
\end{remark}

\section{Complex reflection groups}\label{crg}
For an $n$-dimensional complex vector space $W$, a linear isomorphism of $W$ of finite order is said to be a {\it reflection} in $W$ if it has exactly $(n-1)$ eigenvalues equal to $1$. A {\it complex reflection group} $R$ in $W$ is a group generated by reflections in $W$. The space $W$ is called the {\it reflection representation} of $R$. We say $R$ is irreducible if the $R$-invariant complement of the subspace $W^R$, which is fixed pointwise by $R$, in $W$ is irreducible. If there exists a direct sum 
$W=W_{1}\oplus W_{2}\oplus\cdots\oplus W_{t}$, where $W_{i}$ is non-trivial proper subspace of $W$ for each $1 \leq i \leq n$, such that $W_1, W_2, \ldots, W_t$ are permuted among themselves under the action of $R$, then we say that $R$ is {\it imprimitive}. By Shephard-Todd classification, the groups $\Grpn$, for $n > 1$, are the only finite irreducible imprimitive complex reflection groups \cite[Section 2]{ST54}. 

Suppose that $G := \Z/r\Z$ is the cyclic group of order $r$ with $\zeta$, a primitive $r$-th root of unity. Define $\Drpn$ to be the subgroup of $GL_n(\C)$ consisting of diagonal matrices as:
\begin{equation*}
 \Drpn := \left\{ \begin{bmatrix}
              \zeta^{i_1} & 0 & \ldots & 0 \\
              0 & \zeta^{i_2}  & \ldots & 0 \\
               \vdots & \vdots & \ddots & \vdots  \\
               0 & 0 & \ldots & \zeta^{i_n}\\
             \end{bmatrix} 
             \mid i_1 + i_2 + \cdots + i_n \equiv 0 (\mbox{mod }p)\right\}.
\end{equation*}
Let $S_n$ be the group of $n \times n$ permutation matrices. Define $\Grpn$ to be the subgroup of $GL(n, \C)$ generated by $\Drpn$ and $S_n$. Since $S_n$ normalizes $\Drpn$ and $\Drpn \bigcap S_n = \{I_n\}$, where $I_n$ is the identity matrix, so the group $\Grpn$ is a semidirect product: $$\Grpn = \Drpn \rtimes S_n.$$ Thus, as a subgroup of $GL_n(\C)$, the group $\Grpn$ consists of generalized permutation matrices with nonzero entries being $r$-th roots of unity and the $m$-th power of the product of nonzero entries is one. Also, the elements of $\Grpn$ can be written as $(n+1)$-tuple:
$$\Grpn = \{(\zeta^{i_1}, \zeta^{i_2}, \ldots, \zeta^{i_n}, \pi) \mid  i_1 + i_2 + \cdots + i_n \equiv 0 (\mbox{mod }p), \pi \in S_n\}.$$ 

The particular case when $p=1$ is the group $G(r,1,n)$, the wreath product of the group $G$ by the symmetric group $S_n$, of order $r^nn!$. Taking the exact sequence 
\begin{align*}
 1 \longrightarrow \Grpn  \longrightarrow  G(r,1,n) & \longrightarrow \Z/p\Z \longrightarrow 1 \\
  (\zeta^{i_1}, \zeta^{i_2}, \ldots, \zeta^{i_n}, \pi) & \mapsto \zeta^{i_1 + i_2 + \cdots + i_n},
\end{align*}
we see that $\Grpn$ is a normal subgroup of the group $G(r,1,n)$ of index $p$. So, the order of the group $\Grpn$ is $(r^nn!)/p$. 

Some families of groups which are special cases of $\Grpn$ are:
\begin{enumerate}[(a)]
 \item cyclic group of order $r$, i.e., $ \Z/r\Z = G(r,1,1)$,
 \item dihedral group of order $2r$, $D_{2r} = G(r,r,2)$,
 \item symmetric group on $n$ symbols, $S_n = G(1,1,n)$,
 \item Weyl group of type $B_n$ (also called hyperoctahedral group) is $G(2,1,n)$,
 \item Weyl group of type $D_n$ is $G(2,2,n)$. 
\end{enumerate}

Let $G(n)$ be an isomorphic copy of $G$ in $G(r,1,n)$ defined as $$G(n) := \{(1, \ldots, 1,g_n, (1))\mid g_n \in G\}.$$ Assume that $S_{n-1}$ is the subgroup of $S_n$ consisting of elements fixing $n$. The groups $G(r,1,n-1) \times G(n)$ and $\Grpn$ are subgroups of $G(r,1,n)$. Let $\Lrpn$ be the subgroup of $\Grpn$ defined as:
\begin{align*}
 \Lrpn := 
 & (G(r,1,n-1) \times G(n))\bigcap \Grpn \\
 = & ((G^{n-1} \rtimes S_{n-1}) \times G(n))\bigcap (\Drpn \rtimes S_n) \\ 
 = & (G^n \rtimes S_{n-1})\bigcap (\Drpn \rtimes S_n) \\ 
 = & \Drpn \rtimes S_{n-1}.
\end{align*}
As a subgroup of $GL_n(\C)$, the group $\Lrpn$ consists of those elements in $\Grpn$ such that the $(n,n)$-th entry is nonzero. 
For $p=1$, we have
\begin{align*}
 L(r,1,n) &= G^n \rtimes S_{n-1} \\ 
 & = \big(G^{n-1} \rtimes S_{n-1}\big) \times G(n) \\  
 &= G(r,1,n-1) \times G(n).
\end{align*}
The order of $L(r,1,n)$ is $r^n(n-1)!$. Taking the exact sequence 
\begin{align*}
 1 \longrightarrow \Lrpn  \longrightarrow  L(r,1,n) & \longrightarrow \Z/p\Z \rightarrow 1 \\
  (g_1, g_2, \ldots, g_n, \pi) & \mapsto g_1g_2 \cdots g_n,
\end{align*}
we see that $\Lrpn$ is a normal subgroup of the group $L(r,1,n)$ of index $p$. 
Thus, the order of $\Lrpn$ is $(r^n(n-1)!)/p$.  

Given an $r$-tuple of Young diagrams $(\lambda_1, \lambda_2, \ldots, \lambda_r) \in \Y(r,n-1)$, choose one $i\in \{1,2, \ldots, r\}$, take $\lambda_i$ ($\lambda_i$ may be empty also), color it by $n$ and denote by $\lambda_i^n$. We note that $\lambda_i^n$ denotes the same Young diagram $\lambda_i$, but it has the color $n$. The {\it $(n,i)$-colored $r$-tuple of Young diagrams}, denoted by $\La^{(n,i)} := (\lambda_1, \lambda_2, \ldots,\lambda_{i-1}, \lambda_i^n, \lambda_{i+1}, \ldots, \lambda_r)$, consists of the $r$-tuple $(\lambda_1, \lambda_2, \ldots, \lambda_r) \in \Y(r,n-1)$ with $i$-th component $\La_i$ colored by $n$. Let $\Y^n(r,n-1)$ denote the set of all $(n,i)$-colored $r$-tuples of Young diagrams with total $n-1$ boxes for $i = 1,2, \ldots, r$.

\begin{lemma}
The irreducible $L(r,1,n)$-modules are indexed by the elements of $\Y^n(r,n-1)$.  
\end{lemma}

\begin{proof}
 The irreducible representations of $G$ are $\sigma_1, \sigma_2, \ldots, \sigma_r$ (defined in Section \ref{ov}). Suppose that $V^{\lambda}$ is the irreducible representation of $G(r,1,n-1)$ corresponding to $\lambda \in \Y(r,n-1)$. Then,
$$\{V^{\La} \otimes \sigma_i \mid \La \in \Y(r,n-1) , i = 1, \ldots, r \}$$
is the set of irreducible representations of $L(r,1,n)$ which is indexed by the elements of the set $\Y^n(r,n-1)$.
\end{proof}

We describe the branching rule from $G(r,1,n)$ to $L(r,1,n)$.  For $\mu = (\mu_1,\mu_2, \ldots, \mu_r) \in \Y(r,n)$ with $\mu_i \neq \emptyset$, let $\mu \downarrow i$ denote the set of elements $\nu^{(n,i)} \in \Y^n(r,n-1)$ such that $\nu$ is obtained from $\mu$ by removing the box at an inner corner of $\mu_i$ and then coloring the $i$-th component of $\nu$ by $n$ to obtain $(n,i)$-colored $r$-tuple $\nu^{(n,i)}$. Assume that $V^{\mu}$ and $V^{\nu^{(n,i)}}$ are the irreducible $G(r,1,n)$-module and $L(r,1,n)$-module corresponding to $\mu \in \Y(r,n)$ and $\nu^{(n,i)} \in \Y^n(r,n-1)$ respectively.

\begin{thm}[Branching rule from $G(r,1,n)$ to $L(r,1,n)$]\label{brgl1}
We have  
\begin{equation}
 \mbox{\em Res}_{L(r,1,n)}^{G(r,1,n)}(V^{\mu}) = \bigoplus_{i=1}^{r}\left(\bigoplus_{\nu \in \mu \downarrow i} V^{\nu^{(n,i)}}\right).
\end{equation}
\end{thm}
{\bf Remark}: We take equality in place of isomorphism because the restriction rule is multiplicity free which makes the decomposition canonical and we identify $V^{\nu^{(n,i)}}$ with the corresponding $L(r,1,n)$-submodule of $V^{\mu}$.

\begin{proof}
 Since $\nu_j = \mu_j$ for $j \neq i$ and $|\nu_i| = |\mu_i| - 1$, therefore given a $GZ$-subspace of $V^{\mu}$, there exists a $GZ$-subspace of $V^{\nu^{(n,i)}} = V^{\nu} \otimes \sigma_i$ with the same label. Also, for $1 \leq i \leq n-1$, the action of $X_i \in GZ_{n-1,n} \subseteq GZ_{n,n}$ on $GZ$-subspace of $V^{\nu^{(n,i)}}$ is same as its action on GZ-subspace of $V^{\mu}$. A GZ-subspace is uniquely determined by its weight and a GZ-subspace uniquely determines the parametrization of irreducible representation. Thus, $V^{\nu^{(n,i)}}$ appears in the restriction of $V^{\mu}$ as a $L(r,1,n)$-module with multiplicity one since the restriction from $G(r,1,n)$ to $L(r,1,n)$ is multiplicity free (follows from chain (\ref{mfic}) since $H_{m-1,n} = L(r,1,n)$).
\end{proof}

The next step is the parametrization of the irreducible representations of $\Grpn$ and $\Lrpn$ using Clifford theory. Consider the one-dimensional representation $$\delta_0 : G(r,1,n) \longrightarrow \C^*$$  $$ \delta_0(g_1, g_2, \ldots, g_n, \sigma) = g_1g_2\ldots g_n.$$ As a $G(r,1,n)$-module, $\delta_0$ is parametrized by $(\emptyset, (n), \emptyset, \ldots, \emptyset)$ and $\Grpn \subseteq \mbox{Ker} (\delta_0^m)$.  We use the same notation $\delta_0$ to denote the restriction of $\delta_0$ to $L(r,1,n)$. It will be clear from the context whether we consider $\delta_0$ as a $G(r,1,n)$-module or as a $L(r,1,n)$-module.  As a $L(r,1,n)$-module, $\delta_0$ is parametrized by the $(n,2)$-colored $r$-tuple $(\emptyset, (n-1)^n, \emptyset, \ldots, \emptyset)$ and $\Lrpn \subseteq \mbox{Ker} (\delta_0^m)$.  The cyclic group $C$ generated by $\delta_0^m$ is of order $p$. Thus $$ C \cong G(r,1,n)/\Grpn \cong L(r,1,n)/\Lrpn.$$ 

Define the shift map on $\Y(r,n)$ as $\sh : \Y(r,n) \longrightarrow \Y(r,n)$ by $$  (\La_1, \La_2, \ldots, \La_r) \mapsto (\La_r, \La_1, \La_2, \ldots, \La_{r-1}).$$
Using the same notation, the shift map on $\Y^n(r,n-1)$ is defined as 
\begin{align*}
\sh : \Y^n(r,n-1) & \longrightarrow \Y^n(r,n-1) \mbox{ by } \\ 
(\mu_1, \mu_2, \ldots,\mu_i^n,\ldots,\mu_r) & \mapsto (\mu_r, \mu_1, \mu_2, \ldots,\mu_{i-1}, \mu_i^n,\ldots,\mu_{r-1}),
\end{align*}
where the $r$-tuples on L.H.S. and R.H.S. are $(n,i)$-colored and $(n,i+1)$-colored respectively.

Suppose that $V^{\La}$ and $V^{\mu^{(n,i)}}$ denote the irreducible representations of $G(r,1,n)$ and $L(r,1,n)$ parametrized by the $r$-tuple $\La = (\La_1, \La_2, \ldots, \La_r) \in \Y(r,n)$ and $(n,i)$-colored $r$-tuple $\mu^{(n,i)} = (\mu_1, \mu_2, \ldots,\mu_i^n,\ldots,\mu_r) \in \Y^n(r,n-1)$ for some $i \in \{1,2,\ldots,r\}$ respectively.

The following lemma is proved using Okounkov-Vershik approach. Part (a) is  \cite[Theorem 24]{MRW04} and it was proved there using $*$-rim hook tableaux.   

\begin{lemma}\label{shl}
For $\La \in \Y(r,n)$ and $\mu^{(n,i)} \in \Y^n(r,n-1)$, the following are true:
\begin{enumerate}[(a)]
 \item As $G(r,1,n)$-modules, $$ \delta_0 \otimes V^{\La} \; \cong \; V^{\emph{sh}(\La)}.$$
 \item As $L(r,1,n)$-modules, $$ \delta_0 \otimes V^{\mu^{(n,i)}} \; \cong \; V^{\emph{sh}(\mu^{(n,i)})}.$$
\end{enumerate}
\end{lemma}

\begin{proof}
\begin{enumerate}[(a)]
 \item A $GZ$-subspace of an irreducible representation of $G(r,1,n)$ is uniquely determined by its weight. Also, a $GZ$-subspace uniquely determines the $r$-tuple of Young diagrams in $\Y(r,n)$ which parametrize the irreducible representation of which it is a $GZ$-subspace.

For $\La = (\La_1, \La_2, \ldots, \La_r) \in \Y(r,n)$ with $y_i :=  |\La_i| $, let $R$ be the standard $r$-tuple of Young tableaux written in row major order. The $GZ$-subspace of type $V_R$ is isomorphic to 
$$\underbrace{\sigma_1 \otimes \cdots \otimes \sigma_1}_{y_1\mbox{-fold}}\otimes \underbrace{\sigma_2 \otimes \cdots \otimes \sigma_2}_{y_2\mbox{-fold}} \otimes \cdots \otimes \underbrace{\sigma_r \otimes \cdots \otimes \sigma_r}_{y_r\mbox{-fold}}$$ as a $G^n$-module. For $i =1,2, \ldots, n$ and $GZ$-basis element $$v_R = \underbrace{v_1 \otimes \cdots \otimes v_1}_{y_1\mbox{-fold}}\otimes \underbrace{v_2 \otimes \cdots \otimes v_2}_{y_2\mbox{-fold}} \otimes \cdots \otimes \underbrace{v_r \otimes \cdots \otimes v_r}_{y_r\mbox{-fold}}$$ we have 
$ X_i(v_R) = r c(b_R(i))(v_R).$

The $GZ$-subspace of $\delta_0$ is given by $n$-fold $\sigma_2 \otimes \cdots \otimes \sigma_2$ with $GZ$-basis element given by $n$-fold $v_2 \otimes \cdots \otimes v_2.$ Thus, the $GZ$-subspace of $\delta_0 \otimes V^{\La}$ is 
\begin{align*}
 & (\underbrace{\sigma_1 \otimes \cdots \otimes \sigma_1}_{y_1\mbox{-fold}}\otimes \underbrace{\sigma_2 \otimes \cdots \otimes \sigma_2}_{y_2\mbox{-fold}} \otimes \cdots \otimes \underbrace{\sigma_r \otimes \cdots \otimes \sigma_r}_{y_r\mbox{-fold}}) \otimes (\underbrace{\sigma_2 \otimes \cdots \otimes \sigma_2}_{n \mbox{-fold}}) \\
 & \cong \; \underbrace{\sigma_2 \otimes \cdots \otimes \sigma_2}_{y_1\mbox{-fold}}\otimes \underbrace{\sigma_3 \otimes \cdots \otimes \sigma_3}_{y_2\mbox{-fold}} \otimes \cdots \otimes \underbrace{\sigma_1\otimes \cdots \otimes \sigma_1}_{y_r\mbox{-fold}} \\
 & \cong \; V_{\sh(R)},
\end{align*}
isomorphic as $G^n$-module, with basis element $v'$ being
$$
\underbrace{(v_1 \otimes v_2) \cdots \otimes (v_1 \otimes v_2)}_{y_1\mbox{-fold}}\otimes \underbrace{(v_2 \otimes v_2) \cdots \otimes (v_2 \otimes v_2)}_{y_2\mbox{-fold}} \otimes \cdots \otimes \underbrace{(v_r \otimes v_2) \cdots \otimes (v_r \otimes v_2)}_{y_r\mbox{-fold}}). 
$$
Also, for $1 \leq i \neq j \leq n$, we have $$\g_i^l \g_j^{-l} s_{ij}(v_2 \otimes \cdots \otimes v_2) = v_2 \otimes \cdots \otimes v_2 $$ 
So, for $ 1 \leq i \leq n$,
\begin{align*}
X_i (v') & = (v_2 \otimes \cdots \otimes v_2) \otimes X_i(v_1 \otimes \cdots \otimes v_1\otimes v_2 \otimes \cdots \otimes v_2 \otimes \cdots \otimes v_r \otimes \cdots \otimes v_r) \\
& = r c(b_R(i))(v') \\
& = r c(b_{\sh(R)}(i))(v_{\sh(R)}) = X_i(v_{\sh(R)})
\end{align*}
which implies that $v' = v_{\sh(R)}.$

We have shown that $V_{\sh(R)}$ is a $GZ$-subspace of $\delta_0 \otimes V^{\La}$. Thus, $V^{\sh(\La)}$, corresponding to $r$-tuple $\sh(\La)$, is a $G(r,1,n)$-submodule of $\delta_0 \otimes V^{\La}$. The irreducibility of $\delta_0 \otimes V^{\La}$ implies the result. 

\item This part can be proved by arguments similar to those in part (a). To be able to do so, we note that a $GZ$-subspace of an irreducible representation of $L(r,1,n)$ is uniquely determined by its weight, i.e., its label and the action of Jucys-Murphy elements $X_1, X_2, \ldots, X_{n-1}$ on it. \qedhere
\end{enumerate}
\end{proof}

Lemma \ref{shl} implies Corollaries \ref{shl2} and \ref{shl1}.
Part (a) of Corollary \ref{shl2} is \cite[Corollary 25]{MRW04} and is also stated as Theorem 2.1 in \cite{Ore07}. 

\begin{cor}\label{shl2}
For $\La \in \Y(r,n)$ and $\mu^{(n,i)} \in \Y^n(r,n-1)$, the following are true for $t \in \Z:$
\begin{enumerate}[(a)]
\item As a $G(r,1,n)$-module, $$ \delta_0^t \otimes V^{\La} \cong V^{\emph{sh}^t(\La)}.$$
\item As a $L(r,1,n)$-module, $$ \delta_0^t \otimes V^{\mu^{(n,i)}} \cong V^{\emph{sh}^t(\mu^{(n,i)})}.$$
\end{enumerate}
\end{cor}

\begin{cor}\label{shl1}
For $t \in \Z:$
\begin{enumerate}[(a)]
\item As a $G(r,1,n)$-module, $\delta_0^t$ is parametrized by $(\emptyset, \ldots,\emptyset, (n), \emptyset, \ldots, \emptyset) \in \Y(r,n)$, where $(n)$ occurs at $(t+1)(\mbox{mod }r)$-th component. 
\item As a $L(r,1,n)$-module, $\delta_0^t$ is parametrized by the $(n,(t+1)(\mbox{mod }r))$-colored $r$-tuple $(\emptyset, \ldots,\emptyset, (n-1)^n, \emptyset, \ldots, \emptyset) \in \Y^n(r,n-1)$, where $(n-1)$ also occurs at $(t+1)(\mbox{mod }r)$-th component.
\end{enumerate}
\end{cor}

We define a combinatorial object $(m,p)$-necklace as in \cite[p. 174]{HR98} which will be useful in parametrization of irreducible $\Grpn$-modules and $\Lrpn$-modules.

Let $\lambda = (\lambda_1, \lambda_2, \ldots, \lambda_r) \in \Y(r,n)$. For each $i$ such that $1 \leq i \leq m$, consider the $p$-tuple $$\tla_{(i)} := (\lambda_{i}, \lambda_{m+i}, \lambda_{2m+i}, \ldots, \lambda_{(p-1)m+i}).$$ Depict $\tla_{(i)}$ as a $p$-necklace in the following way: the circular necklace, with centre on the $x$-axis, has $p$ nodes  and lies in a vertical $xy$-plane with the first node $\lambda_{i}$ placed at the point, where tangent to the necklace in $(y > 0)$-half plane is parallel to the $x$-axis. The placement of nodes is done in clockwise direction with the $j$-th node being $\lambda_{(j-1)m+i}$ and placed at a clockwise angle of $2\pi/(j-1)$ with $y$-axis for $j =2,\ldots,p$. A {\it $(m,p)$-necklace} of total $n$ boxes obtained from $\La \in \Y(r,n)$, denoted by $\tla$, is a $m$-tuple  $$ \tla = (\tla_{(1)}, \tla_{(2)}, \ldots, \tla_{(m)}),$$ where $\tla_{(i)}$ is a $p$-necklace for each $1 \leq i \leq m$. For $ 1 \leq j \leq p$ and $1 \leq i \leq m$, let $\tla_{(i,j)}$ denote the $j$-th node in $\tla_{(i)}$, i.e., $\tla_{(i,j)} = \La_{(j-1)m+i}$. Thus, we have 
\begin{equation*}
 \sum\limits_{i=1}^{m}\sum\limits_{j=1}^p\tla_{(i,j)} = n.
\end{equation*}
Two $(m,p)$-necklaces, $\tla$ and $\tm$, both of total boxes $n$, are said to be {\em equivalent} if for some integer $t$, $\tla_{(i,j)} = \tm_{(i,(j+t)(mod~p))}$ for all $1 \leq j \leq p$ and $1 \leq i \leq m$. Let $\Y(m,p,n)$ denote the set of inequivalent $(m,p)$-necklaces of total $n$ boxes.

Note that for any element $\mu \in [\La] $, the stabilizer subgroup $C_\mu = C_\La$. So, the stabilizer subgroup of a representative of $[\La]$ can be written as $C_\La$ while considering $(m,p)$-necklace $\tla$.

\begin{example}\label{exn}
An example of a $(3,4)$-necklace of total $~30$ boxes obtained from $$\La = ((2,1), (3,2), (2,1,1), (1), (1,1), (1,1), (1,1,1), (2,1), (1), (2), (2), (2) ) \mbox{ is}:$$

\noindent 
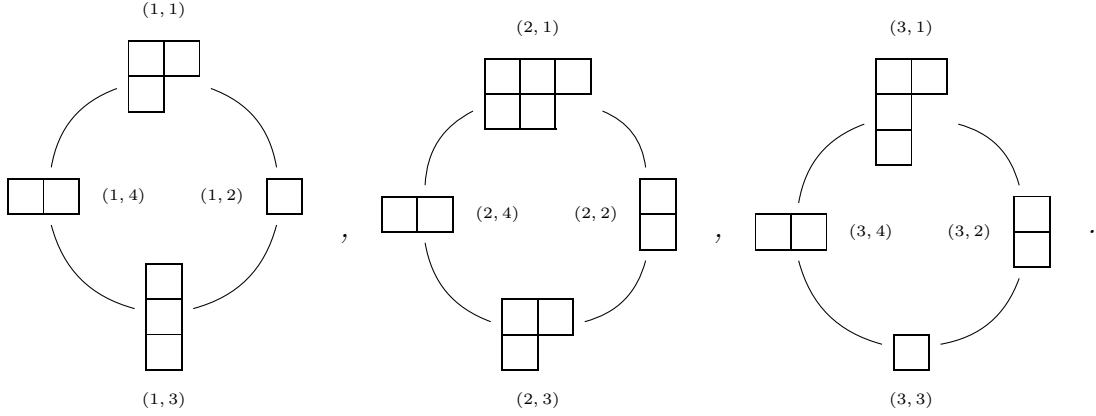
\begin{tikzpicture}[scale=0.8,mycirc/.style={circle,fill=black, minimum size=0.1mm, inner sep = 1.5pt}]

 \node[label=above:{\tiny{$(1,1)$}}] (1) at (0,2) {$\yng(2,1)$};
 \node[label=left:{\tiny{$(1,2)$}}] (2) at (2,0) {$\yng(1)$};
 \node[label=below:{\tiny{$(1,3)$}}] (3) at (0,-2) {$\yng(1,1,1)$};
 \node[label=right:{\tiny{$(1,4)$}}] (4) at (-2,0) {$\yng(2)$};
  \draw  (1) to[bend left=30] (2);
  \draw  (2) to[bend left=30] (3);
  \draw  (3) to[bend left=30] (4);
  \draw   (4) to[bend left=30] (1);
    \node(a) at (3,-0.7) {,};
 \end{tikzpicture}
 \begin{tikzpicture}[scale=0.8,mycirc/.style={circle,fill=black, minimum size=0.1mm, inner sep = 1.5pt}]
 \node[label=above:{\tiny{$(2,1)$}}] (1) at (0,2) {$\yng(3,2)$};
 \node[label=left:{\tiny{$(2,2)$}}] (2) at (2,0) {$\yng(1,1)$};
 \node[label=below:{\tiny{$(2,3)$}}] (3) at (0,-2) {$\yng(2,1)$};
 \node[label=right:{\tiny{$(2,4)$}}] (4) at (-2,0) {$\yng(2)$};
  \draw  (1) to[bend left=30] (2);
  \draw  (2) to[bend left=30] (3);
  \draw  (3) to[bend left=30] (4);
  \draw   (4) to[bend left=30] (1);
  \node(a) at (3,-0.4) {,};
 \end{tikzpicture}
 \begin{tikzpicture}[scale=0.8,mycirc/.style={circle,fill=black, minimum size=0.1mm, inner sep = 1.5pt}]
 \node[label=above:{\tiny{$(3,1)$}}] (1) at (0,2) {$\yng(2,1,1)$};
 \node[label=left:{\tiny{$(3,2)$}}] (2) at (2,0) {$\yng(1,1)$};
 \node[label=below:{\tiny{$(3,3)$}}] (3) at (0,-2) {$\yng(1)$};
 \node[label=right:{\tiny{$(3,4)$}}] (4) at (-2,0) {$\yng(2)$};
  \draw  (1) to[bend left=30] (2);
  \draw  (2) to[bend left=30] (3);
  \draw  (3) to[bend left=30] (4);
  \draw   (4) to[bend left=30] (1);
  \node(a) at (3,0) {.};
  \end{tikzpicture}
  \end{example}

Theorem \ref{irrgp} and its proof follows the expositions in \cite{S89, MY98, BB07}.

\begin{thm}\label{irrgp}
 The irreducible $\Grpn$-modules are parametrized by the ordered pairs $(\tla, \delta)$, where $\tla \in \Y(m,p,n)$ and $\delta \in C_{\lambda}$. Given $\La \in \Y(r,n)$, the restriction of the corresponding $G(r,1,n)$-module $V^{\La}$  to $\Grpn$ has multiplicity free decomposition given as:
 $$ \mbox{\em Res}_{G(r,p,n)}^{G(r,1,n)}(V^{\La}) \; = \; \bigoplus_{\delta \in C_{\La}} V^{(\tla,\delta)}.$$ Also, for $\mu \in [\La]$, $$\mbox{\em Res}_{G(r,p,n)}^{G(r,1,n)}(V^{\La}) \; \cong \; \mbox{\em Res}_{G(r,p,n)}^{G(r,1,n)}(V^{\mu}).$$
\end{thm}

\begin{proof}
The group $C = \langle \delta_0^m \rangle$ acts on the set of irreducible $G(r,1,n)$-modules. For $\La \in \Y(r,n)$, suppose that $[\La]$ denotes the elements in $\Y(r,n)$ which parametrize the irreducible $G(r,1,n)$-modules in the orbit of $V^{\La}$. Using Corollary \ref{shl2}(a), we have 
\begin{align*}
 [\La] & = \{{\nu} \mid \nu = \sh^{im}(\La) \mbox{ for some } i=0,1,\ldots,p-1 \}.
 \end{align*}
 Let the order of the orbit $[\La]$ be $b(\La)$. Then, the order of the stabilizer subgroup $C_{\La}$ is $u(\La) := \frac{p}{b(\La)}$. Also, $C_{\La}$ is generated by $\delta_0^{b(\La)m}$. The result follows from Theorem \ref{cd}.
\end{proof}

Given $\tm \in \Y(m,p,n-1)$, the {\em $(n,i,j)$-colored $(m,p)$-necklace}, denoted by $\tm^{(n,i,j)}$, is obtained by coloring $\tm_{(i,j)}$ by $n$, for $1 \leq i \leq m$ and $ 1 \leq j \leq p$. The colored $(m,p)$-necklaces, $\tm^{(n,i,j)}$ and $\tilde{\nu}^{(n,s,t)}$, are {\em equivalent} if and only if 
\begin{enumerate}[(i)]
\item $i=s$, and $j = (t+l) (\mbox{mod }p)$ for some $l \in \Z$, and
\item the corresponding $\tm$ and $\tilde\nu$ are equivalent as $(m,p)$-necklaces using the same $l$ as in (i), i.e., $\tm_{(a,b)} = \tilde\nu_{(a,(b+l)(mod~p))}$ for all $1 \leq a \leq m$ and $1 \leq b \leq p$.
\end{enumerate}
Let $\Y^n(m,p,n-1)$ be the set of inequivalent $(n,i,j)$-colored $(m,p)$-necklaces of total $n-1$ boxes for all $1 \leq i \leq m$, $1 \leq j \leq p$.

\begin{example}
Corresponding to the example \ref{exn}, the following is a colored $(3,4)$-necklace where we take $(i,j) = (2,3)$:

 \noindent 
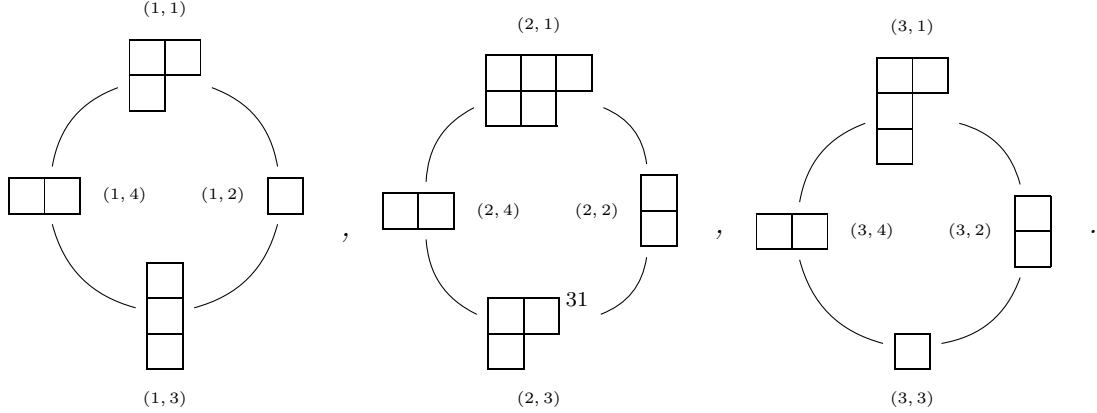
\begin{tikzpicture}[scale=0.8,mycirc/.style={circle,fill=black, minimum size=0.1mm, inner sep = 1.5pt}]
 \node[label=above:{\tiny{$(1,1)$}}] (1) at (0,2) {$\yng(2,1)$};
 \node[label=left:{\tiny{$(1,2)$}}] (2) at (2,0) {$\yng(1)$};
 \node[label=below:{\tiny{$(1,3)$}}] (3) at (0,-2) {$\yng(1,1,1)$};
 \node[label=right:{\tiny{$(1,4)$}}] (4) at (-2,0) {$\yng(2)$};
  \draw  (1) to[bend left=30] (2);
  \draw  (2) to[bend left=30] (3);
  \draw  (3) to[bend left=30] (4);
  \draw   (4) to[bend left=30] (1);
    \node(a) at (3,-0.7) {,};
 \end{tikzpicture}
 \begin{tikzpicture}[scale=0.8,mycirc/.style={circle,fill=black, minimum size=0.1mm, inner sep = 1.5pt}]
 \node[label=above:{\tiny{$(2,1)$}}] (1) at (0,2) {$\yng(3,2)$};
 \node[label=left:{\tiny{$(2,2)$}}] (2) at (2,0) {$\yng(1,1)$};
 \node[label=below:{\tiny{$(2,3)$}}] (3) at (0,-2) {${\yng(2,1)}^{\;31}$};
 \node[label=right:{\tiny{$(2,4)$}}] (4) at (-2,0) {$\yng(2)$};
  \draw  (1) to[bend left=30] (2);
  \draw  (2) to[bend left=30] (3);
  \draw  (3) to[bend left=30] (4);
  \draw   (4) to[bend left=30] (1);
  \node(a) at (3,-0.4) {,};
 \end{tikzpicture}
 \begin{tikzpicture}[scale=0.8,mycirc/.style={circle,fill=black, minimum size=0.1mm, inner sep = 1.5pt}]
 \node[label=above:{\tiny{$(3,1)$}}] (1) at (0,2) {$\yng(2,1,1)$};
 \node[label=left:{\tiny{$(3,2)$}}] (2) at (2,0) {$\yng(1,1)$};
 \node[label=below:{\tiny{$(3,3)$}}] (3) at (0,-2) {$\yng(1)$};
 \node[label=right:{\tiny{$(3,4)$}}] (4) at (-2,0) {$\yng(2)$};
  \draw  (1) to[bend left=30] (2);
  \draw  (2) to[bend left=30] (3);
  \draw  (3) to[bend left=30] (4);
  \draw   (4) to[bend left=30] (1);
  \node(a) at (3,0) {.};
  \end{tikzpicture}
 \end{example}

By depicting $\mu = (\mu_1, \mu_2, \ldots,\mu_t^n,\ldots,\mu_r) \in \Y^n(r,n-1)$ as a $(m,p)$-necklace, we get $\tm^{(n,i,j)} \in \Y^n(m,p,n-1)$, where $t = (j-1)m+i$ for a  unique pair $(i,j)$ such that $1 \leq i \leq m$ and $1 \leq j \leq p$.
 
\begin{thm}\label{irrlp}
The irreducible $\Lrpn$-modules are parametrized by the elements of $\Y^n(m,p,n-1)$.
For $\mu^{(n,t)} \in \Y^n(r,n-1)$, the restriction of the corresponding irreducible $L(r,1,n)$-module $V^{\mu^{(n,t)}}$ to $\Lrpn$ has multiplicity free decomposition given as:
$$ \mbox{\em Res}^{L(r,1,n)}_{\Lrpn}(V^{\mu^{(n,t)}}) ~ = ~ V^{\tm^{(n,i,j)}},$$
where $t = (j-1)m+i$ for a  unique pair $(i,j)$ such that $1 \leq i \leq m$ and $1 \leq j \leq p$. Also, for any $~\nu^{(n,s)} \in [\mu^{(n,t)}]$, $$ \mbox{\em Res}^{L(r,1,n)}_{\Lrpn}(V^{\mu^{(n,t)}}) ~ \cong ~ \mbox{\em Res}^{L(r,1,n)}_{\Lrpn}(V^{\nu^{(n,s)}}).$$
\end{thm}

\begin{proof}
The group $C = \langle \delta_0^m \rangle$ acts on the set of irreducible $L(r,1,n)$-modules. For $\mu^{(n,t)} \in \Y^n(r,n-1)$, suppose that $[\mu^{(n,t)}]$ denotes the elements in $\Y^n(r,n-1)$ which parametrize the irreducible $L(r,1,n)$-modules in the orbit of $V^{\mu^{(n,t)}}$. Using Corollary \ref{shl2}(b), we have 
\begin{align*}
 [\mu^{(n,t)}] & = \{\omega^{(n,y)} \mid \omega^{(n,y)} = \sh^{zm}(\mu^{(n,t)}) \mbox{ for some } z=0,1,\ldots,p-1 \}.
\end{align*}

Since the color is also shifting, therefore, the number of elements in the orbit is $p$ and thus the stabilizer subgroup consists of identity element only. The results follow from Theorem \ref{cd}.
\end{proof}

{\em Branching rule from $\Grpn$ to $\Lrpn$.}
The construction of higher Specht polynomials for $\Grpn$ from the higher Specht polynomials for $G(r,1,n)$ was described in \cite{MY98} to decompose a module isomorphic to left regular $\Grpn$-module into its irreducible submodules. Applying a similar (but not identical) construction on the canonical $GZ$-bases of irreducible $G(r,1,n)$-modules obtained in Okounkov-Vershik approach in Section \ref{pre}, we construct the bases of irreducible $\Grpn$-modules in Theorem \ref{eb}. We use such constructed basis to show in Theorem \ref{gl} that the irreducible $\Grpn$-modules $V^{(\tla,\delta_1)}$ and $V^{(\tla,\delta_2)}$, for $\tla \in \Y(m,p,n)$ and $\delta_1, \delta_2 \in C_{\La}$, are isomorphic as $\Lrpn$-modules. Theorem \ref{gl} is useful in the proof of Theorem \ref{brgpl} for description of branching rule from $\Grpn$ to $\Lrpn$.

Fix $\La \in \Y(r,n)$. Define the shift map
$\sh : \Tab \longrightarrow \Tab $ 
by $$(T_1, T_2, \ldots, T_r) \mapsto (T_r, T_1, \ldots, T_{r-1}).$$

Since $C_{\La}$ is generated by $\delta_0^{mb(\La)}$, the $G(r,1,n)$-modules $V^{\La}$ and $\delta_0^{mb(\La)} \otimes V^{\La}$  are isomorphic. Suppose that $T \in \Tab$ and ${\bf 1}_{\delta_0^{mb(\La)}}$ is the basis element of one-dimensional $G(r,1,n)$-module ${\delta_0^{mb(\La)}}$. Using Corollary \ref{shl2}(a), define the $G(r,1,n)$-linear isomorphism $\mathpzc{E}: V^{\La} \longrightarrow \delta_0^{mb(\La)} \otimes V^{\La}$ by
$$ v_T \mapsto {\bf 1}_{\delta_0^{mb(\La)}} \otimes v_{\sh^{-mb(\La)}(T)}.$$
Also, the map $\mathpzc{F}: \delta_0^{mb(\La)} \otimes V^{\La} \longrightarrow V^{\La}$  given by 
$ {\bf 1}_{\delta_0^{mb(\La)}} \otimes v_T \mapsto v_T$
is a $G(r,p,n)$-linear isomorphism.

The associator of $V^{\La}$ is given by 
\begin{align}\label{at}
\mathpzc{A}^{\La} = \mathpzc{F}\mathpzc{E} : V^{\La} & \longrightarrow V^{\La} \nonumber \\
v_T & \mapsto v_{\sh^{-mb(\La)}(T)}.
\end{align}
For $h = 1, 2, \ldots, r$, we define 
\begin{equation*}
\Tab_h = \{T = (T^1, T^2, \ldots, T^r) \in \Tab \mid n \in T^{r-\nu}, 0 \leq \nu < h \}.
\end{equation*}
For $T \in \Tab_{mb(\La)}$, we get the following $u(\La)$  distinct standard $r$-Young tableaux: $$ T, \sh^{mb(\La)}(T), \sh^{2mb(\La)}(T), \ldots, \sh^{(u(\La)-1)mb(\La)}(T). $$  
An element $\delta \in C_{\La} = \langle \delta_0^{mb(\La)} \rangle $ can be identified with $\zeta^{lmb(\La)}$ for some $0 \leq l \leq u(\La)-1$.
Fixing $\delta \in C_{\La}$, we define, for each $T \in \Tab_{mb(\La)}$, 
\begin{equation*}
 v^{(\delta)}_T := \sum\limits_{t=0}^{u(\La)-1} \zeta^{tlmb(\La)}v_{\sh^{tmb(\La)}(T)}.
\end{equation*}
The linear independence of $\{v_T \mid T \in \Tab\}$ implies that $\{ v^{(\delta)}_T \mid T \in \Tab_{mb(\La)}\}$, for a fixed $\delta \in C_{\La}$, is linearly independent.  

\begin{thm}\label{eb}
 For $\La = (\La_1, \La_2, \ldots, \La_r) \in \Y(r,n)$, consider $\tla \in \Y(m,p,n)$. For each $\delta \in C_{\La}$, define $$V^{(\tla,\delta)} := \C\mbox{-}span\{v^{(\delta)}_T \mid T \in \Tab_{mb(\La)}\}.$$ The following are true:
 \begin{enumerate}[(a)]
  \item The eigenspace decomposition of $V^{\La}$ with respect to the associator $\mathpzc{A}^{\La}$ is:
  \begin{equation}
   V^{\La} = \bigoplus_{\delta \in C_{\La}}V^{(\tla,\delta)}.
  \end{equation}
 \item The eigenspace $V^{(\tla,\delta)}$, for $\delta \in C_{\La}$, is an irreducible $\Grpn$-module.
 \item The set $\{V^{(\tla,\delta)} \mid \tla \in \Y(m,p,n), \delta \in C_{\La} \}$ is the complete set of irreducible $\Grpn$-modules.
 \end{enumerate}
\end{thm}

\begin{proof}
It can be seen from the definition of the associator $\mathpzc{A}^{\La}$ in (\ref{at}) that $$\mathpzc{A}^{\La}(v^{(\delta)}_T) = \zeta^{lmb(\La)}(v^{(\delta)}_T), \mbox{ for }\delta \in C_{\La}.$$ This implies that the subspaces $V^{(\tla,\delta)}$, for $\delta \in C_{\La}$, are contained in the distinct eigenspaces of $\mathpzc{A}^{\La}$. Thus, we have 
  \begin{equation}\label{sub}	
  \bigoplus_{\delta \in C_{\La}}V^{(\tla,\delta)} \subset V^{\La}.                                                                                                                                                                                                                                                                         
  \end{equation}
 Also for each $\delta \in C_{\La}$, the dimension of $V^{(\tla,\delta)}$ is equal to the number of elements in $\Tab_{mb(\La)}$, denoted by $\#(\Tab)$. This implies that we have $$\dim V^{(\tla,\delta)} = \frac{1}{u(\La)}\#(\Tab) = \frac{1}{u(\La)}\dim V^{\La}.$$ Thus, the dimensions of both sides in (\ref{sub}) are equal which implies equality in (\ref{sub}). This proves part (a). 
 The proofs of parts (b) and (c)  follow from Clifford theory and part (a) of this theorem.
\end{proof}

\begin{thm}\label{gl}
 For a fixed $\tla \in \Y(m,p,n)$ and $\delta_1, \delta_2 \in C_{\La}$, we have
 \begin{equation*}
  \mbox{\em Res}_{L(r,p,n)}^{G(r,p,n)}(V^{(\tla,\delta_1)}) ~ \cong ~ \mbox{\em Res}_{L(r,p,n)}^{G(r,p,n)}(V^{(\tla,\delta_2)}).
 \end{equation*}
\end{thm}

\begin{proof}
 The linear map $ \theta: V^{(\tla,\delta_1)} \longrightarrow V^{(\tla,\delta_2)} $ defined by setting
 $$ v^{(\delta_1)}_T  \mapsto v^{(\delta_2)}_T, ~ \mbox{for } T \in \Tab_{mb(\La)},$$
is an $\Lrpn$-module isomorphism.
\end{proof}

Given $\tla \in \Y(m,p,n), 1 \leq i \leq m, 1 \leq j \leq p$, let $\tla \downarrow (i,j)$ denote the set of all elements in $\Y^n(m,p,n-1)$ obtained by deleting a box from an inner corner in $\tla_{(i,j)}$ and then coloring the corresponding node by $n$. For a fixed $1 \leq i \leq m$, define $J(i) \subseteq \{1,2, \ldots,p\} $ such that for $s,t \in J(i), s \neq t$, we have  $\tla \downarrow (i,s) \bigcap \tla \downarrow (i,t) = \emptyset$. 
 If $\tla \downarrow (i,s) \bigcap \tla \downarrow (i,t) \neq \emptyset$, then $\tla \downarrow (i,s) = \tla \downarrow (i,t)$.

\begin{thm}[Branching rule from $\Grpn$ to $\Lrpn$]\label{brgpl} 
 For $\tla \in \Y(m,p,n)$ and $\delta \in C(\La)$, we have
 \begin{equation*}
\mbox{\em Res}_{L(r,p,n)}^{G(r,p,n)}(V^{(\tla, \delta)}) \cong \bigoplus_{i=1}^{m}\bigoplus_{j \in J(i)}\left(\bigoplus_{\tm^{(n,i,j)} \in \tla \downarrow (i,j)} V^{\tm^{(n,i,j)}}\right),  
 \end{equation*}
and the branching rule from $\Grpn$ to $\Lrpn$ is multiplicity free.
\end{thm}

\begin{proof}
We use the transitivity of restriction from $G(r,1,n)$ to $\Lrpn$:
$$G(r,1,n) \supset L(r,1,n) \supset L(r,p,n) \mbox{ and } G(r,1,n) \supset G(r,p,n) \supset L(r,p,n).$$ Given $\tla$, we have $ \La \in \Y(r,n)$. Considering $V^{\La}$ as $L(r,1,n)$-module, Theorem \ref{brgl1} implies that  
\begin{equation}\label{gl1v}
 \mbox{ Res}_{L(r,1,n)}^{G(r,1,n)}(V^{\La}) \cong \bigoplus_{t=1}^{r}\big(\bigoplus_{\mu \in \La \downarrow t} V^{\mu^{(n,t)}}\big).
\end{equation}
Writing $t = (j-1)m+i$, where $1 \leq i \leq m, 1 \leq j \leq p$, we note that $u(\La)$ distinct elements of $\Y^n(r,n-1)$ 
\begin{equation}\label{shu}
\mu^{(n,t)}, \sh^{mb(\La)}(\mu^{(n,t)}), \ldots, \sh^{(u(\La)-1)mb(\La)}(\mu^{(n,t)})
\end{equation}
give rise to the same $\tm^{(n,i,j)} \in \Y^n(m,p,n-1)$ and $j \in J(i)$. Also, from $\mu^{(n,s)}$ (not in (\ref{shu})) such that $s= (y-1)m+i$, where $ 1 \leq y \leq p$, we get $\tm^{(n,i,y)} \in \Y^n(m,p,n-1)$, not equivalent to $\tm^{(n,i,j)}$, and thus $y \in J(i)$ and $y \neq j$. Restricting $V^{\La}$ as $\Lrpn$-module in (\ref{gl1v}), Theorem \ref{irrlp} implies that  
\begin{equation}\label{glla}
 \mbox{ Res}_{\Lrpn}^{G(r,1,n)}(V^{\La}) \cong \left(\bigoplus_{i=1}^{m}\bigoplus_{j \in J(i)}\left(\bigoplus_{\tm^{(n,i,j)} \in \lambda \downarrow (i,j)} V^{\tm^{(n,i,j)}}\right)\right)^{\oplus u(\La)}.
\end{equation}

Considering $V^{\La}$ as $\Grpn$-module, Theorem \ref{irrgp} implies that 
\begin{equation}\label{a1}
\mbox{ Res}_{\Grpn}^{G(r,1,n)}(V^{\La}) = \bigoplus_{\delta \in C_{\La}} V^{(\tla,\delta)}.
\end{equation}
Further restricting $V^{\La}$ as $\Lrpn$-module in (\ref{a1}), using Theorem \ref{gl} and the order of $C_{\La}$ being $u(\La)$, we get  
\begin{equation}\label{ggl}
\mbox{Res}_{\Lrpn}^{G(r,1,n)}(V^{\La}) \cong  \bigg(\mbox{Res}_{\Lrpn}^{\Grpn}(V^{(\tla,\delta)})\bigg)^{\oplus u(\La)},
\end{equation}
where $\delta \in C_{\La}.$
The result follows from (\ref{glla}) and (\ref{ggl}). 
\end{proof}

\section{Schur-Weyl duality for Tanabe algebras}\label{swdu}
Let $V = \C^n$ be the $n$-dimensional vector space with standard basis $\{v_1, v_2, \ldots, v_n\}$. There is a natural action of $GL_n(\C)$ on $V$. For $k \in \Z_{\geq 0}$, consider the $k$-fold tensor product $\Vk = V \otimes V \otimes \cdots \otimes V$ with the basis $$\{v_{i_1} \otimes v_{i_2} \otimes \cdots \otimes v_{i_k} \mid 1 \leq i_1, i_2, \ldots, i_k \leq n \}.$$ With respect to this basis,  $F \in \mbox{End}(\Vk)$ can be written as a matrix $\left(F_{i_{1'}, \cdots, i_{k'}}^{i_1, \cdots, i_k}\right)$ such that 
$$(\vi)F = \sum\limits_{1 \leq i_{1'}, i_{2'}, \ldots, i_{k'} \leq n} F_{i_{1'}, \cdots, i_{k'}}^{i_1, \cdots, i_k} (v_{i_{1'}} \otimes v_{i_{2'}} \otimes \cdots \otimes v_{i_{k'}}).$$
The action of $GL_n(\C)$ on $\Vk$ is given by $$g(v_{i_1} \otimes v_{i_2} \otimes \cdots \otimes v_{i_k}) = g v_{i_1} \otimes g v_{i_2} \otimes \cdots \otimes g v_{i_k}$$ for $ g \in GL_n(\C)$ and $v_{i_1} \otimes v_{i_2} \otimes \cdots \otimes v_{i_k} \in \Vk$. The symmetric group $S_n$ can be identified with the subgroup of permutation matrices of  $GL_n(\C)$. Also, we can identify the subgroup $S_{n-1}$ of $S_n$ fixing $n$ with the subgroup of of permutation matrices having $(n,n)$-th entry as $1$ of $GL_n(\C)$. The action of $S_n$ on $\Vk$ is given by the restriction of the action of $GL_n(\C)$ to $S_n$.  Define $\Vkh := V^{\otimes k} \otimes v_n$, a subspace of $\Vkk$, which is isomorphic to $\Vk$ as a $S_{n-1}$-module.

Define a map 
\begin{align*}
\phi_k : ~   \C A_k(n) & \rightarrow  \mbox{End}(V^{\otimes k}) \\
         d & \mapsto  \phi_k(d)                 
\end{align*}
such that for $d \in A_k$ and for $1 \leq i_1, i_2, \ldots, i_k, i_{1'}, i_{2'}, \ldots, i_{k'} \leq n$, 
$$(\vi)(\phi_k(d)) = \sum\limits_{1\leq i_{1'}, i_{2'}, \ldots, i_{k'} \leq n}(\phi_k(d))^{i_1,i_2, \ldots, i_k}_{i_{1'}, i_{2'}, \ldots, i_{k'}}(\vip)$$
where
\begin{equation}\label{dcoeff}
 (\phi_k(d))^{i_1,i_2, \ldots, i_k}_{i_{1'}, i_{2'}, \ldots, i_{k'}} =
           \begin{cases}
	    1, & \mbox{ if } i_r = i_s \mbox{ when } r \mbox{ and } s \mbox{ are} \\ 
	    &\mbox{ in the same block of } d, \\ 
           0, & \mbox{ otherwise.} 
         \end{cases}
\end{equation}
This defines a right action of $\C A_k(n)$ on $\Vk$ as:
$$(\vi)d \; := \; (\vi)(\phi_k(d)).$$ 
It follows from (\ref{xd}) and (\ref{dcoeff}) that for all $d \in A_k$, 
\begin{align}\label{xdcoeff}
 (\phi_k(x_d))^{i_1,i_2, \ldots, i_k}_{i_{1'}, i_{2'}, \ldots, i_{k'}} =
 \begin{cases}
    1, & \mbox{ if } i_r = i_s \mbox{ if and only if } r \mbox{ and } s \\ 
    &  \mbox{ are in the same block of } d, \\ 
           0, & \mbox{ otherwise.} 
 \end{cases}
\end{align}
The action of the partition algebra $\C A_{k+\frac{1}{2}}(n)$ on $V^{\otimes (k+\frac{1}{2})}$ is $$\phi_{k+\frac{1}{2}}: \C A_{k+\frac{1}{2}}(n) \longrightarrow \mbox{End}(V^{\otimes (k+\frac{1}{2})})$$ given by $\phi_{k+\frac{1}{2}} = {\phi_{k+1}}_{|_{\C A_{k+\frac{1}{2}}(n)}}$.

The following theorem is \cite[Theorem 3.6]{HR05} which shows that $\C A_k(n)$ and $\C A_{k+\frac{1}{2}}(n)$ are in Schur-Weyl duality with $S_n$ and $S_{n-1}$ acting on $\Vk$ and $\Vkh$ respectively.

\begin{thm}\phantomsection\label{oswd}
\begin{enumerate}[(a)]
\item The image of the map $\phi_k : ~   \C A_k(n)  \rightarrow  \mbox{End}(\Vk)$ is $\mbox{End}_{S_n}(\Vk)$ and the kernel is given by $\C \mbox{-}span \{x_d \mid d \mbox{ has more than } n \mbox{ blocks}\}$. Thus, the partition algebra $\C A_k(n)$ is isomorphic to $\mbox{End}_{S_n}(\Vk)$ if and only if $n \geq 2k$. 
 
\item The image of the map $\phi_{k+\frac{1}{2}} : ~   \C A_{k+\frac{1}{2}}(n)  \rightarrow \mbox{End}(\Vkh)$ is $\mbox{End}_{S_{n-1}}(\Vkh)$ and the kernel is given by 
$\C\mbox{-}span \{x_d \mid d \mbox{ has more than } n \mbox{ blocks}\}$. Thus, the partition algebra $\C A_{k+\frac{1}{2}}(n)$ is isomorphic to $\mbox{End}_{S_{n-1}}(\Vkh)$ if and only if $n \geq 2k+1$.  
\end{enumerate}
\end{thm}

Let $\Pi_k(r,n)$ and $\Pi_{k+\frac{1}{2}}(r,n)$ be subsets of $\Pi_k(r)$ and $\Pi_{k+\frac{1}{2}}(r)$ (defined in Section \ref{ta}) respectively consisting of those elements which have at most $n$ blocks. Define
\begin{align*}
 & \Pi_k(r,p,n) := \Pi_k(r,n) \bigcup \Lambda_k(r,p,n), \mbox{ and } \\ 
 & \Pi_{k+\frac{1}{2}}(r,p,n) := \Pi_{k+\frac{1}{2}}(r,n) \bigcup \Lambda_{k+\frac{1}{2}}(r,p,n), 
\end{align*}
subsets of $A_k(r,p,n)$ and $A_{k+\frac{1}{2}}(r,p,n)$ respectively.

The actions of $\Grpn$ and $\Lrpn$ on $V$ are given by restrictions of the action of $GL_n(\C)$ on $V$. Also, $V$ is the reflection representation of $G(r,p,n)$. We note that $\C$-span$\{v_n\}$ is a $\Lrpn$-invariant subspace of $V$.
The following lemma gives bases of the centralizer algebras of the diagonal actions of $\Grpn$ and $\Lrpn$ on $\Vk$ and $\Vkh$ respectively. Part (a) is \cite[Lemma 2.1]{T97} and we follow the proof there to prove part (b) here.

\begin{lemma}\phantomsection\label{ca}
\begin{enumerate}[(a)]
\item $ \{\phi_k(x_d)) \mid d \in \Pi_k(r,p,n)\}$ is a basis of $\mbox{End}_{\Grpn}(V^{\otimes k})$.
\item  $ \{\phi_{k+\frac{1}{2}}(x_d) \mid d \in \Pi_{k+\frac{1}{2}}(r,p,n)\}$ is a basis of $\mbox{End}_{\Lrpn}(V^{\otimes ({k+\frac{1}{2}})})$. 
\end{enumerate}
\end{lemma}

\begin{proof} (b) An element $d \in \Pi_{k+\frac{1}{2}}(r,p,n)$ has at most $n$ blocks. By part (b) of Theorem \ref{oswd}, $\phi_{k+\frac{1}{2}}(x_d) \neq 0$. Also, $$ \{\phi_{k+\frac{1}{2}}(x_d) \mid d \in \Pi_{k+\frac{1}{2}}(r,p,n)\} \subset \{\phi_{k+\frac{1}{2}}(x_d) \mid d \in A_{k+\frac{1}{2}}\}$$ is a linearly independent set.

Since $S_{n-1}$ is a subgroup of $\Lrpn$, thus we have $$\mbox{End}_{\Lrpn}(V^{\otimes ({k+\frac{1}{2}})}) \subset \mbox{End}_{S_{n-1}}(V^{\otimes ({k+\frac{1}{2}})}).$$
Choose $0 \neq F \in \mbox{End}_{\Lrpn}(V^{\otimes ({k+\frac{1}{2}})})$. Then, $F$ can be written as 
 \begin{align*}
  F  & = \sum\limits_{d \in A_{k+\frac{1}{2}}}a_d \phi_{k+\frac{1}{2}}(x_d)\\
     & = \sum\limits_{d \in \Pi_{k+\frac{1}{2}}(r,p,n)}a_d \phi_{k+\frac{1}{2}}(x_d)
         + \sum_{\substack{d \in A_{k+\frac{1}{2}}, \\ d \notin \Pi_{k+\frac{1}{2}}(r,p,n)}}a_d \phi_{k+\frac{1}{2}}(x_d).
 \end{align*}
with $\phi_{k+\frac{1}{2}}(x_d) \neq 0$ and $a_d \neq 0$ for some $d \in A_{k+\frac{1}{2}}.$
Fix such a $d \in A_{k+\frac{1}{2}}$ and let $1 \leq i_1, \ldots,i_k, i_{1'}, \ldots, i_{k'} \leq n$ with $i_{k+1}=i_{(k+1)'}=n$ such that $$(\phi_k(x_d))^{i_1,i_2, \ldots, i_k,n}_{i_{1'}, i_{2'}, \ldots, i_{k'},n} = 1.$$
For $1 \leq u \leq n$, define $$B_u := \{j \in \{1, \ldots, k+1, 1',\ldots, (k+1)'\} \mid i_j = u \}.$$ Note that $d = (B_1, B_2, \ldots, B_n)$, where some of the blocks $B_1, B_2, \ldots, B_{n-1}$ may be empty and $\{k+1, (k+1)'\} \subseteq B_n$.

For $ 1 \leq i \leq n$, define $$t_i := (1, \ldots, 1, \zeta^p, 1, \ldots, 1),$$ where $\zeta^p$ is $i$-th component, and for $ 1 \leq i \neq j \leq n$, define $$h_{ij} := (1, \ldots, 1, \zeta, 1, \ldots, 1, \zeta^{-1}, 1, \ldots, 1),$$  where $\zeta$ and $\zeta^{-1}$ are $i$-th and $j$-th components respectively.  The elements $t_i$, for $1 \leq i \leq n$, and the elements $h_{ij}$, for $1 \leq i \neq j \leq n$, together generate $\Drpn$.

For $1 \leq i \leq n$,
$$t_i^{-1} F t_i (v_{i_1} \otimes v_{i_2} \otimes \cdots \otimes v_{i_k}\otimes v_{i_{k+1}}) =  (v_{i_1} \otimes v_{i_2} \otimes \cdots \otimes v_{i_k}\otimes v_{i_{k+1}})F$$
implies that 
\begin{align}\label{ti}
 & \sum\limits_{i_{1'}, i_{2'}, \ldots, i_{k'}} \zeta^{p(N(B_i)-M(B_i))} F_{i_{1'}, \cdots, i_{k'},i_{(k+1)'}}^{i_1, \cdots, i_k,i_{k+1}} (v_{i_{1'}} \otimes v_{i_{2'}} \otimes \cdots \otimes v_{i_{k'}}\otimes v_{i_{(k+1)'}}) \nonumber \\ 
 & = \sum\limits_{i_{1'}, i_{2'}, \ldots, i_{k'}} F_{i_{1'}, \cdots, i_{k'},i_{(k+1)'}}^{i_1, \cdots, i_k,i_{k+1}} (v_{i_{1'}} \otimes v_{i_{2'}} \otimes \cdots \otimes v_{i_{k'}}\otimes v_{i_{(k+1)'}}).
\end{align}
For $1 \leq i \neq j \leq n$,
$$h_{ij}^{-1} F h_{ij} (v_{i_1} \otimes v_{i_2} \otimes \cdots \otimes v_{i_k}\otimes v_{i_{k+1}}) =  (v_{i_1} \otimes v_{i_2} \otimes \cdots \otimes v_{i_k}\otimes v_{i_{k+1}})F$$
implies that 
\begin{align}\label{hij}
 & \sum\limits_{i_{1'}, i_{2'}, \ldots, i_{k'}} \zeta^{(N(B_i)-M(B_i))-(N(B_j)-M(B_j))} F_{i_{1'}, \cdots, i_{k'},i_{(k+1)'}}^{i_1, \cdots, i_k,i_{k+1}} (v_{i_{1'}} \otimes v_{i_{2'}} \otimes \cdots \otimes v_{i_{k'}}\otimes v_{i_{(k+1)'}}) \nonumber \\ 
 & = \sum\limits_{i_{1'}, i_{2'}, \ldots, i_{k'}} F_{i_{1'}, \cdots, i_{k'},i_{(k+1)'}}^{i_1, \cdots, i_k,i_{k+1}} (v_{i_{1'}} \otimes v_{i_{2'}} \otimes \cdots \otimes v_{i_{k'}}\otimes v_{i_{(k+1)'}}).
\end{align} 
From (\ref{ti}) and (\ref{hij}) we have
\begin{align}
 & N(B_i)  \equiv M(B_i)(\mbox{mod }m), \mbox{ for } 1 \leq i \leq n, \label{nbmbm} \\
 & N(B_i) - M(B_i) \equiv N(B_j)- M(B_j) (\mbox{mod }r), \mbox{ for } 1 \leq i \neq j
 \leq n. \label{nbmbr}
\end{align}

The following two cases arise.
\begin{enumerate}[(i)]
 \item If $N(B_1)  \equiv M(B_1) (\mbox{mod }r)$, then (\ref{nbmbr}) implies that  $N(B_i)  \equiv M(B_i) (\mbox{mod }r)$ for all $1 \leq i \leq n$. So, we have $$d \in \Pi_{k+\frac{1}{2}}(r,n).$$ 
 \item If $N(B_1)  \not\equiv M(B_1) (\mbox{mod }r)$, then (\ref{nbmbr}) implies that $N(B_i)  \not\equiv M(B_i)(\mbox{mod }r)$ for all $1 \leq i \leq n$. Thus, the number of elements, $ N(B_i) + M(B_i)$, in the block $B_i$ is  nonzero for all $1 \leq i \leq n$. So,  all the $n$ blocks, $B_1, \ldots, B_n$, in $d$ are nonempty. Along with (\ref{nbmbm}), we get $d \in \Lambda_{k+\frac{1}{2}}(r,p,n)$. 
\end{enumerate}
Combining both the cases we get that $d \in \Pi_{k+\frac{1}{2}}(r,p,n)$.
\end{proof} 

Recall from Section \ref{ta} that Tanabe algebras $\Tn$ and $\Thn$ are subalgebras of partition algebras $\C A_k(n)$ and $\C A_{k+\frac{1}{2}}(n)$ respectively. The actions of $\Tn$ and $\Thn$ on $\Vk$ and $\Vkh$ respectively are given by: 
\begin{align*}
 & \psi_k: ~ \Tn  \longrightarrow \mbox{End}(\Vk) \mbox{ with } \psi_k := {\phi_k}_{|_{\Tn}},\\
  \mbox{and } & \psi_{k+\frac{1}{2}} : ~ \Thn \longrightarrow \mbox{End}(\Vkh) \mbox{ with } \psi_{k+\frac{1}{2}} := {\phi}_{{k+\frac{1}{2}}{|_{\Thn}}}.
\end{align*}

The next theorem shows that $\Tn$ and $\Thn$ are in Schur-Weyl duality with $\Grpn$ and $\Lrpn$ acting on $\Vk$ and $\Vkh$ respectively.

\begin{thm}\phantomsection\label{swd}
 \begin{enumerate}[(a)]
\item The image of the map $ \psi_k: \Tn \rightarrow \mbox{End}(V^{\otimes k}) $ is $ \mbox{End}_{\Grpn}(V^{\otimes k})$ and the kernel is given by $ \C\mbox{-}span\{ x_d \mid d \in A_k(r,p,n) \mbox{ has more than } n \mbox{ blocks}\}.$ Thus, Tanabe algebra $\Tn$ is isomorphic to $\mbox{End}_{\Grpn}(V^{\otimes k})$ if and only if $n \geq 2k$.

\item The image of the map $ \psi_{k+\frac{1}{2}}: \Thn \rightarrow \mbox{End}(V^{\otimes ({k+\frac{1}{2}})})$ is $ \mbox{End}_{\Lrpn}(V^{\otimes ({k+\frac{1}{2}})})$ and the kernel is given by $ \C\mbox{-}span\{ x_d \mid d \in A_{k+\frac{1}{2}}(r,p,n) \mbox{ has more than } n \mbox{ blocks}\}.$ Thus, Tanabe algebra $\Thn$ is isomorphic to $\mbox{End}_{\Lrpn}(V^{\otimes ({k+\frac{1}{2}})})$ if and only if $n \geq 2k+1$.
\end{enumerate}
\end{thm}

\begin{proof} 
\begin{enumerate}[(a)]
\item For $r = 1$, this is Theorem \ref{oswd}(a) which is Schur-Weyl duality between $\C A_k(n)$ and $S_n$ acting on $\Vk$. Now consider $r \geq 2$.
Using Lemma \ref{ca}(a), we have $$\mbox{End}_{\Grpn}(V^{\otimes k}) = \C\mbox{-}span\{\psi_k(x_d) \mid d \in \Pi_k(r,p,n)\} \subset \psi_k(\Tn).$$ The element $d \in A_k(r,p,n) \setminus \Pi_k(r,p,n) $ has more than $n$ blocks. So, Theorem \ref{oswd}(a) implies that $$(\vi)\psi_k(x_d) = 0,$$
for $\vi \in \Vk$.
Thus, we get the image and kernel as stated in the theorem.  
The kernel of $\psi_k$ is zero if and only if $n \geq 2k$.

\item The proof of this part is along the similar lines as that of part (a) using Lemma \ref{ca}(b) and Theorem \ref{oswd}(b).  \qedhere
\end{enumerate}
\end{proof}

\begin{remark}
Putting $p = 1$ in Theorem \ref{swd}(a) we recover Schur-Weyl duality between $\T_k(r,1,n)$ and $G(r,1,n)$ as given in \cite[Theorem 5.4]{Ore07}. 
\end{remark}

\section{Bratteli diagram of Tanabe algebras}\label{bd}

Let us first study the decomposition of $\Vk$ and $\Vkh$ as $\Grpn$-module and as $\Lrpn$-module respectively. For the rest of the paper, we assume that $r \geq 2$. 

It can be easily seen using Okounkov-Vershik approach that the $G(r,1,n)$-module $V$ is an irreducible module parametrized by $((n-1),(1), \emptyset, \ldots, \emptyset) \in \Y(r,n)$. Using the theory of Section \ref{crg}, we see that for $(r,p,n) \neq (2,2,2)$, the $\Grpn$-module $V$ is an irreducible module parametrized by $(\tla, \delta)$, where $\tla \in \Y(m,p,n)$ and $\delta \in C_{\La}$ are as follows:
\begin{enumerate}[(i)]
 \item If $p \neq r$, then $\tla = (\tla_{(1)}, \ldots, \tla_{(m)})$ with 
 \begin{align*}
  \tla_{(1)} & = ((n-1), \emptyset, \ldots, \emptyset), \\
  \tla_{(2)} & = ((1), \emptyset, \ldots, \emptyset), \\
  \tla_{(i)} & = (\emptyset, \ldots, \emptyset), \; \mbox{ for }  i = 3, \ldots, m,
 \end{align*}
and $\delta = 1$ since $C_{\La} = \{1\}.$
 \item If $p =r$ and $(r,p,n) \neq (2,2,2)$, then $\tla = (\tla_{(1)}) = ((n-1),(1), \emptyset, \ldots, \emptyset)$ and $\delta = 1$ since $C_{\La} = \{1\}.$    
\end{enumerate}
For $(r,p,n) = (2,2,2)$, $V$ is the direct sum of irreducible $G(2,2,2)$-modules parametrzied by $(((1),(1)),1)$ and $(((1),(1)),-1)$.

Suppose that $\textbf{1}_n$ is the trivial representation of $G(r,1,n)$. Then, $ \sigma = \textbf{1}_{n-1} \otimes \sigma_2$ is a one-dimensional representation of $L(r,1,n)$ and thus, by restriction, a representation of $L(r,p,n)$. The parametrization of $\sigma$ as a $L(r,1,n)$-module is $\mu = ((n-1), \emptyset^{n}, \emptyset, \ldots, \emptyset) \in \Y^n(r,n-1)$. The parametrization of $\sigma$ as a $\Lrpn$-module is $\tm^{(n,i,j)} \in \Y^n(m,p,n-1)$ given as follows: 
\begin{enumerate}[(i)]
 \item If $p \neq r$, then $i=2, j =1$ and  $\tm^{(n,2,1)} = (\tm_{(1)}, \tm_{(2)}^{(n,1)}, \ldots, \tm_{(m)})$ with 
 \begin{align*}
  & \tm_{(1)}  = ((n-1), \emptyset, \ldots, \emptyset), \\
  & \tm_{(2)}^{(n,1)}  = (\emptyset^n, \emptyset, \ldots, \emptyset), \\
  & \tm_{(i)}  = (\emptyset, \ldots, \emptyset) \mbox{ for }  i = 3, \ldots, m.
  \end{align*}
 \item If $p =r$, then $i=1, j =2$ and  $\tm^{(n,1,2)} = (\tm_{(1)}^{(n,2)}) =((n-1),\emptyset^n, \emptyset, \ldots, \emptyset)$.
\end{enumerate}
  Using the above parametrizations of $V$ and $\sigma$, by Frobenius reciprocity and Theorem \ref{brgpl}, we have (for $(r,p,n) = (2,2,2)$ also) $$ V \cong \mbox{Ind}_{\Lrpn}^{\Grpn}(\sigma).$$
  
  Let $M$ be a $\Grpn$-module. Then using the tensor identity, we have 
\begin{align*}
\mbox{Ind}_{L(r,p,n)}^{G(r,p,n)}\bigg(\mbox{Res}_{L(r,p,n)}^{G(r,p,n)}(M) \otimes \sigma \bigg) & \cong M \otimes \mbox{Ind}_{L(r,p,n)}^{G(r,p,n)}(\sigma) \\ 
& \cong M \otimes V.
\end{align*}
Thus, taking $M = V^{\otimes(k-1)}$ for $k \geq 1$, we have 
\begin{align}
 \mbox{Ind}_{L(r,p,n)}^{G(r,p,n)}\bigg(\mbox{Res}_{L(r,p,n)}^{G(r,p,n)}(V^{\otimes(k-1)}) \otimes \sigma \bigg) & \cong  V^{\otimes k} \label{ind}\\
 \mbox{and Res}_{L(r,p,n)}^{G(r,p,n)}\bigg(\mbox{Ind}_{L(r,p,n)}^{G(r,p,n)}\bigg(\mbox{Res}_{L(r,p,n)}^{G(r,p,n)}(V^{\otimes(k-1)}) \otimes \sigma \bigg)\bigg) & \cong  V^{\otimes k} \label{res}
 \end{align}
 as $\Grpn$-module and $\Lrpn$-module respectively. 
 
It will be clear from the context whether we consider $\sigma$ as a $L(r,1,n)$-module or as a $\Lrpn$-module.
Given  $\La^{(n,t)} \in \Y^n(r,n-1)$, assume that $V^{\La^{(n,t)}}$ is the corresponding irreducible $L(r,1,n)$-module.
\begin{lemma}\label{tp1} 
For $\La^{(n,t)} \in \Y^n(r,n-1)$
 \begin{equation}
  V^{\La^{(n,t)}} \otimes \sigma = V^{\La^{(n,z)}}
 \end{equation}
where $\La^{(n,z)} \in \Y^n(r,n-1)$ and $z = (t+1)(\mbox{\em mod }r).$
\end{lemma}
\begin{proof}
 Noting that the $GZ$-subspace of $\sigma$ is given by $\underbrace{\sigma_1 \otimes \cdots \otimes \sigma_1}_{(n-1) \mbox{-fold}} \otimes  \sigma_2$ with $GZ$-basis element given by $\underbrace{v_1 \otimes \cdots \otimes v_1}_{(n-1)\mbox{-fold}} \otimes v_2$, the proof is similar to that of Theorem \ref{shl}. 
\end{proof}

Given an $(n,i,j)$-colored $(m,p)$-necklace $\tla^{(n,i,j)} \in \Y^n(m,p,n-1)$, suppose that $V^{\tla^{(n,i,j)}}$ is the corresponding irreducible $\Lrpn$-module.

\begin{lemma}\label{tp} 
For $\tla^{(n,i,j)} \in \Y^n(m,p,n-1)$
 \begin{equation}
  V^{\tla^{(n,i,j)}} \otimes \sigma = V^{\tla^{(n,x,y)}}
 \end{equation}
where $\tla^{(n,x,y)} \in \Y^n(m,p,n-1)$ is obtained from $\tla^{(n,i,j)}$ by the following rule{\em:} 
\begin{enumerate}[(i)]
 \item If $i < m$, then $x=i+1$ and $y=j$;
 \item If $i = m$, then $x=1$ and $y=(j+1)(\mbox{mod }p)$.
\end{enumerate}
\end{lemma}

\begin{proof}
 The proof follows by using Lemma \ref{tp1} and Theorem \ref{irrlp}.
\end{proof}

Define the sets $\vo_k(r,p,n)$ and $\vo_{k+\frac{1}{2}}(r,p,n)$ as follows. 
Let 
$${\vo}_0(r,p,n)= \{(\tla,1)\}$$ where $\tla = (((n),\emptyset, \ldots, \emptyset), (\emptyset, \ldots, \emptyset), \ldots, (\emptyset, \ldots, \emptyset)) \in \Y(m,p,n)$.
For $k \in \Z_{>0}$ the sets $\vo_k(r,p,n) \subseteq \Y(m,p,n) \times C $ and $\vo_{k+\frac{1}{2}}(r,p,n) \subseteq \Y^n(m,p,n-1)$ are obtained by the following recursive rule. 

{\em From $\vo_k(r,p,n)$ to $\vo_{k+\frac{1}{2}}(r,p,n)$}:

For $(\tla,\delta) \in \vo_k(r,p,n)$, let $\tla_{(i,j)-} \in \Y(m,p,n-1)$ be the set of  $(m,p)$-necklaces obtained by deleting an inner corner from $\tla_{(i,j)}$. 
For $\tm \in \tla_{(i,j)-}$, color $\tm$ by $(n,i,j)$ to obtain $\tm^{(n,i,j)} \in \vo_{k+\frac{1}{2}}(r,p,n)$. 

{\em From $\vo_{k+\frac{1}{2}}(r,p,n)$ to $\vo_{k+1}(r,p,n)$}:

For $\tm^{(n,i,j)} \in \vo_{k+\frac{1}{2}}(r,p,n)$, remove the color $(n,i,j)$ to get $\tm \in \Y(m,p,n-1)$ and then add a box to an outer corner, either in the $j$-th node of $(i+1)$-th component of $\tm$ if $1 \leq i \leq m-1$ or in the $(j+1)(\mbox{mod } p)$-th node of the first component of $\tm$ if $i = m$, to obtain $\tilde{\nu} \in \Y(m,p,n)$. Let $C_{\nu}$ be the correspoding stabilizer subgroup. For $\delta \in C_\nu \subseteq C$, the element $(\tilde{\nu}, \delta) \in \vo_{k+1}(r,p,n)$.

\begin{thm}\label{decvk}
The indexing sets of the irreducible $\Grpn$-modules occuring in $\Vk$ and of the irreducible $\Lrpn$-modules occuring in $\Vkh$ are $\vo_k(r,p,n)$ and $\vo_{k+\frac{1}{2}}(r,p,n)$ respectively.
\end{thm}
\begin{proof}
 The proof follows from (\ref{ind}), (\ref{res}), Lemma
 \ref{tp}, branching rule from $\Grpn$ to $\Lrpn$ in Theorem \ref{brgpl}, Frobenius reciprocity and the observation that the spaces $\Vkh$ and $\Vk$ are isomorphic as $\Lrpn$-modules.
\end{proof}

\begin{thm}\label{irrend}
 The indexing sets of the irreducible $\mbox{End}_{\Grpn}(\Vk)$-modules and of the irreducible $\mbox{End}_{\Lrpn}(\Vkh)$-modules are $\vo_k(r,p,n)$ and $\vo_{k+\frac{1}{2}}(r,p,n)$ respectively.
\end{thm}

\begin{proof}
 The proof is a consequence of the centralizer theorem (\cite[Theorem 5.4]{HR05}) and Theorem \ref{decvk}. 
\end{proof}

\begin{thm}\label{dvkp}
 Let $n$ and $k$ be nonnegative integers.
 \begin{enumerate}[(a)]
  \item For $n \geq 2k$, as $(\C[G(r,p,n)], \Tn)$-bimodule,
  $$ \Vk \cong \bigoplus\limits_{(\tla,\delta) \in \vo_k(r,p,n)} \bigg(V^{(\tla, \delta)} \otimes \T_k^{(\tla, \delta)}\bigg),$$
 where $V^{(\tla, \delta)}$ is the irreducible $G(r,p,n)$-module and $\T_k^{(\tla,\delta)}$ is the irreducible $\Tn$-module parametrized by $(\tla,\delta) \in \vo_k(r,p,n)$. Also
 \begin{align*}
 \mbox{\em dim}(T_k^{(\tla, \delta)}) = & \mbox{ the number of paths from }(((n),\emptyset, \ldots, \emptyset),1)  \in \vo_0(r,p,n) \\ 
 &\mbox{ to } (\tla, \delta) \in \vo_k(r,p,n) \mbox{ in the Bratteli diagram }\widehat{\T}(r,p,n).
 \end{align*}
 
 \item For $n \geq 2k+1$, as $(\C[L(r,p,n)], \Thn))$-bimodule,
  $$ \Vkh \cong \bigoplus\limits_{\tm^{(n,i,j)} \in \vo_{k+\frac{1}{2}}(r,p,n)} \bigg(V^{\tm^{(n,i,j)}} \otimes \T_{k+\frac{1}{2}}^{\tm^{(n,i,j)}}\bigg),$$
 where $V^{\tm^{(n,i,j)}}$ is the irreducible $L(r,p,n)$-module and $\T_{k+\frac{1}{2}}^{\tm^{(n,i,j)}}$ is the irreducible 
 
 $\Thn$-module parametrized by $\tm^{(n,i,j)} \in \vo_{k+\frac{1}{2}}(r,p,n)$ and 
 \begin{align*}
 \mbox{\em dim}(T_k^{\tm^{(n,i,j)}}) = & \mbox{ the number of paths from }(((n),\emptyset, \ldots, \emptyset),1)  \in \vo_0(r,p,n) \\ 
 &\mbox{ to } \tm^{(n,i,j)} \in \vo_{k+\frac{1}{2}}(r,p,n) \mbox{ in the Bratteli diagram }\widehat{\T}(r,p,n).
 \end{align*}
 \end{enumerate}

\end{thm}
 
\begin{proof}
  The proofs of $(a)$ and $(b)$ follow from Theorem \ref{swd}(a) and (b) respectively along with the centralizer theorem, Theorem \ref{decvk} and Theorem \ref{irrend}.
 \end{proof}
 
For $(\tla, \delta) \in \vo_k(r,p,n)$, define the set $A_{k-\frac{1}{2}}^{(\tla, \delta)}$ as consisting of the elements $\tm^{(n,i,j)} \in \vo_{k-\frac{1}{2}}(r,p,n)$ for some $ 1 \leq i \leq m$ and $1 \leq j \leq p$ such that $(\tla, \delta)$ is obtained from $\tm^{(n,i,j)}$ while constructing $\vo_k(r,p,n)$ from $\vo_{k-\frac{1}{2}}(r,p,n)$.
For $\tm^{(n,i,j)} \in \vo_{k+\frac{1}{2}}(r,p,n)$, define the set $A_{k}^{\tm^{(n,i,j)}}$ as consisting of the elements $(\tla, \delta) \in \vo_{k}(r,p,n)$ such that $\tm^{(n,i,j)}$ is obtained from $(\tla, \delta)$  while constructing $\vo_{k+\frac{1}{2}}(r,p,n)$ from $\vo_k(r,p,n)$. 
 
 \begin{cor}\phantomsection\label{brtkh}
 \begin{enumerate}[(a)]
  \item For $n\geq 2k$ and for $(\tla, \delta) \in \vo_k(r,p,n)$, we have $$\mbox{\em Res}_{\Tmn}^{\Tn}(\T^{(\tla,\delta)}_{k}) = \bigoplus_{\tm^{(n,i,j)} \in A_{k-\frac{1}{2}}^{(\tla, \delta)}} \T_{k-\frac{1}{2}}^{\tm^{(n,i,j)}}.$$
  \item For $n\geq 2k+1$ and for $\tm^{(n,i,j)} \in \vo_{k+\frac{1}{2}}(r,p,n)$, we have $$\mbox{\em Res}_{\Tn}^{\Thn}(\T^{\tm^{(n,i,j)}}_{k+\frac{1}{2}}) = \bigoplus_{(\tla, \delta) \in A_k^{\tm^{(n,i,j)}}} \T_k^{(\tla, \delta)}.$$
 \end{enumerate}
 \end{cor}
 
 \begin{proof}
 \begin{enumerate}[(a)]
  \item Using Theorem~\ref{dvkp}(a) and (\ref{ind}),
  \begin{eqnarray*}
  \T^{(\tla,\delta)}_{k}&\cong& \mbox{Hom}_{\Grpn}(V^{\otimes k},V^{(\tla,\delta)})\\
  &\cong& \mbox{Hom}_{\Grpn}({\mbox{Ind}^{\Grpn}_{\Lrpn}}((\mbox{Res}^{\Grpn}_{\Lrpn}V^{\otimes (k-1)})\otimes \sigma),V^{(\tla,\delta)}).
  \end{eqnarray*}
  From the above isomorphism and Frobenius reciprocity, we get
  \begin{align}\label{restkh}
  \mbox{Res}_{\Tmn}^{\Tn} \T^{(\tla,\delta)}_{k}
  &\cong  \mbox{Res}_{\Tmn}^{\Tn}\mbox{Hom}_{\Lrpn}((\mbox{Res}^{\Grpn}_{\Lrpn}V^{\otimes (k-1)})\otimes \sigma,
  {\mbox{Res}^{\Grpn}_{\Lrpn}}V^{(\tla,\delta)}) \nonumber \\
  &\cong  \mbox{Hom}_{\Lrpn}(V^{\otimes (k-\frac{1}{2})}, \sigma' \otimes {\mbox{Res}^{\Grpn}_{\Lrpn}}V^{(\tla,\delta)})
  \end{align}
where $\sigma'$ is the contragredient representation of $\sigma$. The $L(r,1,n)$-representation  $\sigma'= \sigma^{-1} = \underbrace{\sigma \otimes \cdots \otimes \sigma}_{(r-1)\mbox{-fold}}$ and thus, $\sigma'$ is parametrized by $((n-1),\emptyset,\ldots,\emptyset,\emptyset^{n})\in \Y^{n}(r,n)$. First using Theorem~\ref{brgpl} and then by the repeated application of Lemma \ref{tp}, we compute $\sigma' \otimes {\mbox{Res}^{\Grpn}_{\Lrpn}}V^{(\tla,\delta)}$. Then, from (\ref{restkh}) and Theorem~\ref{dvkp}(b), we have the restriction rule.

\item The proof is along the similar lines as that of part (a) using Theorem \ref{dvkp}(b), (\ref{res}), Theorem \ref{brgpl}, Frobenius reciprocity and Theorem \ref{dvkp}(a).
\qedhere
 \end{enumerate}
\end{proof}

Orellana \cite[p.\,614]{Ore07} describes the rule for recursively constructing Bratteli diagram for the tower of algebras 
$$ \T_0(r,1,n) \subseteq \T_1(r,1,n) \subseteq \T_2(r,1,n) \subseteq \cdots.$$
We consider the tower of Tanabe algebras 
\begin{equation}\label{bt} 
\T_0(r,p,n) \subseteq \T_{\frac{1}{2}}(r,p,n) \subseteq \T_1(r,p,n) \subseteq \T_{\frac{3}{2}}(r,p,n)  \subseteq \cdots \subseteq \T_{\lfloor \frac{n}{2} \rfloor}(r,p,n)
\end{equation}
and using Theorems \ref{decvk}, \ref{irrend} and Corollary \ref{brtkh}, construct the corresponding Bratteli diagram $\widehat{\T}(r,p,n)$ recursively by the following rule: For $l \in \frac{1}{2}\Z_{\geq 0}$, the vertices at $l$-th level of Bratteli diagram are elements of the set $\vo_l(r,p,n)$. A vertex $\mathpzc{v}_l$ at $l$-th level is connected by an edge with a vertex $\mathpzc{v}_{l+\frac{1}{2}}$ at $(l+\frac{1}{2})$-th level if and only if $\mathpzc{v}_{l+\frac{1}{2}}$ is obtained from $\mathpzc{v}_l$ while constructing $\vo_{l+\frac{1}{2}}(r,p,n)$ from $\vo_l(r,p,n)$. The  Bratteli diagram of Tanabe algebras is a simple graph.

\begin{remark}\label{cp}
For $t \in \Z_{\geq 0}$, $t \leq \lfloor \frac{n}{2} \rfloor $ and $ (\tla, \delta) \in \vo_t(r,p,n)$, the stabilizer subgroup $C_{\La}$ is non-trivial if and only if $(r,p,n) =(2,2,2k)$ and $t = k$; in this case $C_{\La} = \{1,-1\}$. Thus, for $n \geq 2k$, there is a one-to-one correspondence between the irreducible representations of the same degree occuring at $t$-th level in Bratteli diagrams for $\T_k(r,1,n)$ and $\Tn$ if and only if $(r,p,n, t) \neq (2,2,2k,k)$; the correspondence in terms of parametrization is $\La \mapsto (\tla, 1)$.
\end{remark}

We draw Bratteli diagram for $(r,p,n) = (2,2,4)$. Note that at level $k = 2$, a node parametrized by $(\tla, -1)$ also appears when $\La = ((2),(2))$ because $C_\La$ is nontrivial. 
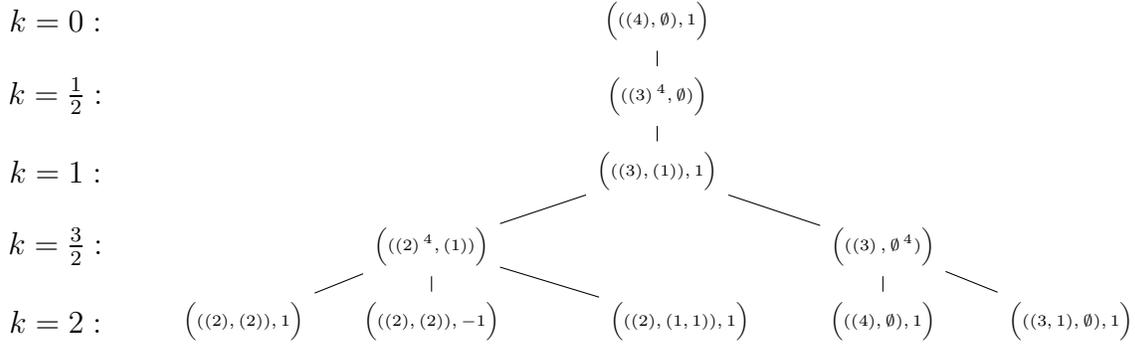
\begin{center}
 \begin{tikzpicture}
 \node (a) at (-3,0) {$k=0:$};
 \node (b) at (-3,-1) {$k=\frac{1}{2}:$};
 \node (c) at (-3,-2) {$k=1:$};
 \node (d) at (-3,-3) {$k=\frac{3}{2}:$};
 \node (e) at (-3,-4) {$k=2:$};
  \node (1) at (5,0) {\tiny{$\bigg(((4), \emptyset), 1\bigg)$}};
  \node (2) at  (5,-1) {\tiny{$\bigg({((3)\,}^{4}, \emptyset) \bigg)$}};
  \node (3) at  (5,-2) {\tiny{$\bigg(((3),  (1)), 1\bigg)$}};
  \node (4) at (2,-3) {\tiny{$\bigg({((2)\,}^{4}, (1))\bigg)$}};
  \node (5) at (8,-3) {\tiny{$\bigg({((3)\,}, \emptyset \,^4)\bigg)$}};
  \node (6) at (-0.5,-4) {\tiny{$\bigg(((2), (2)),1\bigg)$}};
  \node (7) at (2,-4) {\tiny{$\bigg(((2), (2)),-1\bigg)$}};
  \node (8) at (5.3,-4) {\tiny{$\bigg(((2), (1,1)),1\bigg)$}};
  \node (9) at (8,-4) {\tiny{$\bigg(((4), \emptyset),1\bigg)$}};
  \node (10) at (10.5,-4) {\tiny{$\bigg(((3,1), \emptyset),1\bigg)$}};
  \draw (1) -- (2) -- (3) -- (4) -- (6);
  \draw (3) -- (5);
  \draw (4) -- (7);
  \draw (4) -- (8);
  \draw (5) -- (9);
  \draw (5) -- (10);
 \end{tikzpicture}
 \end{center} 

The following is a part of Bratteli diagram of Tanabe algebras when $(r,p,n) = (6,2,6)$. Note that  $\T_\frac{5}{2}^{\pv}$ corresponding to $\pv = \left(((4), \, \emptyset), \, ((1)^{6}, \,\emptyset),\,( \emptyset, \, \emptyset))\right)$ is of dimension two. (Due to limitation of width of the page, we have written the nodes at level $k = \frac{5}{2}$ in two rows.)

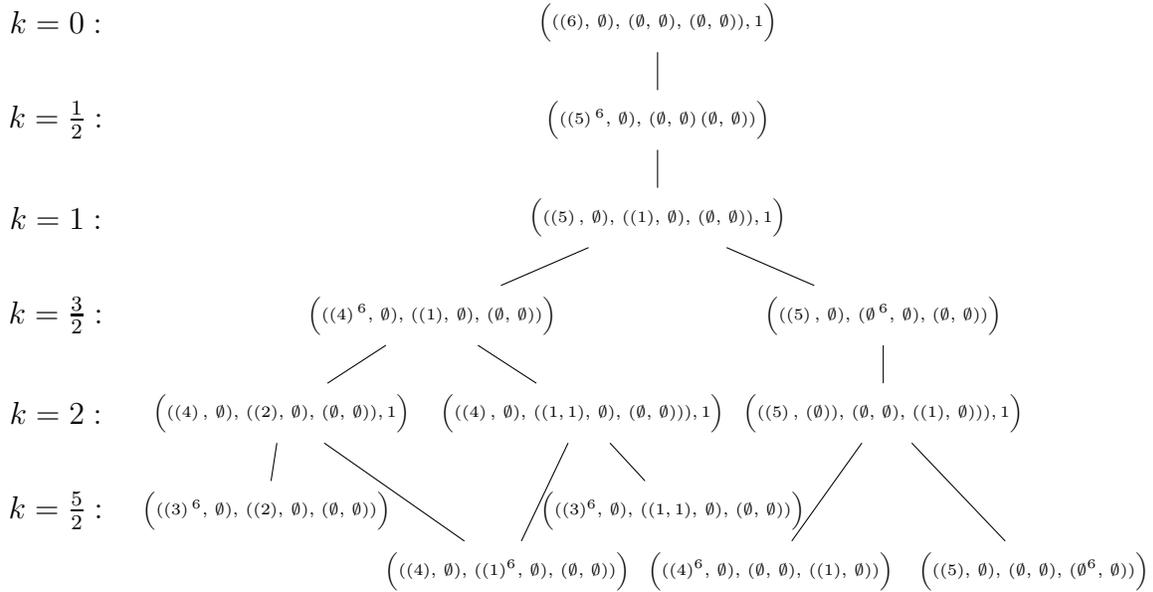
\begin{center}
 \begin{tikzpicture}
 \node (a) at (-3,0) {$k=0:$};
 \node (b) at (-3,-1.3) {$k=\frac{1}{2}:$};
 \node (c) at (-3,-2.6) {$k=1:$};
 \node (d) at (-3,-3.9) {$k=\frac{3}{2}:$};
 \node (e) at (-3,-5.2) {$k=2:$};
 \node (f) at (-3,-6.5) {$k=\frac{5}{2}:$};
  \node (1) at (5,0) {\tiny{$\bigg(((6), \, \emptyset), \,( \emptyset, \, \emptyset),\,( \emptyset, \, \emptyset)), 1\bigg)$}};
  \node (2) at  (5,-1.3) {\tiny{$\bigg({((5)\,}^{6}, \, \emptyset), \, (\emptyset, \,\emptyset)\,( \emptyset, \, \emptyset)) \bigg)$}};
  \node (3) at  (5,-2.6) {\tiny{$\bigg(((5)\,, \, \emptyset), \, ((1), \,\emptyset),\,( \emptyset, \, \emptyset)),1\bigg)$}};
  \node (4) at (2,-3.9) {\tiny{$\bigg({((4)\,}^{6}, \, \emptyset), \, ((1), \,\emptyset),\,( \emptyset, \, \emptyset))\bigg)$}};
  \node (5) at (8,-3.9) {\tiny{$\bigg({((5)\,}, \, \emptyset), \, (\emptyset \,^6, \,\emptyset),\,( \emptyset, \, \emptyset))\bigg)$}};
  \node (6) at (0,-5.2) {\tiny{$\bigg(((4)\,, \, \emptyset), \, ((2), \,\emptyset),\,( \emptyset, \, \emptyset)),1\bigg)$}};
  \node (7) at (4,-5.2) {\tiny{$\bigg(((4)\,, \, \emptyset), \, ((1,1), \,\emptyset),\,( \emptyset, \, \emptyset))),1\bigg)$}};
  \node (8) at (8,-5.2) {\tiny{$\bigg(((5)\,, \, (\emptyset)), \, (\emptyset, \,\emptyset),\,((1), \, \emptyset))),1\bigg)$}};
  \node (9) at (-0.2,-6.5){\tiny{$\bigg({((3)\,}^{6}, \, \emptyset), \, ((2), \,\emptyset),\,( \emptyset, \, \emptyset))\bigg)$}};
  \node (10) at (3,-7.3) {\tiny{$\bigg(((4), \, \emptyset), \, ((1)^{6}, \,\emptyset),\,( \emptyset, \, \emptyset))\bigg)$}};
  \node (11) at (5.2,-6.5) {\tiny{$\bigg(((3)^{6}, \, \emptyset), \, ((1,1), \,\emptyset),\,( \emptyset, \, \emptyset))\bigg)$}};
  \node (12) at (6.5,-7.3) {\tiny{$\bigg(((4)^{6}, \, \emptyset), \, (\emptyset, \,\emptyset),\,((1), \, \emptyset))\bigg)$}};
  \node (13) at (10,-7.3) {\tiny{$\bigg(((5), \, \emptyset), \, (\emptyset, \,\emptyset),\,( \emptyset^{6}, \, \emptyset))\bigg)$}};
  \draw (1) -- (2) -- (3) -- (4) -- (6);
  \draw (3) -- (5);
  \draw (4) -- (7);
  \draw (5) -- (8);
  \draw  (6)-- (9);
  \draw (6) -- (10);
  \draw (7) -- (10);
  \draw (7) -- (11);
  \draw (8) -- (12);
  \draw (8) -- (13);
 \end{tikzpicture}
 \end{center}

\section{Jucys-Murphy elements for Tanabe algebras}\label{jme}
 Recall from Section \ref{pre} that the Jucys-Murphy elements for $G(r,1,n)$ are:   
\begin{align*}
& X_1 = 0, \\
\mbox{ and }&  X_j = \sum_{i=1}^{j-1}\sum\limits_{l=0}^{r-1}\zeta_i^{l}\zeta_j^{-l}s_{ij},\;\;2\leq
j\leq n.
\end{align*}
For $T \in \Tab$, we have $$\sum\limits_{b \in \lambda}c(b) := \sum\limits_{j=1}^n c(b_{T}(j))$$
and it is easily seen that $\sum\limits_{b \in \lambda}c(b)$ is independent of the choice of $T \in \Tab$.
\begin{lemma}\phantomsection\label{ka}
\begin{enumerate}[(a)]
 \item For $r,n \in \Z_{\geq 0}$, $$ \ka_{r,n} :=  \frac{1}{r}\sum\limits_{l=0}^{r-1}\sum\limits_{1 \leq i < j \leq n}\zeta_i^{l}\zeta_j^{-l}s_{ij}$$
 is a central element of $\C[G(r,p,n)]$ and $\ka_{r,n} = \sum\limits_{b \in \lambda}c(b)$ as operators on $V^{(\tla,\delta)},$ the irreducible $G(r,p,n)$-module parametrized by $(\tla, \delta) \in \Y(m,p,n) \times C_\La$.
 \item For $r,n \in \Z_{\geq 0}$, $$ \ka_{r,n-1}  := \frac{1}{r}\sum\limits_{l=0}^{r-1}\sum\limits_{1 \leq i < j \leq n-1}\zeta_i^{l}\zeta_j^{-l}s_{ij}$$
 is a central element of $\C[L(r,p,n)]$ and $\ka_{r,n-1} = \sum\limits_{b \in \mu}c(b)$ as operators on $V^{\mu^{(n,i)}}$, the irreducible $L(r,1,n)$-module parametrized by $\mu^{(n,i)} \in \Y^n(r,n-1)$.
 \end{enumerate}
 \end{lemma}
\begin{proof}
\begin{enumerate}[(a)]
\item First, we consider the case $p=1$. Being the sum of elements in the conjugacy class of $(1,1, \ldots, 1, s_{12})$, $\ka_{r,n}$ is a central element of $\C[G(r,1,n)]$  and  
 \begin{equation*}
  \ka_{r,n} =  \frac{1}{r}\sum_{j=1}^{n} X_j.
 \end{equation*}
 For the irreducible $G(r,1,n)$-module $V^{\La}$ parametrized by $\La \in \Y(r,n)$, the canonical decomposition of $V^\La$ into $GZ$-subspaces is $$V^{\lambda} = \bigoplus\limits_{T \in {\it Tab}(r, \La)}V_T.$$  
Using Theorem \ref{gzt} for the action of Jucys-Murphy elements of $G(r,1,n)$, we have  $$\ka_{r,n}(v_T) = \sum\limits_{j=1}^n c(b_T(j))(v_T) = \bigg(\sum\limits_{b \in \lambda}c(b)\bigg)v_T$$ where $v_T$ is $GZ$-basis element corresponding to $T \in \Tab$. Thus, $\ka_{r,n} = \sum\limits_{b \in \lambda}c(b) $ as operators on $V^{\La}$.  
For a divisor $p$ of $r$, note that $\ka_{r,n} \in \C[\Grpn] \subseteq \C[G(r,1,n)]$. Thus, $\ka_{r,n}$ is a central element of $\C[\Grpn]$ also, and its action on the irreducible $\Grpn$-module $V^{(\tla, \delta)}$ follows by restricting the action of $G(r,1,n)$ on $V^\La$.

\item The proof is along the similar lines as that of part (a). \qedhere
\end{enumerate}
\end{proof}

Now, we describe a specific central element in $\C[G(2,2,2k)]$. The conjugacy class $C$ of the element $(1,1,\ldots,1,(12)(34) \cdots (2k-1, 2k))$ in $G(2,1,2k)$ consists of elements of the 
form $(a_1, a_2, \ldots, a_{2k}, (i_1, i_2)(i_3,i_4) \cdots (i_{2k-1}, i_{2k}))$ such 
that $(i_1, i_2), (i_3,i_4), \ldots, (i_{2k-1}, i_{2k})$ are mutually disjoint transpositions in $S_{2k}$,  and $a_{i_j}a_{i_{j+1}} = 1 $ for all $j = 1, 3, \ldots, 2k-1$ with $a_i \in G = \Z/2\Z = \{1,-1\}$ for all $i=1,\ldots, 2k$. Using \cite[Theorem 11]{Re77}, the conjugacy class of $(1,1,\ldots,1,(12)(34) \cdots (2k-1, 2k))$ in $G(2,1,2k)$ decomposes into two conjugacy classes, denoted by $C_1$ and $C_2$, in $G(2,2,2k)$ with 
representatives $$c_1 = (-1,-1,1, \ldots, 1, (12)(34) \cdots (2k-1, 2k))$$ $$ \mbox{ and } c_2 = (1,1, \ldots, 1, (12)(34) \cdots (2k-1, 2k))~~~~~~~~~~~~~~~$$ respectively.  The classes $C_1$ and $C_2$ consist of those elements in $C$ such that the number of pairs $(a_{i_j},a_{i_{j+1}}) = (-1,-1) $, where $j = 1, 3, \ldots, 2k-1$, is odd and even respectively. Let $z_1$ and $z_2$ be the conjugacy class sums of $C_1$ and $C_2$ respectively. Define $z = z_2 - z_1$ which is a central element in $\C[G(2,2,2k)]$.
 
\begin{lemma}\phantomsection\label{actz}
Let $\La \in \Y(2,2k)$.
\begin{enumerate}[(a)]
 \item For $\La \neq ((k),(k))$, $z=0$ as operators on the irreducible $G(2,2,2k)$-module $V^{(\tla, 1)}$. 
 \item For $\La = ((k),(k))$, $z = 2^k k!$ as operators on the irreducible $G(2,2,2k)$-module $V^{(\tla, 1)}$ and $z = -2^k k!$ as operators on the irreducible $G(2,2,2k)$-module $V^{(\tla, -1)}$.
\end{enumerate}
\end{lemma}

\begin{proof} 
In the following, we use \cite[Theorem 6.10]{MS16} to describe the action of $z$ on irreducible $G(2,2,2k)$-modules. 
The irreducible $G(2,1,2k)$-module $V^{\La}$ parametrized by $\La = (\La_1, \La_2) \in \Y(2,2k)$ has a $GZ$-basis element $v_R$ where $R = (R_1, R_2)$ is the element of $\mbox{Tab}(2,\La)$ written in row major order, i.e., the entries in $R_1$ are  in from $1 ,\ldots, |\La_1|$ and entries in $R_2$ are from $|\La_1|+1, \ldots, |\La_1|+|\La_2|$, both filled in row major order.
We have the following cases: 
\begin{enumerate}[(a)]
 \item $\La \neq ((k),(k))$: $V^{\La}$ remains irreducible as $G(2,2,2k)$-module and $V^{(\tla, 1)} = V^\La$ with $v_R^{(1)} = v_R$ as one of the basis elements using the parametrization of irreducible $G(2,2,2k)$-module in Theorem \ref{irrgp} and construction of basis of irreducible $G(2,2,2k)$-modules. 
 
 Let $Y$ be the set of those $\pi \in S_{2k}$ which can be written as a product of disjoint transpositions such that the elements of each transposition are either in $R_1$ or in $R_2$. For a fixed $\pi \in Y$, the action of $\sum(a_1, \ldots, a_{2k}, \pi)$ on $v_R^{(1)}$, where the sum is over all such elements in $C_1$, is equal to the action of $\sum(b_1, \ldots, b_{2k}, \pi)$ on $v_R^{(1)}$, where the sum is over all such elements in $C_2$. 
 
 The coefficient of $v_R$ in $t v_R^{(1)}$ is zero for any $t \in C_1 \cup C_2$ which is of the form $t = (a_1, a_2, \ldots, a_{2k}, (i_1, i_2)(i_3,i_4) \cdots (i_{2k-1}, i_{2k}))$ such that there is at least one transposition $(i_y, i_{y+1})$ with one of $i_y, i_{y+1}$ being from the entries in $R_1$ and the other being from the entries in $R_2$. 
 
 Thus, we have $zv_R^{(1)} = 0.$
 \item $\La = ((k),(k))$: $V^{\La}$ decomposes into two irreducible as $G(2,2,2k)$-modules $V^{(\tla, 1)}$ and $V^{(\tla, -1)}$ with $v_R^{(1)} = v_R+v_{\sh(R)}$ and $v_R^{(-1)} = v_R-v_{\sh(R)}$ as one of their basis elements respectively. Analogous to part (a), for a fixed $\pi \in Y$, the action of $\sum(a_1, \ldots, a_{2k}, \pi)$ on $v_R$ and $v_{\sh(R)}$, where the sum is over all such elements in $C_1$, is equal to the action of $\sum(b_1, \ldots, b_{2k}, \pi)$ on $v_R$ and $v_{\sh(R)}$, where the sum is over all such elements in $C_2$, respectively.
 
 Let $P$ be the set of those $\beta \in S_{2k}$ which can be written as a product of disjoint transpositions such that one element of each transposition is from $1,\ldots,k$ and the other one is from $k+1, \ldots, 2k$. The order of $P$ is $k!$. For $(a_1, \ldots, a_{2k}, \beta) \in C_1$ and $\beta \in P$, $(a_1, \ldots, a_{2k}, \beta)v_R = - v_{\sh(R)}$ and $(a_1, \ldots, a_{2k}, \beta)v_{\sh(R)} = - v_R$.
 For $(a_1, \ldots, a_{2k}, \beta) \in C_2$ and $\beta \in P$, $(a_1, \ldots, a_{2k}, \beta)v_R = v_{\sh(R)}$ and $(a_1, \ldots, a_{2k}, \beta)v_{\sh(R)} = v_R$.
 
 For those elements $(a_1, \ldots, a_{2k}, \gamma) \in C_1 \cup C_2$ such that $\gamma \notin Y \cup P$, the coefficients of both $v_R$ and $v_{\sh(R)}$ in both $(a_1, \ldots, a_{2k}, \gamma)v_R$ and $(a_1, \ldots, a_{2k}, \gamma)v_{\sh(R)}$ are zero.
 Thus,
 $$ z v_R^{(1)} = (2^k k!) v_R^{(1)} \mbox{ and } z v_R^{(-1)} = -(2^k k!) v_R^{(-1)}. $$
 Thus, we get the scalars as stated in the theorem.  \qedhere
 \end{enumerate}

\end{proof}

Assume that $S$ is a subset of $\{1, 2, \ldots, k\}$, $I$ is a subset of $S \bigcup S'$ and $I^c$ denotes the complement of $I$ in $S \bigcup S'$, where $S'$ is the set of all $j'$ such that $j \in S$. Define the elements $b_S$ and $ d_I$ of the partition monoid $A_k$:
\begin{align*}
 b_S = \{S \bigcup S', \{l,l'\}_{l \notin S} \} \mbox{ and } d_I = \{I, I^c,  \{l,l'\}_{l \notin S} \}.
\end{align*}
Thus, $b_S \in \Pi_k(r)$. Also, it is easy to see that $$d_I = d_{I^c}, \; d_{S \bigcup S'} = d_{\emptyset} = b_S,  \;  \mbox{ and }  d_{\{l,l'\}} = d_{\{l,l'\}^c} = b_{S\setminus\{l\}}.$$ 

\begin{example}
For $k=6$, $S=\{1,2,4\}$, and $I=\{1,4'\}\subset S\cup S'$, $b_S$ and $d_I$ are:

  
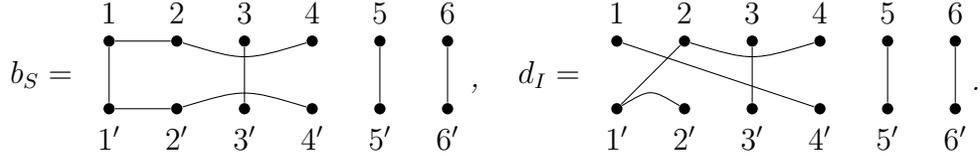
\begin{tikzpicture}[scale=0.9,
    mycirc/.style={circle,fill=black, minimum size=0.1mm, inner sep = 1.5pt}
    ]
    \node (1) at (-1,0.5) {$b_S =$};
    \node[mycirc,label=above:{$1$}] (n1) at (0,1) {};
    \node[mycirc,label=above:{$2$}] (n2) at (1,1) {};
    \node[mycirc,label=above:{$3$}] (n3) at (2,1) {};
    \node[mycirc,label=above:{$4$}] (n4) at (3,1) {};
    \node[mycirc,label=above:{$5$}] (n5) at (4,1) {};
    \node[mycirc,label=above:{$6$}] (n6) at (5,1) {};
    \node[mycirc,label=below:{$1'$}] (n1') at (0,0) {};
    \node[mycirc,label=below:{$2'$}] (n2') at (1,0) {};
    \node[mycirc,label=below:{$3'$}] (n3') at (2,0) {};
    \node[mycirc,label=below:{$4'$}] (n4') at (3,0) {};
    \node[mycirc,label=below:{$5'$}] (n5') at (4,0) {};
    \node[mycirc,label=below:{$6'$}] (n6') at (5,0) {};
    \draw (n1) -- (n2);
    \draw (n1) -- (n1');
    \draw (n1')-- (n2');
    \draw (n2) .. controls(2,0.7) .. (n4);
    \draw (n2').. controls(2,0.3) .. (n4');
    \draw (n3) -- (n3');
    \draw (n5) -- (n5');
    \draw (n6) -- (n6');
    \node(a) at (5.4,0.3) {,};

    \node (1) at (6.5,0.5) {$d_I =$};
    \node[mycirc,label=above:{$1$}] (n7) at (7.5,1) {};
    \node[mycirc,label=above:{$2$}] (n8) at (8.5,1) {};
    \node[mycirc,label=above:{$3$}] (n9) at (9.5,1) {};
    \node[mycirc,label=above:{$4$}] (n10) at (10.5,1) {};
    \node[mycirc,label=above:{$5$}] (n11) at (11.5,1) {};
    \node[mycirc,label=above:{$6$}] (n12) at (12.5,1) {};
    \node[mycirc,label=below:{$1'$}] (n7') at (7.5,0) {};
    \node[mycirc,label=below:{$2'$}] (n8') at (8.5,0) {};
    \node[mycirc,label=below:{$3'$}] (n9') at (9.5,0) {};
    \node[mycirc,label=below:{$4'$}] (n10') at (10.5,0) {};
    \node[mycirc,label=below:{$5'$}] (n11') at (11.5,0) {};
    \node[mycirc,label=below:{$6'$}] (n12') at (12.5,0) {};
    \draw (n7) -- (n10');
    \draw (n8) -- (n7');
    \draw (n8) .. controls(9.5,0.7) .. (n10);
    \draw (n7') .. controls(8,0.3) .. (n8');
    \draw (n9) -- (n9');
    \draw (n11) -- (n11');
    \draw (n12) -- (n12');
    \node(a) at (12.8,0.3) {.};   
    
\end{tikzpicture}
\end{example}

Following the notation of Section\ref{ta}, let $N(I)$ and $M(I)$ denote the number of elements in top row and bottom row of the block $I$ respectively. 
For $k \in \Z_{\geq 0}$, we define an element $Z_{k,r} \in \T_k(r,1,n) \subseteq \Tn$:
\begin{equation*}
 Z_{k,r} = \binom{n}{2}+\sum\limits_{|S| \geq 1}(-1)^{|S|}\bigg((n-1)b_S + \sum\limits_{N(I) \equiv M(I) (mod~r)}(-1)^{N(I)-M(I)}(d_I - b_S)\bigg),
\end{equation*}
where the outer sum is over all the nonempty subsets $S$ of $\{1,2,\ldots,k\}$ and the inner sum is over $I \subseteq S \bigcup S'$ such that $d_I \in \Pi_k(r)$ and $d_I \neq b_S$. 

Define an element $Z_{k+\frac{1}{2},r} \in \T_{k+\frac{1}{2}}(r,1,n) \subseteq \Thn$ as follows:
\begin{align*}
 Z_{k+\frac{1}{2},r}  & = \binom{n}{2}  + \sum\limits_{\substack{|S| \geq 1\\ k+1 \notin S}}(-1)^{|S|}\bigg((n-1)b_S + \sum\limits_{N(I) \equiv M(I) (mod~r)}(-1)^{N(I)-M(I)}(d_I - b_S)\bigg) \\ 
&+ \sum\limits_{\substack{ k+1 \in S \\|S| \geq 1}}(-1)^{|S|}\bigg((n-1)b_S + \sum\limits_{\substack{\{k+1, (k+1)'\} \subset I \mbox{\tiny{or}} I^c \\ N(I) \equiv M(I) (mod~r)}}(-1)^{N(I)-M(I)}(d_I - b_S)\bigg).
\end{align*}
where the first outer sum is over all the nonempty subsets $S$ of $\{1,2,\ldots,k+1\}$ such that $k+1 \notin S$ and the inner sum in that is over $I \subseteq S \bigcup S'$ such that $d_I \in \Pi_{k+\frac{1}{2}}(r)$ and $d_I \neq b_S$; 
and the second outer sum is over all the nonempty subsets $S$ of $\{1,2,\ldots,k, k+1\}$ such that $k+1 \in S$ and the inner sum in that is over $I \subseteq S \bigcup S'$ such that 
\begin{center}
$\{k+1, (k+1)'\} \subseteq I \mbox{ or } I^c$, $d_I \in \Pi_{k+\frac{1}{2}}(r)$ and $d_I \neq b_S$.
\end{center}

The elements $Z_{k,r}$ and $Z_{k+\frac{1}{2},r}$ and the idea of the proof of the next theorem are from the online notes \cite{Ram2010M}.

\begin{thm}\phantomsection\label{zkr}
\begin{enumerate}[(a)]

\item Let $k \in \Z_{\geq 0}$. Then,
\begin{equation*}
 \ka_{r,n} = Z_{k,r} \mbox{ and } \ka_{r,n-1} = Z_{k+\frac{1}{2}, r} 
\end{equation*}
as operators on $\Vk$ and $\Vkh$ respectively.

\item Let $k \in \Z_{\geq 0}$. Then $Z_{k,r}$ is a central element of $\Tn$. For $n \geq 2k$ $$Z_{k,r} = \sum\limits_{b \in \lambda}c(b) $$ as operators on $\T_k^{(\tla, \delta)}$, the irreducible $\Tn$-module parametrized by 
$(\tla, \delta) \in \vo_k(r,p,n)$. 

Also, $Z_{k+\frac{1}{2}, r}$ is a central element of $\Thn$. For $n \geq 2k+1$, $$Z_{k+\frac{1}{2},r} = \sum\limits_{b \in \mu}c(b) $$ as operators on $\T_{k+\frac{1}{2}}^{\tm^{(n,i,j)}}$, the irreducible $\Thn$-module parametrized by $\tm^{(n,i,j)} \in \vo_{k+\frac{1}{2}}(r,p,n)$.
\end{enumerate}
\end{thm}

\begin{proof}
\begin{enumerate}[(a)]
\item We express the action of $\ka_{r,n}$ in terms of matrices $E_{i,j}$.
\begin{align}\label{kact}
       & 2 \ka_{r,n}(\vi) \nonumber \\ 
     = &\frac{1}{r}\sum\limits_{l=0}^{r-1}\sum\limits_{1 \leq i \neq j \leq n}\g_i^{l}\g_j^{-l}s_{ij} v_{i_1} \otimes \g_i^{l}\g_j^{-l}s_{ij} v_{i_2}
       \otimes \cdots \otimes \g_i^{l}\g_j^{-l}s_{ij} v_{i_k} \nonumber \\
     = &\frac{1}{r}\sum\limits_{l=0}^{r-1}\sum\limits_{1 \leq i \neq j \leq n} 
     (1-E_{ii}-E_{jj}+\g^{l}E_{ij}+\g^{-l}E_{ji})v_{i_1} \otimes \cdots \nonumber \\ 
     & ~~~~~~~~~~~~~~~~\otimes (1-E_{ii}-E_{jj}+\g^{l}E_{ij}+\g^{-l}E_{ji})v_{i_k}.
       \end{align}

Let $S$ be a subset of $\{1,2, \ldots, k\}$ such that $S^c$ corresponds to the tensor positions where $1$ is acting, $I \subset S \bigcup S'$ corresponds to the tensor positions that must equal $i$ and $I^c$ corresponds to the tensor positions that must equal $j$. Let 
\begin{align*}
c_{S,I} := & \prod_{t \in S^c} (\delta_{i_ti_{t'}})(-1)^{\#(\{t,t'\}\subset I) + \#(\{t,t'\}\subset I^c)} \\ 
& \times \zeta^{l(\#(\{t \in I,t' \in I^c\}) - \#(\{t \in I^c ,t' \in I\}))}\prod_{t \in I}(\delta_{i_ti})\prod_{t \in I^c}(\delta_{i_tj}) 
\end{align*}
Thus, expanding (\ref{kact}), we get that $2 \ka_{r,n}(\vi)$ equals  
\begin{equation}\label{exp}
\frac{1}{r}\sum\limits_{S \subset \{1,\ldots, k\}} \sum\limits_{\itp}\sum\limits_{l=0}^{r-1}\sum\limits_{1 \leq i \neq j \leq n}\sum\limits_{I \subset S \bigcup S'}c_{S,I}(\vip).
\end{equation}

Now, we take various cases of $S$ and $I$ to compute the above expression (\ref{exp}). Let $|S| = 0$, then $I$ is empty set and $$c_{S,I} = c_{\emptyset, \emptyset} = \prod\limits_{t \in \{1, \ldots,k \}}\delta_{i_ti_{t'}}.$$ The corresponding summand in (\ref{exp}) is $$(n^2-n)(\vi).$$

Assume that $|S| \geq 1$ and we consider various cases of $I \subset S \bigcup S'$. Since the whole sum is symmetric in $i$ and $j$ and in $I$ and $I^c$, therefore, the sum obtained is same when $I$ is interchanged with $I^c$. If $I = S \bigcup S'$, then $$c_{S,I} =  \prod_{t \in S^c} (\delta_{i_ti_{t'}})(-1)^{|S|}\prod_{t \in S \bigcup S'}(\delta_{i_ti}),$$ and thus the corresponding summand in expression (\ref{exp}) is 
\begin{align*}
& \frac{1}{r} \sum\limits_{\itp}\sum\limits_{l=0}^{r-1}(n-1)\sum\limits_{1 \leq i \leq n}c_{S,I}(\vip) \\ 
= & (n-1)(-1)^{|S|}b_S(\vi).
\end{align*}
We get an identical summand for the case $I = \emptyset$.

Consider $I \subsetneq S \bigcup S'$ and $N(I) \not\equiv M(I) (\mbox{mod }r)$. Let $$T(I) = \{t \in \{1,2,\ldots, k\} \mid t \in I \},$$ $$D(I) = \{t' \in \{1',2',\ldots, k'\} \mid t' \in I \}, $$ $$ \mbox{and }B(I) = \{t \in I \mid t' \in I \}.~~~~~~~~~~$$ Thus, $N(I) = |T(I)|, ~M(I) = |D(I)|$. Also, we can see that $$\#(\{t \in I,t' \in I^c\}) = N(I) - |B(I)|$$ $$ \mbox{and }\#(\{t \in I^c ,t' \in I\}) = M(I) - |B(I)|.$$ Thus,
\begin{align*}
 \#(\{t \in I,t' \in I^c\})-\#(\{t \in I^c ,t' \in I\}) = N(I)-M(I) \not\equiv 0(\mbox{mod }r).
\end{align*}
In this case, since the sum of all the $r$-th roots of unity is zero, so the summand for all such $I$ in expression (\ref{exp}) is zero. 

Now, consider those subsets $I \subsetneq S \bigcup S'$ such that $N(I) \equiv M(I) (\mbox{mod }r)$. Define $B(I)' := \{t'\mid t \in B(I)\}$, thus $$ \{t \in \{1,2,\ldots, k\} \mid \{t,t'\}\subset I\} = B(I), $$ $$\mbox{ and } \{ t \in \{1,2,\ldots, k\} \mid \{t,t'\}\subset I^c\} = (S \setminus T(I)) \setminus (D(I) \setminus B(I)').$$ This implies that  
\begin{align*}
(-1)^{\#(\{t,t'\}\subset I) + \#(\{t,t'\}\subset I^c)} & = (-1)^{|B(I)| + (|S| - N(I)) -(M(I) - |B(I)|)} \\
& = (-1)^{|B(I)| + (|S| - N(I))+(M(I) - |B(I)|)} \\
& = (-1)^{|S| - (N(I)- M(I))} \\
& = (-1)^{|S| + (N(I) - M(I))}.
\end{align*}
Thus, for the subsets $I$ such that $N(I) \equiv M(I) (\mbox{mod }r)$, we get 
the summand in  expression (\ref{exp}) as:
\begin{align*}
 &\frac{1}{r}(-1)^{\#(\{t,t'\}\subset I) + \#(\{t,t'\}\subset I^c)} \sum\limits_{\itp}\sum\limits_{l=0}^{r-1}\prod_{t \in S^c} (\delta_{i_ti_{t'}})\\ 
 & \times \sum\limits_{1 \leq i \neq j \leq n} \prod_{t \in I} (\delta_{i_ti})\prod_{t \in I^c}(\delta_{i_tj})(\vip)  \\
 & =  (-1)^{|S| + (N(I) - M(I))}\sum\limits_{\itp}\prod_{t \in S^c} (\delta_{i_ti_{t'}}) \left(\sum\limits_{1 \leq i, j \leq n} \prod_{t \in I} (\delta_{i_ti})\prod_{t \in I^c}(\delta_{i_tj})\right.\\ 
 &\left. -\sum\limits_{1 \leq i = j \leq n} \prod_{t \in I} (\delta_{i_ti})\prod_{t \in I^c}(\delta_{i_tj})\right)(\vip)\\
 & = (-1)^{|S| + (N(I) - M(I))}(d_I-{b_S})(\vi).
\end{align*}
Also, for the subsets $I$ such that $N(I) \equiv M(I) (\mbox{mod }r)$, we also have $N(I^c) \equiv M(I^c) (\mbox{mod }r)$ and thus we get an identical summand by interchanging $I$ and $I^c$.

Combining all the above cases together, we get that, as operators on $\Vk$, 
\begin{equation*}
 \ka_{r,n} = Z_{k,r}.
\end{equation*}

Now we prove the second part of (b). We have 
\begin{equation*}
 (1-E_{ii}-E_{jj}+E_{ii}E_{jj})(v_n) = 
 \begin{cases}
   0, & \mbox{if } i = n \mbox{ or } j=n, \\
   v_n, & \mbox{otherwise.}
  \end{cases}
\end{equation*}
Thus, 
\begin{align}\label{kact1}
       & 2 \ka_{r,n-1}(\vi \otimes v_n) \nonumber \\ 
     = & \frac{1}{r}\sum\limits_{l=0}^{r-1}\sum\limits_{1 \leq i \neq j \leq n-1}\g_i^{l}\g_j^{-l}s_{ij}(\vi \otimes v_n) \nonumber \\
     = &\frac{1}{r}\sum\limits_{l=0}^{r-1}\sum\limits_{1 \leq i \neq j \leq n}\g_i^{l}\g_j^{-l}s_{ij} v_{i_1} \otimes \g_i^{l}\g_j^{-l}s_{ij} v_{i_2}
       \otimes \nonumber \\
       & \cdots \otimes \g_i^{l}\g_j^{-l}s_{ij} v_{i_k} \otimes (1-E_{ii}-E_{jj}+E_{ii}E_{jj})(v_n) \nonumber \\
     = & \frac{1}{r}\sum\limits_{l=0}^{r-1}\sum\limits_{1 \leq i \neq j \leq n}\g_i^{l}\g_j^{-l}s_{ij}(\vi)\otimes v_n \nonumber \\
     + &  \frac{1}{r}\sum\limits_{l=0}^{r-1}\sum\limits_{1 \leq i \neq j \leq n} 
     \bigg((1-E_{ii}-E_{jj}+\g^{l}E_{ij}+\g^{-l}E_{ji})v_{i_1} \otimes \cdots  \nonumber \\ 
     &  \otimes (1-E_{ii}-E_{jj}+\g^{l}E_{ij}+\g^{-l}E_{ji})v_{i_k}\bigg) \otimes (-E_{ii}-E_{jj})v_n \nonumber \\
     + & \frac{1}{r}\sum\limits_{l=0}^{r-1}\sum\limits_{1 \leq i \neq j \leq n}\g_i^{l}\g_j^{-l}s_{ij} v_{i_1} \otimes \g_i^{l}\g_j^{-l}s_{ij} v_{i_2}
       \otimes \cdots \otimes \g_i^{l}\g_j^{-l}s_{ij} v_{i_k} \otimes +E_{ii}E_{jj}v_n.
     \end{align}

In the expression (\ref{kact1}), the first summand is equal to $2\ka_{r,n}(\vi)$ which has been calculated in the first part of (b). Since $i \neq j$, so the last summand is zero. Expanding the middle summand gives 

\begin{equation*}
\frac{1}{r}\sum\limits_{\substack{S \subset \{1,\ldots, k+1\} \\ k+1 \in S} } \sum\limits_{\itp}\sum\limits_{l=0}^{r-1}\sum\limits_{1 \leq i \neq j \leq n}\sum\limits_{\substack{I \subset S \bigcup S' \\ \{k+1, (k+1)'\} \subset I \mbox{\tiny{or}} I^c}}c_{S,I}(\vip).
\end{equation*}

The case $|S| = 0$ does not arise because $k+1 \in S$. For $|S| >1$, we consider various cases of $I \subset S \bigcup S'$ which are: 
\begin{align*}& (i) I = S \bigcup S', \\ 
& (ii)\{k+1,(k+1)'\} \subset I \subsetneq S \bigcup S',~N(I) \not\equiv M(I) (\mbox{mod }r), \mbox{and } \\
& (iii) \{k+1,(k+1)'\} \subset I \subsetneq S \bigcup S',~N(I) \equiv M(I) (\mbox{mod }r).
\end{align*}
and identical summands arise when $I$ is interchanged with $I^c$ in the cases $(i), (ii), $ and $(iii)$. Thus, the middle summand gives us 
\begin{equation*}
 \sum\limits_{\substack{ k+1 \in S \\|S| \geq 1}}(-1)^{|S|}2\bigg((n-1)b_S + \sum\limits_{\substack{\{k+1, (k+1)'\} \subset I \mbox{\tiny{or}} I^c \\ N(I) \equiv M(I) (mod~r)}}(-1)^{N(I)-M(I)}(d_I - b_S)\bigg).
\end{equation*}

So, as opeartors on $\Vkh$, we have $\ka_{r,n-1} = \Z_{k+\frac{1}{2},r}.$

\item  Using Theorem \ref{dvkp}(a) and using Lemma \ref{ka}(a), we get that for $n \geq 2k$, $Z_{k,r}$ acts on $\T_k^{(\tla, \delta)}$ as the constant stated in the theorem. Therefore, $Z_{k,r}$ is a central element of $\Tn$ for $n\geq 2k$. Since the multiplication of elements of $\T_k(r,p,n)$ is a polynomial in $n$, therefore $$Z_{k,r} x_d = x_d Z_{k,r} \mbox{ for all } x_d \in \T_k(r,p,n) \mbox{ and } n \geq 2k  \; \; \; \Longrightarrow  \; \; \; Z_{k,r} x_d = x_d Z_{k,r}$$ for all $x_d \in \T_k(r,1,n)$ and for all $n$. 

Theorem \ref{dvkp}(b) and Lemma \ref{ka}(b) along with the arguments similar to the above imply the result for $Z_{k+\frac{1}{2},r}.$
\qedhere
\end{enumerate}
\end{proof}

In the light of Remarks \ref{222k} and \ref{cp}, $(r,p,n) = (2,2,2k)$ is the special case. 
Define $M_{k,2,2} := x_d \in \T_k(2,2,2k)$, where $d$ is the only element in $\Lambda_k(2,2,2k)$ and $d$ consists of $2k$ blocks, each vertex being a block. The element $M_{k,2,2}$ is a central element of $\T_k(2,2,2k)$.

\begin{thm}\phantomsection\label{mk22}
\begin{enumerate}[(a)]
\item Let $k \in \Z_{\geq 0}$. Then,
$M_{k,2,2} = \frac{1}{2^k}z$
as operators on $\Vk$.

\item Let $k \in \Z_{\geq 0}$. Then, for $\La \neq ((k),(k))$, $M_{k,2,2} = 0$ as operators on the irreducible $\T_k(2,2,2k)$-module $\T_k^{(\tla, 1)}$. 
For $\La = ((k),(k))$, $$M_{k,2,2} = k!$$ as operators on the irreducible $\T_k(2,2,2k)$-module $\T_k^{(\tla, 1)}$ and $$M_{k,2,2} = -k!$$ as operators on the irreducible $\T_k(2,2,2k)$-module $\T_k^{(\tla, -1)}$.
\end{enumerate}
\end{thm}

\begin{proof}
 \begin{enumerate}[(a)]
\item The action of $M_{k,2,2}$ on $\Vk$ is:
\begin{align*}
  M_{k,2,2} (v_{i_1} \otimes \cdots \otimes v_{i_k}) = 
  \begin{cases}
   \sum\limits_{\pi} v_{\pi(j_1)} \otimes \cdots \otimes v_{\pi(j_k)}, & \mbox{if } i_1, i_2, \ldots, i_k \mbox{ are distinct } \\
   & \mbox{elements of } \{1, \ldots, 2k\}, \\
   0, & \mbox{otherwise,}
  \end{cases}
\end{align*}
where $\pi$ varies over all the permutations of $ \{j_1, \ldots, j_k\} = \{1,\ldots, 2k\} \setminus \{i_1, \ldots, i_k\}$.

Now, we discuss the action of $z$ on $\Vk$.
Consider $v_{i_1} \otimes \cdots \otimes v_{i_k} \in \Vk$ such that $i_1, \ldots, i_k$ are distinct elements of $\{1, \ldots, 2k\}$. Then, 
$$(a_1, \ldots, a_{2k}, (i_1, j_1)  \cdots (i_{k},j_{k})) (v_{i_1} \otimes \cdots \otimes v_{i_k}) =  -(v_{j_1} \otimes \cdots \otimes v_{j_k}), $$
if $(a_1, \ldots, a_{2k}, (i_1, j_1)  \cdots (i_{k},j_{k})) \in C_1 $
and 
$$(a_1, \ldots, a_{2k}, (i_1, j_1)  \cdots (i_{k},j_{k})) (v_{i_1} \otimes \cdots \otimes v_{i_k}) = (v_{j_1} \otimes \cdots \otimes v_{j_k}), $$
if $(a_1, \ldots, a_{2k}, (i_1, j_1)  \cdots (i_{k},j_{k})) \in C_2$,
where in each case $$ \{j_1, \ldots, j_k\} = \{1,\ldots, 2k\} \setminus \{i_1, \ldots, i_k\}.$$
For a fixed $(i_1, j_1)  \cdots (i_{k},j_{k})$ element in $S_{2k}$, there are $2^{k-1}$ elements of the form $(a_1, \ldots, a_{2k},(i_1, j_1)  \cdots (i_{k},j_{k}))$ in each of $C_1$ and $C_2$. 

Consider an element of the form $(a_1, \ldots, a_{2k}, (x_1, y_1)  \cdots (x_{k},y_{k})) \in C$ such that at least one pair, say $\{x_1, y_1\} \subset \{i_1, \ldots, i_k\}.$ Then, one of $x_2, \ldots, x_k$, say $x_k$, is different from $i_1, \ldots, i_k$ and one can choose $y_k \in \{1, \ldots, 2k \} \setminus \{i_1, \ldots, i_k, x_k, y_2, \ldots, y_{k-1}\}.$ Now, $(a_{x_k}, a_{y_k}) = (1,1)$ or $(a_{x_k}, a_{y_k}) = (-1,-1)$ keeps the sign of the action of $(a_1, \ldots, a_{2k}, (x_1, y_1)  \cdots (x_{k},y_{k}))$ on $(v_{i_1} \otimes \cdots \otimes v_{i_k})$ same. 

Given $(b_1, \ldots, b_{2k}, (x_1, y_1)  \cdots (x_{k},y_{k})) \in C_1$ such that  $(b_{x_1}, b_{y_1}) = (1,1)$, we have the element $(f_1, \ldots, f_{2k}, (x_1, y_1)  \cdots (x_{k},y_{k})) \in C_2$, such that $(f_{x_i}, f_{y_i}) = (b_{x_i}, b_{y_i})$ for $i \neq k$ and $(f_{x_k}, f_{y_k}) = - (b_{x_k}, b_{y_k})$ and 
\begin{align*}
 & (b_1, \ldots, b_{2k}, (x_1, y_1)  \cdots (x_{k},y_{k})) (v_{i_1} \otimes \cdots \otimes v_{i_k}) \\
 & = (f_1, \ldots, f_{2k}, (x_1, y_1)  \cdots (x_{k},y_{k})) (v_{i_1} \otimes \cdots \otimes v_{i_k}).
\end{align*}
A similar analysis can be done if $(b_{x_1}, b_{y_1}) = (-1,-1)$.

If at least two of $i_1, \ldots, i_k$ are same, say $i_1 = i_2$, then we can find a pair $(a_{x_k}, a_{y_k})$ such that the action of $(a_1, \ldots, a_{2k}, (x_1, y_1)  \cdots (x_{k},y_{k}))$ on $(v_{i_1} \otimes \cdots \otimes v_{i_k})$ has the same sign whether $(a_{x_k}, a_{y_k}) = (1,1)$ or $(-1,-1)$. A similar analysis as above shows that corresponding to any element $(b_1, \ldots, b_{2k}, (x_1, y_1)  \cdots (x_{k},y_{k})) \in C_1$ such that  we can find the element $(f_1, \ldots, f_{2k}, (x_1, y_1)  \cdots (x_{k},y_{k})) \in C_2$, such that
\begin{align*}
 & (b_1, \ldots, b_{2k}, (x_1, y_1)  \cdots (x_{k},y_{k})) (v_{i_1} \otimes \cdots \otimes v_{i_k}) \\
 & = (f_1, \ldots, f_{2k}, (x_1, y_1)  \cdots (x_{k},y_{k})) (v_{i_1} \otimes \cdots \otimes v_{i_k}).
\end{align*}
Collecting all the cases, we have 
\begin{align*}
  z (v_{i_1} \otimes \cdots \otimes v_{i_k}) = 
  \begin{cases}
   (2^k)\sum\limits_{\pi} v_{\pi(j_1)} \otimes \cdots \otimes v_{\pi(j_k)}, & \mbox{if } i_1, \ldots, i_k \mbox{ are distinct } \\
   & \mbox{elements of } \{1, \ldots, 2k\}, \\
   0, & \mbox{otherwise,}
  \end{cases}
\end{align*}
where $\pi$ varies over all the permutations of $ \{j_1, \ldots, j_k\} = \{1,\ldots, 2k\} \setminus \{i_1, \ldots, i_k\}$.
\item The proof is clear by using part (a) of this theorem, Theorem \ref{dvkp}(a) and Lemma \ref{actz}.
\end{enumerate}
\end{proof}

For $ l \in \frac{1}{2}\Z_{> 0}$, define the Jucys-Murphy elements of $\T_l(r,p,n)$ as follows:
\begin{align*}
M_{\frac{1}{2},r} & = 1, \\
\mbox{and }M_{l,r} & = Z_{l,r} - Z_{l-\frac{1}{2},r}.
\end{align*}
In addition to these elements, $\T_k(2,2,2k)$ has one more Jucys-Murphy element which is $M_{k,2,2}$.

\begin{thm}\label{jm}
 Let $l \in \frac{1}{2}\Z_{\geq 0}$ and let $n$ be a positive integer.
 
 \begin{enumerate}[(a)]
  \item The elements $M_{\frac{1}{2},r}, M_{1,r}, \ldots, M_{l-\frac{1}{2},r},M_{l,r}$ commute with each other in $\T_l(r,p,n).$
  
  \item Assume that $n \geq 2l$. Let $\mathpzc{v}_l \in \vo_l(r,p,n)$ and $\T_{l}^{\pv_l}$ be the irreducible $\T_l(r,p,n)$-module parametrized by $\pv_l$. Then there is a unique, up to scalars, basis $$\{u_\mathpzc{P} \mid \PP \mbox{ is a path in } \widehat{\T}(r,p,n) \mbox{ from } \pv_0 = ((n), \emptyset, \ldots, \emptyset) \mbox{ to } \pv_l\}$$ of $\T_{l}^{\pv_l}$ such that, for all $\PP = (\pv_0, \pv_{\frac{1}{2}}, \pv_1, \ldots, \pv_l)$, and for all $k \in \Z_{\geq 0}$, $k \leq l$
  \begin{equation*}
   M_{k,r}(u_\PP) = c(\pv_k / \pv_{k-\frac{1}{2}})u_\PP,
  \end{equation*}
and   
\begin{equation*}
M_{k+\frac{1}{2},r}(u_\PP)  = - c(\pv_k / \pv_{k+\frac{1}{2}})u_\PP,
\end{equation*}
where $\pv_k / \pv_{k-\frac{1}{2}} $ and $\pv_k / \pv_{k+\frac{1}{2}}$ denote the box by which $\pv_k$ differs from $\pv_{k-\frac{1}{2}}$ and $\pv_{k+\frac{1}{2}}$ as $r$-tuple of Young diagrams respectively.
\item For $(r,p,n) = (2,2,2k)$, the element $M_{k,2,2}$ commutes with the elements $M_{\frac{1}{2},2},$ $M_{1,2},$ $\ldots, M_{2k-\frac{1}{2},2},M_{2k,2}$. The scalars by which the Jucys-Murphy elements of $\T_k(2,2,2k)$ act on the basis (as given by part (b)) of $\T_{k}^{(((k),(k)),1)}$ and $\T_{k}^{(((k),(k)),-1)}$ are same except for $M_{k,2,2}$.
\end{enumerate}
\end{thm}

\begin{proof}
 \begin{enumerate}[(a)]
  \item For $i, j \in \frac{1}{2}\Z_{\geq 0}$ and $i \leq j \leq l$, we have $Z_{i,r}, Z_{j,r} \in \T_j(r,p,n)$ and $Z_{j,r}$ is a central element of $\T_j(r,p,n) \subseteq \T_l(r,p,n)$, thus $Z_{i,r} Z_{j,r} = Z_{j,r}Z_{i,r}.$ Since $M_{j,r} = Z_{j,r} - Z_{j-\frac{1}{2},r}$, thus Jucys-Murphy elements commute with each other in $\T_l(r,p,n)$.
  \item  The branching rule from $\T_j(r,p,n)$ to $\T_{j-\frac{1}{2}}(r,p,n)$ is multiplicity free for all $j \in \frac{1}{2}\Z_{\geq 0}$ and $n \geq 2j$. Thus, $\T_{l}^{\pv_l}$ has canonical decomposition as irreducible $\T_{l-\frac{1}{2}}$-module:
  $$ \T_{l}^{\pv_l} = \bigoplus\limits_{\pv_{l-\frac{1}{2}} \in \vo_{l-\frac{1}{2}}(r,p,n)} \T^{\pv_{l-\frac{1}{2}}}_{l-\frac{1}{2}}$$
  such that there is an edge from $\pv_{l-\frac{1}{2}}$ to $\pv_l$ in $\widehat{\T}(r,p,n)$. Further, iterating this decomposition, a canonical decomposition of $\T_{l}^{\pv_l}$ into irreducible $\T_{0}(r,p,n)$-modules is obtained:
  \begin{equation}\label{cdec}
  \T_{l}^{\pv_l} = \bigoplus\limits_{\PP} \T_{\PP},
  \end{equation}
  where $\T_{\PP}$ are one-dimensional $\T_{0}(r,p,n)$-modules and the sum is over all paths $\PP = (\pv_0, \pv_{\frac{1}{2}}, \pv_1, \ldots, \pv_l) $ such that $\pv_j \in \vo_j(r,p,n).$ The basis of $\T_{l}^{\pv_l}$ is obtained by choosing a nonzero vector $u_\PP$ in each $\T_{\PP}$ in the decomposition (\ref{cdec}). Such a basis is called the {\em Gelfand-Tsetlin basis} of the corresponding irreducible representation and it is unique, up to scalars. Using the decomposition (\ref{cdec}) and the definition of $u_{\PP}$, we get
  $$\T_j(r,p,n) u_{\PP} = \T_{j}^{\pv_j}, $$ for all $j \in \frac{1}{2}\Z_{\geq 0}$ and $j \leq l$ ,which implies that $u_{\PP}$ is a basis element of $\T_{j}^{\pv_j}$. Thus, for all $j \in \frac{1}{2}\Z_{\geq 0}$ and $j \leq l$, the action of $Z_{j,r}$ on $u_{\PP}$ is as a scalar given in Theorem \ref{zkr}(b). 
  Now, by the definition of Jucys-Murphy elements, we get their actions on $u_{\PP}$. 
  
 \item The element $M_{k,2,2}$ is a central element of $\T_k(2,2,2k)$. For $(r,p,n) = (2,2,2k)$ and $\La = ((k),(k))$, let $\pv_k = (\tla, 1) \in \vo_k(2,2,2k)$, $\pv'_k = (\tla, -1) \in \vo_k(2,2,2k)$. Then $$\T_k^{\pv_k} \cong  \T_k^{\pv'_k}$$ as $\T_{k-\frac{1}{2}}(r,p,n)$-modules. Thus, the part of the paths from $\pv_0$ to $\pv_k$ and $\pv'_k$ are same for $l <k $, $l \in \frac{1}{2}\Z_{\geq 0}$ and so, we have $$M_j(u_\PP) = M_j (u_{\PP'}), \; \; \; j \leq k  \mbox{ and } j \in \frac{1}{2}\Z_{\geq 0}$$ where $u_\PP$ and $u_{\PP'}$ are Gelfand-Tsetlin basis elements of $\T_k^{\pv_k}$ and $\T_k^{\pv'_k}$ respectively. However, by Theorem \ref{mk22}(b), we get that $$M_{k,2,2}u_\PP =  (k!)u_\PP \mbox{ and } M_{k,2,2}u_{\PP'} = -(k!) u_{\PP'},$$
 which proves the result. \qedhere
  \end{enumerate}
\end{proof}

\section*{Acknowledgements}
The authors thank Arun Ram for suggesting the problem, for valuable insights and for his online notes \cite{Ram}. The second author gratefully acknowledges the workshop ``Representation Theory of Symmetric Groups and Related Algebras'' at Institute of Mathematical Sciences, NUS, Singapore, where she had the opportunity to have discussions with Arun Ram. 
The first author has been supported by UFPA-CAPES/PNPD fellowship and visiting professorship at UFPA, Brazil. The second author is supported by post-doctoral fellowship NPDF under DST-SERB, India.

\bibliographystyle{alpha}
\bibliography{rtta}
 
\Addresses

\end{document}